\documentclass{article}

\PassOptionsToPackage{square, numbers}{natbib}
\usepackage[final]{neurips_2022}

\usepackage[utf8]{inputenc} 
\usepackage[T1]{fontenc}    
\usepackage{hyperref}       
\usepackage{url}            
\usepackage{booktabs}       
\usepackage{amsfonts}       
\usepackage{nicefrac}       
\usepackage{microtype}      
\usepackage{xcolor}         
\usepackage{soul}

\usepackage{caption}
\usepackage{subcaption}

\usepackage{verbatim}
\usepackage{amsmath, amssymb}
\usepackage{xcolor}

\usepackage{amsthm}
\usepackage{graphicx}

\usepackage{apptools}

\usepackage[ruled,vlined]{algorithm2e}

\newcommand{\vx}{\mathbf{x}}
\newcommand{\vy}{\mathbf{y}}
\newcommand{\vz}{\mathbf{z}}
\newcommand{\vu}{\mathbf{u}}
\newcommand{\vv}{\mathbf{v}}
\newcommand{\vp}{\mathbf{p}}

\newcommand{\vr}{\mathbf{r}}
\newcommand{\vs}{\mathbf{s}}
\newcommand{\vw}{\mathbf{\omega}}
\newcommand{\lO}{\Tilde{O}}
\newcommand{\vA}{\mathbf{A}}
\newcommand{\vB}{\mathbf{B}}
\newcommand{\vC}{\mathbf{C}}
\newcommand{\vO}{\mathbf{O}}
\newcommand{\vM}{\mathbf{M}}
\newcommand{\vJ}{\mathbf{J}}
\newcommand{\vH}{\mathbf{H}}

\newcommand{\vI}{\mathbf{I}}
\newcommand{\norm}[1]{\left\lvert#1\right\rvert}
\newcommand{\dotprod}[2]{\left \langle #1, #2 \right \rangle}
\newcommand{\grad}{\nabla}
\newcommand{\sX}{\mathcal{X}}
\newcommand{\sY}{\mathcal{Y}}
\newcommand{\sZ}{\mathcal{Z}}

\newcommand{\bR}{\mathbb{R}}
\newcommand{\revprod}{\widetilde{\prod}}
\newcommand{\E}[1]{\mathbb{E}[#1]}
\newcommand{\vm}{\mathbf{m}}
\newcommand{\bE}{\mathbb{E}}
\newcommand{\ttau}{\Tilde{\tau}}

\usepackage{todonotes}

  {
      \theoremstyle{plain}
      \newtheorem{assumption}{Assumption}
      \newtheorem{theorem}{Theorem}
      \newtheorem{lemma}{Lemma}

        \newtheorem*{assumption*}{Assumption}
      \newtheorem{theorem*}{Theorem}
      \newtheorem*{lemma*}{Lemma}
       \newtheorem*{corollary*}{Corollary}
       \newtheorem*{remark*}{Remark}
  }
\AtAppendix{\counterwithin{lemma}{section}}
\AtAppendix{\counterwithin{theorem}{section}}
\AtAppendix{\counterwithin{corollary}{section}}
\AtAppendix{\counterwithin{equation}{section}}

\title{Sampling without Replacement Leads to Faster Rates in Finite-Sum Minimax Optimization}

\author{%
  Aniket~Das  \thanks{Currently at Google Research. Contact: ketd@google.com} \\
  Indian Institute of Technology Kanpur\\
  \texttt{aniketd@iitk.ac.in} \\
  \And
  Bernhard Sch\"{o}lkopf \\
  Max Planck Institute for Intelligent Systems \\
  \texttt{bs@tuebingen.mpg.de} \\
  \AND
  Michael Muehlebach \\
  Max Planck Institute for Intelligent Systems \\
  \texttt{michaelm@tuebingen.mpg.de} \\
}

\begin{document}

\maketitle

\begin{abstract}
We analyze the convergence rates of stochastic gradient algorithms for smooth finite-sum minimax optimization and show that, for many such algorithms, sampling the data points \emph{without replacement} leads to faster convergence compared to sampling with replacement. For the smooth and strongly convex-strongly concave setting, we consider gradient descent ascent and the proximal point method, and present a unified analysis of two popular without-replacement sampling strategies, namely \emph{Random Reshuffling} (RR), which shuffles the data every epoch, and \emph{Single Shuffling} or \emph{Shuffle Once} (SO), which shuffles only at the beginning. We obtain tight convergence rates for RR and SO and demonstrate that these strategies lead to faster convergence than uniform sampling.
Moving beyond convexity, we obtain similar results for smooth nonconvex-nonconcave objectives satisfying a two-sided Polyak-\L{}ojasiewicz inequality. Finally, we demonstrate that our techniques are general enough to analyze the effect of \emph{data-ordering attacks}, where an adversary manipulates the order in which data points are supplied to the optimizer. Our analysis also recovers tight rates for the \emph{incremental gradient} method, where the data points are not shuffled at all.
\end{abstract}

\section{Introduction}
\label{sec:intro}
The approximate solution of large-scale optimization problems using first-order stochastic gradient methods constitutes one of the foundations of classical machine learning. However, emerging problems in machine learning go beyond pattern recognition and involve real-world decision making, where learning algorithms interact with unknown or even adversarial environments or are deployed in multi-agent settings. Decision making in such environments often involves solving a \emph{minimax optimization} problem of the form $\min_{\vx} \max_{\vy} F(\vx,\vy)$, whose analysis has been a focus of research in mathematics, economics, and theoretical computer science \citep{VonNeumannMorgenstern,FacchineiPang,CostisComplexityMinimax}. Recent examples of its applications in machine learning include adversarial learning \citep{Madry_Adv_Robust, Sinha_Adv_Robust, AdversarialExampleGames}, reinforcement learning \citep{Littman94markovgames, MARLSurvey, SimonDuMinimaxRL, GameMBRL}, imitation learning \citep{FazelLQR, CaiLQR, ErmonGAIL}, and generative adversarial networks \citep{GAN_Original, WassersteinGAN}. In most of these applications, the objective $F(\vx, \vy)$ has a finite-sum structure, i.e., $F(\vx, \vy) = 1/n\sum_{i=1}^{n}f_i(\vx, \vy)$ where $n$ denotes the size of the dataset and each component function $f_i$ denotes the objective associated with the $i^{\textrm{th}}$ data point. The resultant problem is known as \emph{finite-sum minimax optimization}:
\begin{equation}
\label{p:minimax-sum}
\min_{\vx \in \bR^{d_{\vx}}} \max_{\vy \in \bR^{d_{\vy}}} \frac{1}{n}\sum_{i=1}^{n} f_i(\vx, \vy).
\end{equation}
When $n$ is large and the $f_i$'s are differentiable (which holds for most applications), approximate solutions to \eqref{p:minimax-sum} are computed using stochastic gradient algorithms. These algorithms typically sample an index $i \in [n]$ at each iteration as per some specified sampling routine, and use the gradients of $f_i$ to compute the next iterate. Among these methods, perhaps the simplest and most commonly used algorithm is Stochastic Gradient Descent Ascent (SGDA), a natural extension of Stochastic Gradient Descent (SGD) for minimax optimization.

Similar to the stochastic optimization literature, analysis of stochastic minimax optimization often assumes that, at every iteration, the index $i$ is sampled \emph{uniformly with replacement}. Analysis of the resulting algorithm closely parallels the analysis of SGD, and relies on the fact that uniform sampling leads to unbiased gradient estimates. Recent works \citep{GowerSGD,GidelSGDAEC} have also extended this paradigm to i.i.d uniform sampling of mini-batches. While uniform sampling assumptions simplify theoretical analysis, practical implementations of these algorithms often deviate from this paradigm, and instead incorporate various heuristics, which are often empirically found to improve runtime. A common and notable heuristic is to replace uniform sampling by procedures that perform multiple passes over the entire dataset, and in each such pass (called an epoch), sample the data points \emph{without replacement}. Thus, each data point is sampled exactly once in every epoch. These procedures are generally implemented using one of the following approaches:

\textbf{Random Reshuffling} (RR): Uniformly sample a random permutation of $[n]$ at the start of every epoch, and process the data points within that epoch as per the order specified by the permutation.

\textbf{Single Shuffling} or \textbf{Shuffle Once} (SO): Uniformly sample a random permutation at the beginning and reuse it across all epochs to order the data points.

\textbf{Incremental Gradient} (IG): Do not permute the data points at all and follow a fixed deterministic data ordering for every epoch.

Sampling without replacement is ubiquitous in both stochastic minimization \citep{BottouRR2009, SafranShamirRR2020, BengioDLOpt} and stochastic minimax optimization \citep{GAN_Original, WassersteinGAN} as it often exhibits faster runtime than uniform sampling. However, these empirical benefits come at the cost of limited theoretical understanding, due to the absence of provably unbiased gradient estimates.

It is well known in the \emph{optimization} literature that SGD with replacement has a tight rate of $O(1/nK)$ for smooth and strongly convex minimization \citep{RakhlinSridharan, JainNagarajSGD}, where $n$ is the number of component functions and $K$ denotes the number of epochs. On the contrary, recent works on SGD without replacement for smooth and strongly convex minimization \citep{AhnRR2020, NagarajRR2020, MischenkoRR2020,NgyuenRR2020} show that both RR and SO achieve a non-asymptotic rate of $\Tilde{\mathcal{O}}(1/nK^2)$, once the number of epochs $K$ is larger than a certain threshold $K_0$ (usually polynomial in the condition number), and thereby converge faster than SGD with replacement. These rates match the lower bound of $\Omega(1/nK^2)$ for RR and SO established in prior works \citep{RajputRR2020, SafranShamirRR2020}, modulo logarithmic factors. For RR, prior works have also established a similar $\Tilde{\mathcal{O}}(1/nK^2)$ rate for nonconvex objectives satisfying the Polyak-\L{}ojasiewicz (P\L{}) inequality \citep{AhnRR2020, MischenkoRR2020}. While the asymptotic behavior of IG has been known to the community for a long time in both smooth and non-smooth settings \citep{BertsekasIG, NedicBertsekas}, non-asymptotic $\Tilde{\mathcal{O}}(1/K^2)$ convergence rates have been established quite recently \citep{MischenkoRR2020, NgyuenRR2020, GurbuzIG2019}, and are complemented by a matching $\Omega(1/K^2)$ lower bound \citep{SafranShamirRR2020}.

\subsection{Contributions}
\label{sec:contri}
Although the empirical benefits of sampling without replacement have been substantiated for minimization, analysis of these methods for \emph{minimax optimization} have received much less attention, despite being widely prevalent in many applications. Our work aims to fill this gap by analyzing these methods for minimax optimization. To this end, our main contributions are as follows:

\textbf{Unified analysis of RR and SO for smooth strongly convex-strongly concave problems:} 
We analyze RR and SO in conjunction with simultaneous \emph{Gradient Descent Ascent} (GDA), calling the resulting algorithms GDA-RR and GDA-SO, respectively. Assuming the components $f_i$ are smooth and $F$ is strongly convex-strongly concave, we present a unified analysis of GDA-RR/SO and establish a convergence rate of $\Tilde{O}(\exp (-K/5\kappa^2) + 1/nK^2)$ for both (where $\kappa$ is the condition number). Comparing with lower bounds, we show that our rates are \emph{nearly tight}, i.e., they differ from the lower bound only by an exponentially decaying term. Moreover, when $K \geq 10 \kappa^2 \log(n^{\nicefrac{1}{2}}K)$, the convergence rate matches the lower bounds for GDA-RR/SO, modulo logarithmic factors, and also converges provably faster than SGDA with replacement. Under the same setting, we obtain similar guarantees for the RR and SO variants of the \emph{Proximal Point Method} (PPM), named PPM-RR and PPM-SO respectively. Our analysis for both GDA-RR/SO and PPM-RR/SO is general enough to cover smooth strongly monotone finite-sum variational inequalities, which covers minimization, minimax optimization, and multiplayer games.

\textbf{RR for smooth two-sided P\L{} objectives:}
We consider a class of nonconvex-nonconcave problems where the objective $F$ satisfies a \emph{two-sided Polyak-\L{}ojasiewicz} inequality. For such problems, we propose an algorithm that combines RR with two-timescale \emph{Alternating Gradient Descent Ascent} (AGDA), which we call AGDA-RR. We show that AGDA-RR has a nearly tight convergence rate of $\Tilde{O}(\exp (-K/365\kappa^3) + 1/nK^2)$ when the gradient variance is uniformly bounded.  When $K \geq 730\kappa^3 \log(n^{1/2}K)$, this rate matches the lower bound (modulo logarithmic factors) and improves on the best known rates of with-replacement algorithms for this class of problems. 

\textbf{Minimax optimization under data ordering attacks:}
Our techniques for analyzing RR/SO generalize to the analysis of finite-sum minimax optimization under \emph{data ordering attacks} \citep{Papernot_DataOrdering_2021}. These attacks target the inherent randomness assumptions of stochastic gradient algorithms, significantly increasing training time and reducing model quality, only by manipulating the order in which the algorithm receives data points, without performing any data contamination. To model these attacks, we propose the \emph{Adversarial Shuffling} (AS) setup, where the data points are shuffled every epoch by a computationally unrestricted adversary. In this setup, we show that GDA and PPM (now called GDA-AS and PPM-AS) have a convergence rate of $\Tilde{O}(\exp (-K/5\kappa^2) + 1/K^2)$ for smooth strongly convex-strongly concave objectives, and AGDA (now called AGDA-AS) has a rate of $\Tilde{O}(\exp (-K/365\kappa^3) + 1/K^2)$ for two-sided P\L{} objectives. We note that, compared to RR and SO, the convergence rate worsens by a factor of $1/n$ for large enough $K$. When $n$ is large (true for most applications), this slowdown significantly impacts convergence and thus, theoretically justifies the empirical observations in prior work \citep{Papernot_DataOrdering_2021}. We also establish that our analysis in the AS regime also applies to the Incremental Gradient (IG) variants of these algorithms (namely GDA-IG, PPM-IG, and AGDA-IG), and use this to show that our obtained rates for GDA-RR and AGDA-RR are nearly tight. 

To the best of our knowledge, our work is the first to: 1) analyze RR, SO, and IG for strongly monotone unconstrained variational inequalities, 2) analyze RR and IG for a class of nonconvex-nonconcave minimax problems, 3) provably demonstrate the advantages of sampling without replacement for both these settings and justify its empirical benefits in a wide variety of problems ranging from minimization, minimax optimization to smooth multiplayer games,  4) analyze sampling without replacement under data-ordering attacks. Furthermore, unlike prior works on sampling without replacement for minimax optimization \citep{Eric_DRO, Eric_DRO_OGDA}, which are restricted to random reshuffling and require the component functions to be convex-concave, Lipschitz, and smooth, our analysis does not impose any restrictions on the components $f_i$ other than smoothness, allowing them to be arbitrary nonconvex-nonconcave functions.

\section{Notation and Preliminaries}
\label{sec:notation}
We work with Euclidean spaces $(\bR^d, \dotprod{.}{.})$ equipped with the standard inner product $\dotprod{\vx_1}{\vx_2}$ and the induced norm $\norm{\vx}$. For any $\vx \in \bR^{d_\vx}$ and $\vy \in \bR^{d_\vy}$, we denote $\vz = (\vx, \vy) \in \bR^d$ where $d = d_\vx+d_\vy$. Moreover, for any $\vz_1 = (\vx_1, \vy_1) \in \bR^d$ and $\vz_2 = (\vx_2, \vy_2) \in \bR^d$, $\dotprod{\vz_1}{\vz_2} = \dotprod{\vx_1}{\vx_2} + \dotprod{\vy_1}{\vy_2}$ and $\norm{\vz_1}^2 = \norm{\vx_1}^2 + \norm{\vy_1}^2$. Whenever $\vz = (\vx, \vy)$ is clear from the context, we write $f(\vx, \vy)$ as $f(\vz)$. We use $\mathbb{S}_n$ to denote the set of all permutations of $[n] = \{1, \ldots, n \}$. For any matrix $\vA$, its operator norm is denoted by $\norm{\vA} = \sup_{\norm{\vx}=1} \norm{\vA \vx}$. We use the $O$ notation to characterize the dependence of our convergence rates on $n$ and $K$, suppressing numerical and problem-specific constants such as $\kappa, \mu, \sigma$, etc. Additionally, we use the $\Tilde{O}$ notation to suppress logarithmic factors of $n$ and $K$.

Our work studies finite-sum minimax optimization \eqref{p:minimax-sum}. Solutions to \eqref{p:minimax-sum} are known as \emph{global minimax points} of $F = 1/n \sum_{i=1}^{n} f_i$, which \emph{we assume to always exist}. We also assume that the components $f_i$ are continuously differentiable, and hence, the same applies to $F$. This allows us to define the \emph{gradient operators} $\omega_i : \bR^d \rightarrow \bR^d$ and $\nu : \bR^d \rightarrow \bR^d$ as follows: 
\begin{align*}
    \omega_i(\vx, \vy) = [\grad_\vx f_i(\vx, \vy), -\grad_\vy f_i(\vx, \vy)], \ \
    \nu(\vx, \vy) = 1/n\sum_{i=1}^{n} \omega_i(\vx, \vy).
\end{align*}
We also impose the following smoothness assumption on the components $f_i$.
\begin{assumption}[Component Smoothness]
\label{as:comp-smooth}
The component functions $f_i$ are $l$-smooth, i.e., each gradient operator $\omega_i$ is $l$-Lipschitz
$$\norm{\omega_i(\vz_2) - \omega_i(\vz_1)} \leq l \norm{\vz_2 - \vz_1}.$$
Consequently, the operator $\nu$ is also $l$-Lipschitz, i.e., $F$ is $l$-smooth.
\end{assumption}
\section{Analysis for Strongly Convex-Strongly Concave Objectives} 
\label{sec:sc-sc-analysis}
In this section, we analyze two very popular without-replacement algorithms for finite-sum minimax optimization, Gradient Descent Ascent (GDA) without replacement and Proximal Point Method (PPM) without replacement. For each of these, we present a unified analysis of the Random Reshuffling (RR) and Shuffle Once (SO) variants (called GDA-RR/SO and PPM-RR/SO respectively). For a fixed $K > 0$, GDA-RR/SO approximately solves \eqref{p:minimax-sum} by iterating over the entire dataset for $K$ epochs, and within each epoch, uses the  operators $\omega_i$ to perform the following iterative update:
\begin{equation}
\label{eqn:gda-wor-update}
\vz^k_i \leftarrow \vz^{k}_{i-1} - \alpha \vw_{\tau_k(i)}(\vz^{k}_{i-1}) \ \forall i \in [n],
\end{equation}
where $\tau_k$ is a uniformly sampled random permutation of $[n]$ and $0 < \alpha < 1/l$ is a constant step-size. GDA-RR resamples $\tau_k$ at the start of every epoch, whereas GDA-SO samples it only once in the beginning. The details of both algorithms are presented in Algorithm \ref{algo:gda-rr-so}. The Proximal Point Method without replacement is a closely related algorithm which, instead of performing gradient descent-style updates within an epoch, solves the following implicit update equation for $\vz^k_i$:
\begin{equation}
\label{eqn:ppm-wor-update}
\vz^k_i = \vz^{k}_{i-1} - \alpha \vw_{\tau_k(i)}(\vz^{k}_{i}) \ \forall i \in [n].
\end{equation}
As before, $\tau_k$ is resampled at every epoch for PPM-RR, and sampled once and fixed for all epochs for PPM-SO. We present the details in Algorithm \ref{algo:ppm-rr-so}. The $l$-smoothness of $\omega_i$ along with the choice of $\alpha < 1/l$ ensures that $\vz^k_i$ is uniquely defined, since it is a fixed point of the contraction mapping $\zeta(\vz) = \vz^{k}_{i-1} - \alpha \vw_{\tau_k(i)}(\vz)$. This method is actually a generalization of the (stochastic) proximal point  method for minimization problems, and is popular for problems where \eqref{eqn:ppm-wor-update} can be solved easily or in closed form. We refer the readers to \citet{RockafellarMinimax, NecoaraSPPM} for a review of this method and its connections to the original proximal point method for minimization.
\begin{figure*}[t]
\begin{minipage}[h]{0.40\textwidth}
\begin{algorithm}[H]
\label{algo:gda-rr-so}
\SetAlgoLined
\SetKwInOut{Input}{Input}
\SetKwInOut{Output}{Output}
\Input{Number of epochs $K$, step-size $\alpha > 0$, and initialization $\vz_0$}
 Initialize $\vz^{1}_{0} \leftarrow \vz_0$\\
 \textcolor{olive}{SO}: Sample  $\tau \sim \textrm{Uniform}(\mathbb{S}_n)$ \\
 \For{$k \in [K]$}{
  \textcolor{violet}{RR}: Sample $\tau_k \sim \textrm{Uniform}(\mathbb{S}_n)$ \\
  \textcolor{olive}{SO}: $\tau_k \leftarrow \tau$ \\
  \textcolor{magenta}{AS}: Adversary chooses $\tau_k \in \mathbb{S}_n$ \\
  \For{$i \in [n]$}{
    $\vz^{k}_{i} \leftarrow \vz^{k}_{i-1} - \alpha  \omega_{\tau_k(i)}(\vz^{k}_{i-1})$\\
  }
   $\vz^{k+1}_{0}  \leftarrow \vz^{k}_{n}$\\
 }
 \caption{GDA-\textcolor{violet}{RR}/\textcolor{olive}{SO}/\textcolor{magenta}{AS}}
\end{algorithm}
\end{minipage}%
\hfill
\begin{minipage}[h]{0.50\textwidth}
\begin{algorithm}[H]
\label{algo:ppm-rr-so}
\SetAlgoLined
\SetKwInOut{Input}{Input}
\SetKwInOut{Output}{Output}
\Input{Number of epochs $K$, step-size $\alpha > 0$, and initialization $\vz_0$}
 Initialize $\vz^{1}_{0} \leftarrow \vz_0$\\
 \textcolor{olive}{SO}: Sample $\tau \sim \textrm{Uniform}(\mathbb{S}_n)$ \\
 \For{$k \in [K]$}{
  \textcolor{violet}{RR}: Sample $\tau_k \sim \textrm{Uniform}(\mathbb{S}_n)$ \\
  \textcolor{olive}{SO}: $\tau_k \leftarrow \tau$ \\
  \textcolor{magenta}{AS}: Adversary chooses $\tau_k \in \mathbb{S}_n$ \\
  \For{$i \in [n]$}{
    Solve the implicit update for $\vz^{k}_i$ where, \\
    $\vz^{k}_{i} = \vz^{k}_{i-1} - \alpha  \omega_{\tau_k(i)}(\vz^{k}_{i})$\\
  }
   $\vz^{k+1}_{0}  \leftarrow \vz^{k}_{n}$\\
 }
 \caption{PPM-\textcolor{violet}{RR}/\textcolor{olive}{SO}/\textcolor{magenta}{AS}}
\end{algorithm}
 \end{minipage}
\caption{GDA-\textcolor{violet}{RR}/\textcolor{olive}{SO}/\textcolor{magenta}{AS} and PPM-\textcolor{violet}{RR}/\textcolor{olive}{SO}/\textcolor{magenta}{AS} for solving \eqref{p:strong-monotone-vi}. \textcolor{violet}{Violet} lines denote steps that are only performed for \textcolor{violet}{RR}, \textcolor{olive}{Olive} lines denote the same for \textcolor{olive}{SO} and \textcolor{magenta}{Magenta} for \textcolor{magenta}{AS}.}
\end{figure*}
\subsection{Setting}
\label{sec:sc-sc-setting}
We analyze GDA-RR/SO and PPM-RR/SO for smooth finite-sum strongly convex-strongly concave (or SC-SC) objectives. This allows us to formulate the minimax optimization problem for $F$ as a \emph{root finding problem} for the gradient operator $\nu$, as described below.
\begin{assumption}[Strong Convexity-Strong Concavity]
\label{as:sc-sc}
The objective $F$ is $\mu$ strongly convex-strongly concave (or SC-SC), i.e., $F(., \vy)$ is $\mu$-strongly convex for any $\vy \in \bR^{d_{\vy}}$ and $-F(\vx, .)$ is $\mu$-strongly convex for any $\vx \in \bR^{d_{\vx}}$.
\end{assumption}
Assumption \ref{as:sc-sc} has the following consequences for the gradient operator $\nu$:
\begin{lemma}
\label{lem:sc-sc-monotone-vi}
Let $F$ satisfy Assumptions \ref{as:comp-smooth} and \ref{as:sc-sc}. Then, the gradient operator $\nu$ is $\mu$-strongly monotone:
$$\dotprod{\nu(\vz_1) - \nu(\vz_2)}{\vz_1 - \vz_2} \geq \mu \norm{\vz_1 - \vz_2}^2 \ \forall \ \vz_1, \vz_2 \in \bR^d.$$
Furthermore, \eqref{p:minimax-sum} admits a unique solution $\vz^*$, which is also the unique solution of $\nu(\vz^*) = 0$.
\end{lemma}
Lemma \ref{lem:sc-sc-monotone-vi} allows us to recast \eqref{p:minimax-sum} for SC-SC objectives as the following \emph{root finding problem}:
\begin{equation}
\label{p:strong-monotone-vi}
\textrm{Find } \vz \in \bR^d \textrm{ such that } \nu(\vz) = 1/n\sum_{i=1}^{n} \vw_i(\vz) = 0.
\end{equation}
Problem \eqref{p:strong-monotone-vi} is more general than SC-SC minimax optimization, and is a special case of strongly monotone variational inequalities \citep{FacchineiPang} without constraints. Notably, \eqref{p:strong-monotone-vi} includes the Nash Equilibrium problem for unconstrained \emph{multiplayer} games with smooth strongly convex objectives \citep{ScutariGamesVIP} and is sometimes called a \emph{finite-sum unconstrained variational inequality} in the literature \citep{GidelSGDAEC}. We also highlight that smooth strongly convex \emph{optimization} is a special case of \eqref{p:strong-monotone-vi}. However, unlike the optimization setting, $\nu$ is no longer restricted to be the gradient of a strongly convex function, which, as we shall see, has important consequences for the attainable convergence rates of our algorithms. 
\subsection{Analysis of RR/SO}
\label{sec:rr-so}
We now state the expected last-iterate convergence guarantees for Algorithms \ref{algo:gda-rr-so} and \ref{algo:ppm-rr-so} for solving \eqref{p:strong-monotone-vi}, where the expecation is taken over the stochasticity of the sampled permutation(s). 
\begin{theorem}[Convergence of GDA-RR/SO and PPM-RR/SO]
\label{thm:gda-wor-convergence}
Consider Problem \eqref{p:strong-monotone-vi} for the $\mu$-strongly monotone operator $\nu(\vz) = \nicefrac{1}{n} \sum_{i=1}^{n} \omega_i(\vz)$ where each $\omega_i$ is $l$-Lipschitz, but not necessarily monotone. Let $\vz^*$ denote the unique root of $\nu$. Then, there exists a step-size $\alpha \leq \nicefrac{\mu}{5nl^2}$ for which both GDA-RR/SO and PPM-RR/SO satisfy the following for any $K \geq 1$:
$$\E{|\vz^{K+1}_0 - \vz^*|^2} \leq 2e^{\nicefrac{-K}{5\kappa^2}}|\vz_0 - \vz^*|^2 + \frac{2\mu^2 + 8 \kappa^2 \sigma^2_* \log^3 (|\nu(\vz_0)|n^{1/2}K/\mu)}{\mu^2 nK^2} = \Tilde{O}(e^{\nicefrac{-K}{5\kappa^2}} + \nicefrac{1}{nK^2}),$$
where $\kappa=\nicefrac{l}{\mu}$ is the condition number and $\sigma_*^2 = \nicefrac{1}{n}\sum_{i=1}^{n}\norm{\vw_{i}(\vz^*)}^2$ is the gradient variance at $\vz^*$.
\end{theorem}
\begin{proof}We present an outline for GDA-RR/SO and defer the full proof to Appendix \ref{app-sec:gda-wor} (for GDA-RR/SO) and Appendix \ref{app-sec:ppm-wor} (for PPM-RR/SO). Furthermore, we recall that the updates of GDA-RR/SO are given by $\vz^{k}_{i} = \vz^{k}_{i-1} - \alpha \vw_{\tau_k(i)}(\vz^{k}_{i-1})$.

We begin with the following key insight from earlier works on sampling without replacement for minimization \citep{HaochenSraRR19, AhnRR2020, NedicBertsekas, GurbuzIG2019}: for small enough step-sizes, the \emph{epoch iterates} $\vz^k_0$ of GD without replacement approximately follow the trajectory of full-batch gradient descent. To this end, we derive the following \emph{epoch-level} update rule for GDA-RR/SO by linearizing $\vw_{\tau_k(i)}(\vz^{k}_{i-1})$ around $\vz^*$:
\begin{equation}
\label{eqn:sketch-gda-rr-so}
\vz^{k+1}_0 - \vz^* = \vH_k(\vz^{k}_0 - \vz^*) + \alpha^2 \vr_k,
\end{equation}
where $\norm{\vH_k} \leq 1 - n\alpha \mu/2$ and $\vr_k = \sum_{i=1}^{n-1} \vA_{\tau_k(i)} \sum_{j=1}^{i} \vw_{\tau_k(j)}(\vz^*)$ with $\norm{\vA_{\tau_k(i)}} \leq le^{1/5}$. The term $\vr_k$ encapsulates the noise of the stochastic gradient updates accumulated over an entire epoch. To ensure convergence, we control the influence of the noise term $\vr_k$ by using standard properties of without-replacement sample means to show that $\mathbb{E}[|\vr_k|^2] \leq l^2 n^3 \sigma^2_* / 4$ for both RR and SO. We then complete the proof by unrolling \eqref{eqn:sketch-gda-rr-so} for $K$ epochs, substituting the upper bounds for $\norm{\vH_k}$ and $\mathbb{E}[|\vr_k|^2]$ wherever necessary, and setting $\alpha = \min \{  \mu/5nl^2, 2 \log (\norm{\nu(\vz_0 )}n^{1/2}K/\mu )/\mu n K \}$.

As we show in Appendix \ref{app-sec:gda}, the update rule \eqref{eqn:sketch-gda-rr-so} resembles the linearized update rule of full batch GDA with added noise. In fact, for $n=1$, $\vr_k = 0$ and thus, we recover the rates of full-batch GDA. Expressing GDA-RR/SO (and later AS) as noisy full-batch GDA in this fashion is a central component of our unified analysis, and relies on the fact that $\sum_{i=1}^{n} \omega_{\tau_k(i)}(\vz^*) = 0 \ \forall \ \tau_k \in \mathbb{S}_n$, which is specific to sampling without replacement. Comparing to SGDA with replacement, we note that sampling the components i.i.d. uniformly as $u(i) \sim \textrm{Uniform}([n])$ gives rise to an \emph{additional} noise term $\alpha \vp_k$ in the update rule, where $\vp_k = \sum_{i=1}^{n} \omega_{u(i)}(\vz^*)$ vanishes only in expecation, and has a variance of $\mathbb{E}[|\vp_k|^2] = n\sigma^2_*$. Subsequently, the dominant noise term for SGDA updates is $O(\alpha^2 n \sigma^2_*)$ whereas that of GDA-RR/SO is $O(\alpha^4 n^3 \sigma^2_*)$, which qualitatively demonstrates the \emph{implicit variance reduction} of sampling without replacement. As we shall see in the complete proof, this allows RR/SO to converge faster (for large enough $K$) by carefully selecting $\alpha$.
\end{proof}
\textbf{Comparison with lower bounds:} Since smooth strongly convex minimization is a special case of \eqref{p:strong-monotone-vi}, the $\Omega(1/nK^2)$ lower bound established in prior works \citep{SafranShamirRR2020, RajputRR2020} for smooth strongly convex minimization using GD with RR/SO also applies to GDA-RR/SO. Comparing with this lower bound, we note that the convergence rate of GDA-RR/SO is \emph{nearly tight}, i.e., it differs from the lower bound only by an exponentially decaying term.  In fact, for $K \geq 10\kappa^2 \log(n^{1/2}K)$, the convergence rate becomes $\Tilde{O}(1/nK^2)$, which precisely matches the lower bound, modulo logarithmic factors. 

\textbf{Comparison with uniform sampling:} Similarly, the $\Omega(1/nK)$ lower bound of SGD with replacement for smooth and strongly convex functions \citep{RakhlinSridharan} also applies to SGDA with replacement. On the contrary, both GDA-RR and GDA-SO converge with a faster rate of $\Tilde{O}(1/nK^2)$ when $K \geq 10\kappa^2 \log(n^{1/2}K)$. Thus, GDA-RR/SO provably outperform SGDA with replacement (modulo logarithmic factors) when $K \geq 10\kappa^2 \log(n^{1/2}K)$. As we show in Appendix \ref{app-sec:gda-wor}, the $\kappa^2$ dependence of this inequality cannot be improved for constant step-sizes. A similar argument also applies to stochastic PPM. To the best of our knowledge, the fastest known convergence rate for stochastic PPM is $O(1/nK)$ for minimizing smooth strongly convex functions \citep{NecoaraSPPM}. Hence, Theorem \ref{thm:gda-wor-convergence} suggests that PPM-RR/SO enjoy a faster $\Tilde{O}(1/nK^2)$ convergence rate \emph{for both minimization and minimax optimization} when $K \geq 10\kappa^2 \log(n^{1/2}K)$.
\subsection{Analysis in the Adversarial Shuffling Regime}
\label{sec:as}
We now consider without-replacement minimax optimization algorithms in an adversarial setting. We focus on a novel class of training-time attacks known as \emph{data ordering attacks} proposed by \citet{Papernot_DataOrdering_2021}. These attacks differ from standard data-perturbation attacks \citep{GoodfellowPGD} and exploit the fact that most implementations of stochastic gradient optimizers do not verify whether the permutation $\tau_k$ is truly sampled at random. \citet{Papernot_DataOrdering_2021} propose three distinct attack strategies, namely, \emph{batch reordering}, which changes the order in which mini-batches are supplied to the algorithm, \emph{batch reshuffling}, which changes the order in which individual data points are supplied, and \emph{replacing} which prevents certain data points from being observed by the algorithm by consistently replacing them with other data points in the training set.

We analyze the convergence of without-replacement GDA and PPM under batch reshuffling attacks. To this end, we consider an adversarial modification of RR/SO where the permutations $\tau_k$, instead of being sampled by the algorithm, are now selected by an adversary using a strategy unknown to the algorithm. We also the assume that, while choosing $\tau_k$, the adversary is computationally unrestricted and has complete knowledge of all the components $\omega_i$, the minimax point $\vz^*$, and the iterates $\vz^{k}_{i}$ observed so far. We call this setup \emph{Adversarial Shuffling} (AS) and obtain convergence rates of GDA and PPM (named GDA-AS and PPM-AS) when solving \eqref{p:strong-monotone-vi}. Thus, our analysis naturally holds for minimization, minimax optimization as well as finite-sum multiplayer games. The details are stated in Algorithms \ref{algo:gda-rr-so} and \ref{algo:ppm-rr-so}, respectively. Our last iterate convergence guarantees are deterministic and hold uniformly over any sequence of permutations $\tau_1, \ldots, \tau_K$ that the adversary can choose. 
\begin{theorem}[Convergence of GDA-AS and PPM-AS]
\label{thm:gda-as-convergence}
Consider Problem \eqref{p:strong-monotone-vi} for the $\mu$-strongly monotone operator $\nu(\vz) = \nicefrac{1}{n} \sum_{i=1}^{n} \omega_i(\vz)$ where each $\omega_i$ is $l$-Lipschitz, but not necessarily monotone. Let $\vz^*$ denote the unique root of $\nu$. Then, there exists a step-size $\alpha \leq \nicefrac{\mu}{5nl^2}$ for which both GDA-AS and PPM-AS satisfy the following for any $K \geq 1$:
$$\max_{\tau_1, \ldots, \tau_K \in \mathbb{S}_n} \! |\vz^{K+1}_0 - \vz^*|^2 \! \leq \! 2e^{\nicefrac{-K}{5\kappa^2}}|\vz_0 - \vz^*|^2 \! +  \frac{2\mu^2 \! + \! 24 \kappa^2 \sigma^2_* \log^3 (|\nu(\vz_0)|K/\mu)}{\mu^2 K^2} \! = \! \Tilde{O}(e^{\nicefrac{-K}{5\kappa^2}} \! +  \nicefrac{1}{K^2}),$$
where $\kappa, \sigma^2_*$ are as defined in Theorem \ref{thm:gda-wor-convergence} and $\tau_1, \ldots, \tau_K$ are the permutations chosen by the adversary.
\end{theorem}
\textbf{Convergence rates of IG and comparison with lower bounds:} We note that the Incremental Gradient and the Incremental Proximal Point Methods, which do not shuffle the data, are a special case of GDA-AS/PPM-AS with $\tau_1, \ldots, \tau_K = id$. Thus, Theorem \ref{thm:gda-as-convergence} also gives us convergence rates for GDA-IG/PPM-IG. Moreover, since GDA-AS generalizes GDA-IG and \eqref{p:strong-monotone-vi} covers minimization, the $\Omega(1/K^2)$ lower bound of IG established in prior works \citep{SafranShamirRR2020} for smooth strongly convex minimization also applies to GDA-AS. Thus, our obtained rate for GDA-AS is nearly tight and matches the lower bound (modulo logarithmic factors) for $K \geq 10\kappa^2 \log(K)$.

\textbf{Effectiveness of batch reshuffling:} When $K \geq 10\kappa^2 \log(K)$, $\Tilde{O}(1/K^2)$ becomes the dominant term in the convergence rate of AS. This is worse than that of RR/SO by a factor of $1/n$ and causes a significant slowdown in convergence, since, in many applications, the dataset size $n$ is much larger than $K$. Thus, our analysis justifies the effectiveness of batch reshuffling attacks in reducing model accuracy and increasing training time, which is empirically verified by \citet{Papernot_DataOrdering_2021}.
\section{RR for Two-Sided P\L{} Objectives}
\label{sec:2pl}
We now analyze RR for a class of smooth nonconvex-nonconcave problems where the objective $F$ satisfies a \emph{two-sided Polyak \L{}ojasiewicz inequality}, first proposed in \citet{YangAGDA2020}. We denote this function class as 2P\L{} and formally state the assumption as follows. 
\begin{assumption}[Two-sided Polyak \L{}ojasiewicz Inequality or 2P\L{} condition]
\label{as:2pl}
For any $\vx \in \bR^{d_{\vx}}, \vy \in \bR^{d_{\vy}}$, the sets $\arg \max_{\Tilde{\vy}} F(\vx, \Tilde{\vy})$ and $\arg \min_{\Tilde{\vx}} F(\Tilde{\vx}, \vy)$ are non-empty. Furthermore, there exist positive constants $\mu_1, \mu_2$ such that $F$ satisfies the following:
\begin{align*}
    \norm{\grad_\vx F(\vx, \vy)}^2 \geq 2\mu_1 [F(\vx, \vy) - \min_{\Tilde{\vx} \in \bR^{d_\vx}} F(\Tilde{\vx}, \vy)], \
    \norm{\grad_\vy F(\vx, \vy)}^2 \geq 2\mu_2 [ \max_{\Tilde{\vy} \in \bR^{d_\vy}} F(\vx, \Tilde{\vy})-F(\vx, \vy)].
\end{align*}
\end{assumption}
The 2P\L{} condition is satisfied in several practical settings, including, but not limited to, robust least squares \citep{RobustLS}, imitation learning for linear quadratic regulators \citep{FazelLQR, CaiLQR}, and various other problems in reinforcement learning and robust control \citep{SimonDuMinimaxRL, CaiLQR}. Clearly, any SC-SC function is 2P\L{}. However, 2P\L{} functions need not be SC-SC, or even convex-concave. We refer the readers to \citet{YangAGDA2020} for a detailed discussion of the 2P\L{} class and its applications.

Analysis of RR for 2P\L{} objectives is challenging not only due to nonconvexity-nonconcavity, but also because $F$ may not have a unique minimax point. Indeed, as we demonstrate in Appendix \ref{app-sec:agda-lemmas}, it is possible to construct 2P\L{} functions where the set of minimax points is an unbounded proper subset of $\bR^d$. Hence, the notion of \emph{gradient variance at the minimax point}, which we used in our earlier analyses, is no longer meaningful. To overcome this, we impose the following assumption. 
\begin{assumption}[Bounded Gradient Variance]
\label{as:bgv}
There exists a positive constant $\sigma$ such that the component gradient operators $\omega_i$ satisfy the following for any $\vz \in \bR^{d}$:
$$1/n\sum_{i=1}^{n} \norm{\vw_i(\vz) - \nu(\vz)}^2 \leq \sigma^2.$$
\end{assumption}
\begin{figure*}[ht]
\begin{minipage}[h]{0.5\textwidth}
\begin{algorithm}[H]
\label{algo:agda-rr}
\SetAlgoLined
\SetKwInOut{Input}{Input}
\SetKwInOut{Output}{Output}
\Input{Number of epochs $K$, step-sizes $\alpha, \beta > 0$, and initialization $(\vx_0, \vy_0)$}
 Initialize $(\vx^{1}_{0}, \vy^{1}_0) \leftarrow (\vx_0, \vy_0)$\\
 \For{$k \in [K]$}{
  Sample a permutation $\tau_k \in \mathbb{S}_n$ \\
  \For{$i \in [n]$}{
    $\vx^{k}_{i} \leftarrow \vx^{k}_{i-1} - \alpha  \grad_{\vx} f_{\tau_k(i)}(\vx^{k}_{i-1}, \vy^k_0)$\\
  }
  Sample a permutation $\pi_k \in \mathbb{S}_n$ \\
  \For{$i \in [n]$}{
    $\vy^{k}_{i} \leftarrow \vy^{k}_{i-1} + \beta  \grad_{\vy} f_{\pi_k(i)}(\vx^{k}_{n}, \vy^k_{i-1})$\\
  }
   $(\vx^{k+1}_{0}, \vy^{k+1}_{0})  \leftarrow (\vx^{k}_{n}, \vy^{k}_{n})$\\
 }
 \caption{AGDA-RR}
\end{algorithm}
\end{minipage}%
\begin{minipage}[h]{0.5\textwidth}
\begin{algorithm}[H]
\label{algo:agda-as}
\SetAlgoLined
\SetKwInOut{Input}{Input}
\SetKwInOut{Output}{Output}
\Input{Number of epochs $K$, step-sizes $\alpha, \beta > 0$, and initialization $(\vx_0, \vy_0)$}
 Initialize $(\vx^{1}_{0}, \vy^{1}_0) \leftarrow (\vx_0, \vy_0)$\\
 \For{$k \in [K]$}{
  Adversary chooses a permutation $\tau_k \in \mathbb{S}_n$ \\
  \For{$i \in [n]$}{
    $\vx^{k}_{i} \leftarrow \vx^{k}_{i-1} - \alpha  \grad_{\vx} f_{\tau_k(i)}(\vx^{k}_{i-1}, \vy^k_0)$\\
  }
  Adversary chooses a permutation $\pi_k \in \mathbb{S}_n$ \\
  \For{$i \in [n]$}{
    $\vy^{k}_{i} \leftarrow \vy^{k}_{i-1} + \beta  \grad_{\vy} f_{\pi_k(i)}(\vx^{k}_{n}, \vy^k_{i-1})$\\
  }
   $(\vx^{k+1}_{0}, \vy^{k+1}_{0})  \leftarrow (\vx^{k}_{n}, \vy^{k}_{n})$\\
 }
 \caption{AGDA-AS}
\end{algorithm}
 \end{minipage}
\end{figure*}
\subsection{Analysis of AGDA-RR and AGDA-AS}
\label{sec:agda-rr-as}
In order to establish the provable benefits of RR for smooth finite-sum minimax optimization of 2P\L{} objectives, we propose the \emph{Alternating Gradient Descent Ascent with Random Reshuffling} (AGDA-RR) algorithm. AGDA-RR achieves near-optimal convergence guarantees for 2P\L{} objectives by combining RR with \emph{alternating updates} \citep{GidelNegativeMomentum, GuojunStabilityGames, GidelStabilityAlt} and \emph{timescale separation} \citep{LinJinJordanSGDA2020, TannerTimescaleLocal, TannerTimescaleGlobal}, two ideas that have been very useful for improving convergence and stability in nonconvex-nonconcave minimax optimization. Within each epoch $k \in [K]$, AGDA-RR uniformly samples a random permutation $\tau_k$, makes one full pass over the dataset, and performs gradient descent (with RR) updates for the variable $\vx$ using the permutation $\tau_k$. This is followed by sampling another permutation $\pi_k$ and performing gradient ascent (with RR) updates for $\vy$ using the permutation $\pi_k$. The detailed procedure is stated in Algorithm \ref{algo:agda-rr}. We also analyze a variant of AGDA-RR in the adversarial shuffling setting, which we call AGDA-AS. The procedure, as described in Algorithm \ref{algo:agda-as}, is almost identical to AGDA-RR, except that the permutations $\tau_k$ and $\pi_k$ are chosen by an adversary. 

Before presenting a convergence analysis, we highlight that the absence of a unique minimax point prevents us from using the squared distance to the optimum as a Lyapunov function. To this end, we use the Lyapunov function $V_{\lambda} : \bR^d \rightarrow \bR$ which was previously suggested by \citet{YangAGDA2020}. We begin by first defining the \emph{best response function} $\Phi : \bR^d \rightarrow \bR$ and its minimum $\Phi^*$ as follows,
\begin{align*}
    \Phi(\vx) = \max_{\vy \in \bR^{d_{\vy}}} F(\vx, \vy), \ \ \
    \Phi^* = \min_{\vx \in \bR^{d_{\vx}}} \Phi(\vx) = \min_{\vx \in \bR^{d_{\vx}}} \max_{\vy \in \bR^{d_{\vy}}} F(\vx, \vy). 
\end{align*}
Assumption \ref{as:2pl} ensures that $\Phi$ is well defined and the existence of a global minimax point guarantees that $\Phi^*$ is finite. Subsequently, for any $\lambda > 0$, we define the Lyapunov function $V_{\lambda}$ as 
$$V_{\lambda}(\vx, \vy) = [\Phi(\vx) - \Phi^*] + \lambda[\Phi(\vx) - F(\vx, \vy)].$$
By definition of $\Phi$, $V_{\lambda}$ is non-negative for any $\lambda > 0$ and $V_{\lambda}(\vz) = 0$ if and only if $\vz$ is a minimax point of $F$. Hence, we present our convergence proofs for AGDA-RR and AGDA-AS in terms of $V_{\lambda}$.
\begin{theorem}[Convergence of AGDA-RR/AS]
\label{thm:agda-rr}
Let Assumptions \ref{as:comp-smooth}, \ref{as:2pl}, and \ref{as:bgv} hold and let $\eta = \nicefrac{73l^2}{2\mu^2_2}$. Then, there exists a step-size $\alpha \leq \nicefrac{1}{5\eta n l}$ such that for $\beta = \eta \alpha$, AGDA-RR satisfies the following for $\lambda = \nicefrac{1}{10}$ and any $K \geq 1$:
$$\mathbb{E}[V_{\lambda}(\vz^{K+1}_0)] \leq e^{\nicefrac{-K}{365\kappa^3}}V_{\lambda}(\vz_0) + \frac{\mu_1 + c\kappa^8\sigma^2 \log^2(V_{\lambda}(\vz_0)n^{1/2}K)}{\mu_1 nK^2} = \Tilde{O}(e^{\nicefrac{-K}{365\kappa^3}} + \nicefrac{1}{nK^2}),$$
where $\kappa = \max \{ \nicefrac{l}{\mu_1}, \nicefrac{l}{\mu_2} \}$ and $c>0$ is a constant independent of $\kappa, \mu_1, \mu_2, \sigma^2$. Under the same setting, AGDA-AS satisfies the following ($\hat{c} > 0$ is a constant independent of $\kappa, \mu_1, \mu_2, \sigma^2$):
$$\max_{\tau_1, \pi_1, \ldots, \tau_K, \pi_K \in \mathbb{S}_n} \! \! V_{\lambda}(\vz^{K+1}_0) \! \leq \! e^{\nicefrac{-K}{365\kappa^3}}V_{\lambda}(\vz_0) + \frac{\mu_1 \! + \! \hat{c}\kappa^8\sigma^2 \log^2(V_{\lambda}(\vz_0)K)}{\mu_1 K^2} \! = \! \Tilde{O}(e^{\nicefrac{-K}{365\kappa^3}} + \nicefrac{1}{K^2}),$$
where $\tau_1, \pi_1 \ldots, \tau_K, \pi_K$ are the permutations chosen by the adversary. 
\end{theorem}
\textbf{Convergence to a Saddle Point:} As demonstrated in Appendix \ref{app-sec:agda-conv-guarantee}, the convergence guarantee of Theorem \ref{thm:agda-rr}, which is presented in terms of $V_\lambda$, can be easily translated into an equivalent convergence guarantee in terms of $\textrm{dist}(\vz, \sZ^*)^2$, where $\sZ^*$ denotes the set of saddle points of $F$. In particular, Theorem \ref{thm:agda-rr} implies the following convergence guarantee for AGDA-RR:
\begin{align*}
   \bE[\textrm{dist}(\vz^{K+1}_0, \sZ^*)^2] = \Tilde{O}(e^{-\nicefrac{K}{365\kappa^3}} + \nicefrac{1}{nK^2}), 
\end{align*}
as well as the following convergence rate for AGDA-AS:
\begin{align*}
    \max_{\tau_1, \pi_1, \ldots, \tau_K, \pi_K \in \mathbb{S}_n}  \textrm{dist}(\vz^{K+1}_0, \sZ^*)^2 = \Tilde{O}(e^{-\nicefrac{K}{365\kappa^3}} + \nicefrac{1}{K^2}),
\end{align*}

\textbf{Comparison with lower bounds:} Strongly convex minimization is a special case of 2P\L{} minimax optimization, since minimizing the strongly convex function $f$ is equivalent to minimax optimization of the 2P\L{} function $F(\vx, \vy) = f(\vx) - \dotprod{\vy}{\vy}$. In fact, the $\vx$ iterates of AGDA-RR for $F$ are exactly that of GD with RR for $f$. Hence, the $\Omega(1/nK^2)$ lower bound for strongly convex minimization using GD with RR also applies to AGDA-RR. Thus our convergence rate for AGDA-RR is nearly tight and matches the lower bound (modulo logarithmic factors) for $K \geq 730 \kappa^3 \log(n^{1/2}K)$. Similarly, the Incremental Gradient version of AGDA is a special case of AGDA-AS with $\tau_1,\pi_1,\ldots,\tau_K,\pi_K = id$ and hence, AGDA-AS is nearly tight and matches the $\Omega(1/K^2)$ lower bound (modulo logarithmic factors) for $K \geq 730 \kappa^3 \log(K)$.

\textbf{Comparison with stochastic AGDA:} Similarly, the $\Omega(1/nK)$ lower bound of SGD with replacement also holds for the Stochastic AGDA algorithm \citep{YangAGDA2020}, which samples the component functions with replacement and performs two-timescale alternating updates similar to AGDA-RR. Hence, Theorem \ref{thm:agda-rr} demonstrates that AGDA-RR provably outperforms stochastic AGDA when $K \geq 730 \kappa^3 \log(n^{1/2}K)$.

\textbf{Bounded iterate assumption} Assumption \ref{as:bgv} is also used in analyzing RR for P\L{} function minimization \citep{MischenkoRR2020}. In this setting, an alternative \emph{bounded iterate assumption}, which assumes that all the iterates $\vz^k_i$ lie within a compact set, has also been used \citep{AhnRR2020}. As shown in Appendix \ref{app-sec:agda-rr}, our proof of Theorem \ref{thm:agda-rr} easily adapts to this assumption. In the absence of either assumption, \citet{KLRR2021} use time-varying step-sizes to obtain \emph{asymptotic} $O(1/K^2)$ rates for RR on P\L{} (and more generally for K\L{}) minimization.

\begin{figure}[t]
     \hspace{-0.05\textwidth}
     \begin{subfigure}[b]{0.372\textwidth}
         \centering
         \includegraphics[width=\textwidth]{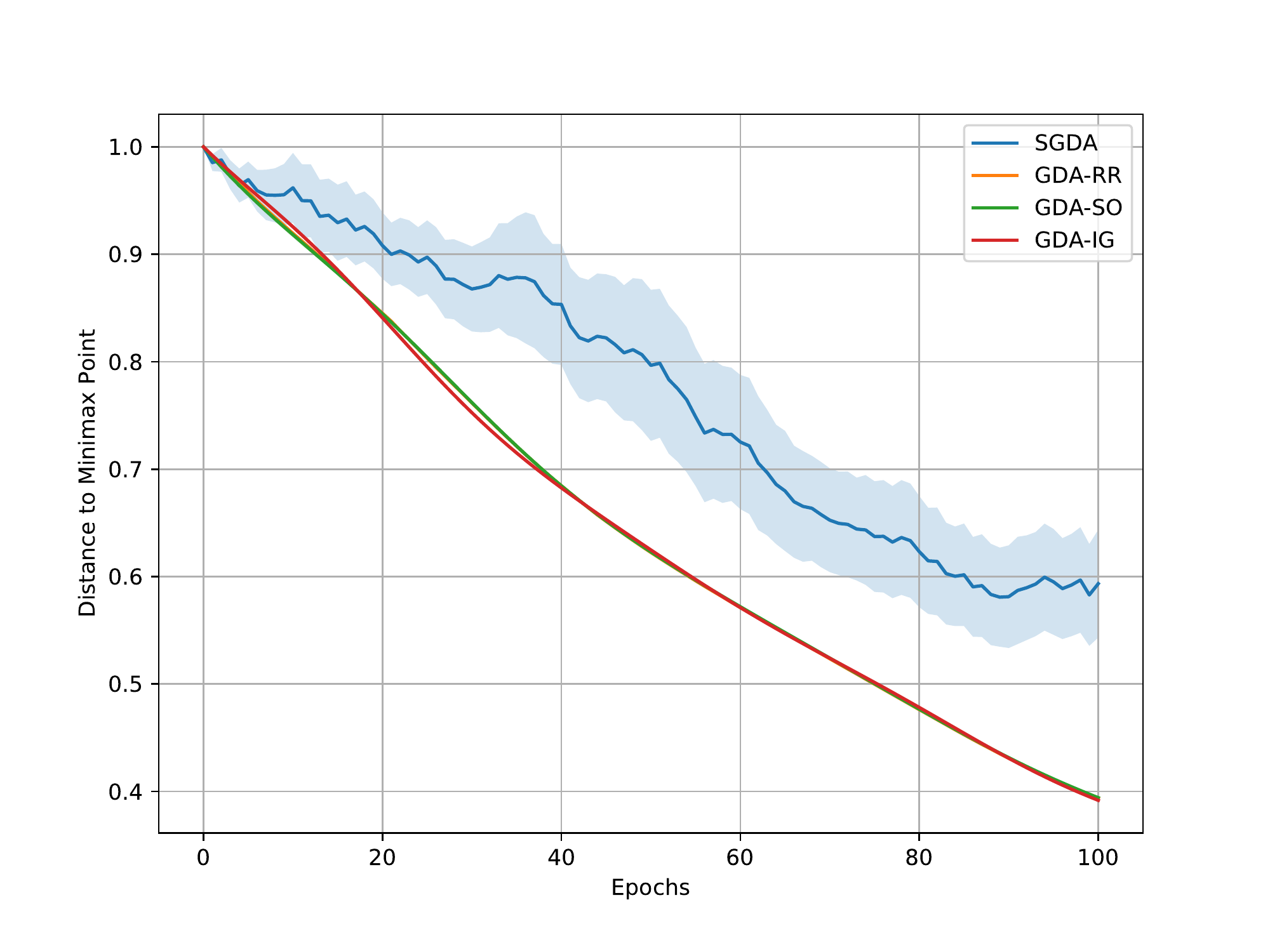}
         \caption{GDA}
         \label{graph:gda}
     \end{subfigure}
     \hspace{-0.045\textwidth}
     \begin{subfigure}[b]{0.372\textwidth}
         \centering
         \includegraphics[width=\textwidth]{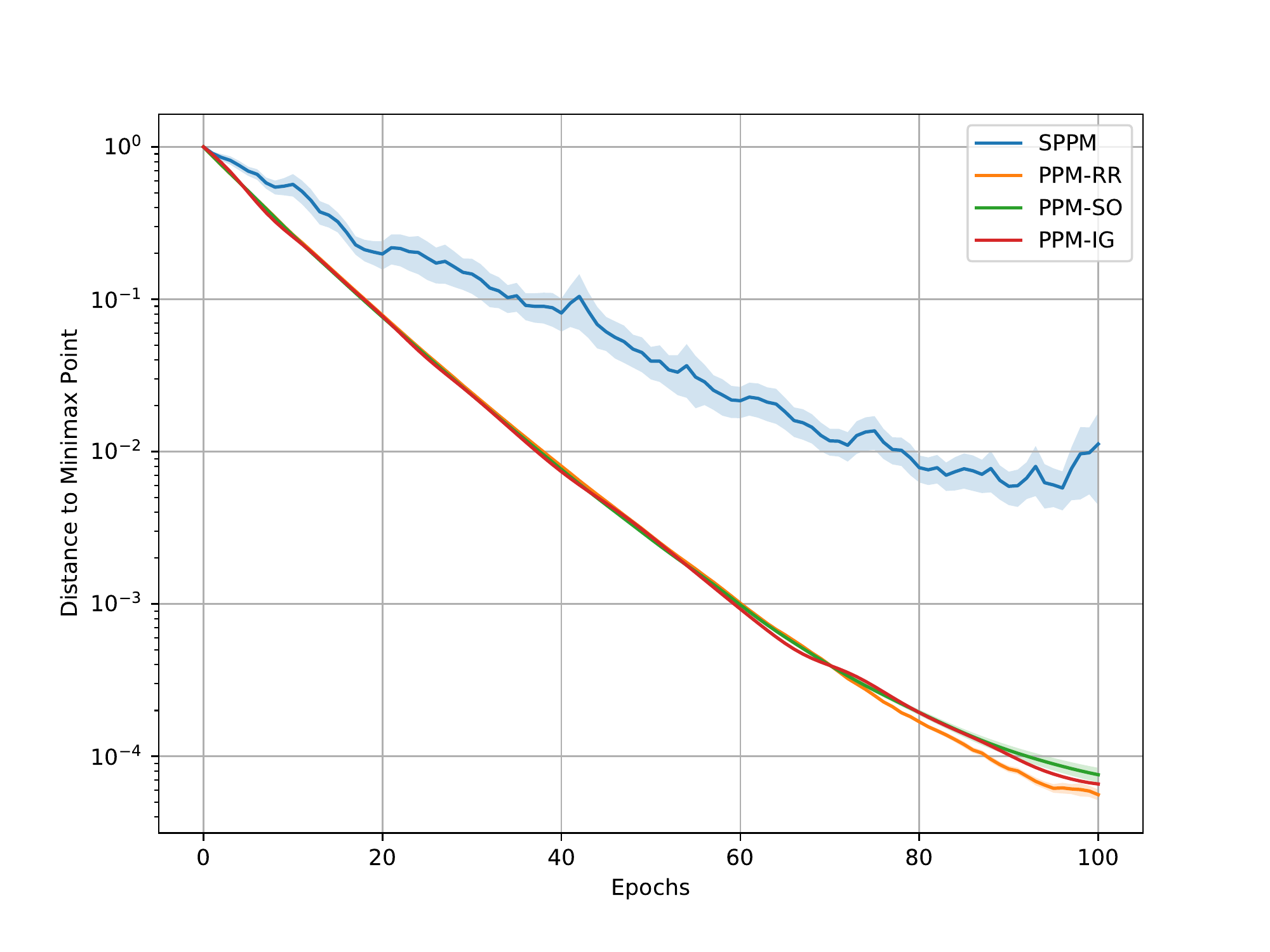}
         \caption{PPM}
         \label{graph:ppm}
     \end{subfigure}
     \hspace{-0.045\textwidth}
     \begin{subfigure}[b]{0.372\textwidth}
         \centering
         \includegraphics[width=\textwidth]{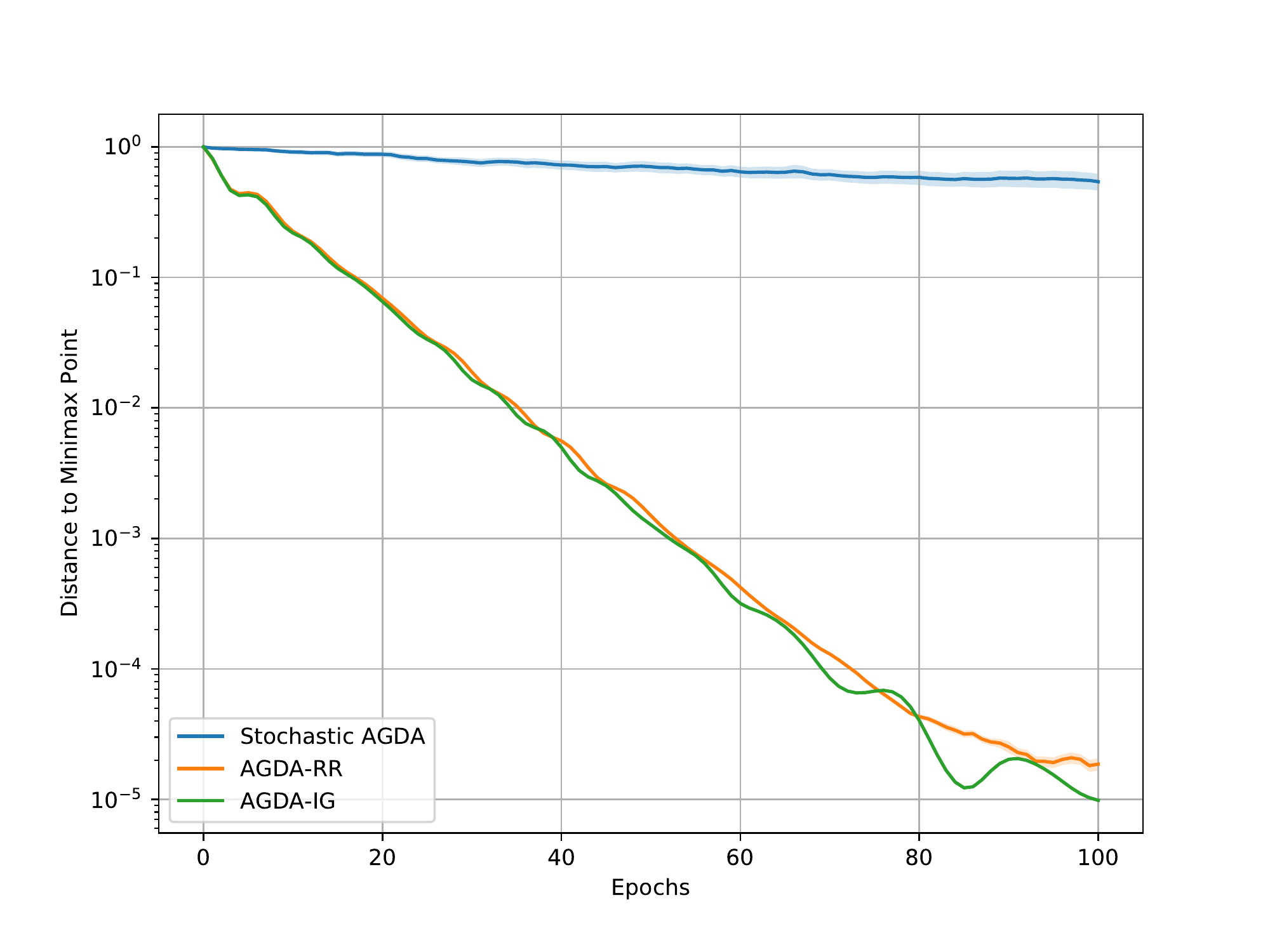}
         \caption{AGDA}
         \label{graph:agda}
     \end{subfigure}
        \caption{Relative distance of the epoch iterates from the global minimax point (i.e. $|\vz^k_0 - \vz^*|^2/|\vz_0 - \vz^*|^2$ vs $k$). The solid lines are the average over 50 runs and the shaded regions are 95\%  confidence intervals. The y-axis of \ref{graph:gda} is on a linear scale whereas that of \ref{graph:ppm} and \ref{graph:agda} is on a logarithmic scale.}
        \label{fig:quad-plots}
\end{figure}

\section{Experiments}
\label{sec:exps}
We evaluate our theoretical results by benchmarking on finite-sum SC-SC quadratic minimax games. This class of problems appears in several applications such as reinforcement learning \citep{SimonDuMinimaxRL}, robust regression, \citep{RobustLS} and online learning \citep{KoolenOnlineQuadratic}. The objective $F$ and the components $f_i$ are given by:
\begin{align*}
F(\vx, \vy) &= 1/n \sum_{i=1}^{n} f_i(\vx, \vy) =  \frac{1}{2} \vx^T \vA \vx + \vx^T \vB \vy - \frac{1}{2} \vy^T \vC \vy, \\
f_i(\vx, \vy) &= \frac{1}{2} \vx^T \vA_i \vx + \vx^T \vB_i \vy - \frac{1}{2} \vy^T \vC_i \vy - \vu_i^T \vx - \vv_i^T \vy,
\end{align*}
where $\vA$ and $\vC$ are strictly positive definite. We generate the components $f_i$ randomly, such that $\sum_{i=1}^{n} \vu_i = \sum_{i=1}^{n} \vv_i = 0$ and the expected singular values of $\vB$ are larger than that of $\vA$ and $\vC$. This ensures that the bilinear coupling term $\vx^T \vB \vy$ is sufficiently strong, since a weak coupling practically reduces to quadratic minimization, which has already been investigated in prior works. Finally, to investigate how the presence of nonconvex-nonconcave components impacts convergence, a few randomly chosen $f_i$'s are allowed to be nonconvex-nonconcave quadratics.  For each algorithm analyzed in the text, we benchmark sampling without replacement against uniform sampling by running each method for 100 epochs using constant step-sizes that are selected independently for each method via grid search. Further details regarding the setup is discussed in Appendix \ref{app-sec:experiments}.

We present our results in Figure \ref{fig:quad-plots}, where we plot the relative distance of the epoch iterates from the minimax point, defined as $|\vz^k_0 - \vz^*|^2/|\vz_0 - \vz^*|^2$, averaged over 50 independent runs. In agreement with our theoretical findings, sampling without replacement consistently outperforms uniform sampling across all three setups. Furthermore, to demonstrate that our observations are not particular to one specific instance, we repeat the experiment for 20 independently sampled quadratic games, and for each instance, perform 5 independent runs of each algorithm and plot the average relative distance of the epoch iterates from the minimax point. The results, presented in Figure \ref{fig:avg-quad-plots}, substantiates the superior convergence of sampling without replacement across multiple instances.

\begin{figure}[t]
     \hspace{-0.05\textwidth}
     \begin{subfigure}[b]{0.372\textwidth}
         \centering
         \includegraphics[width=\textwidth]{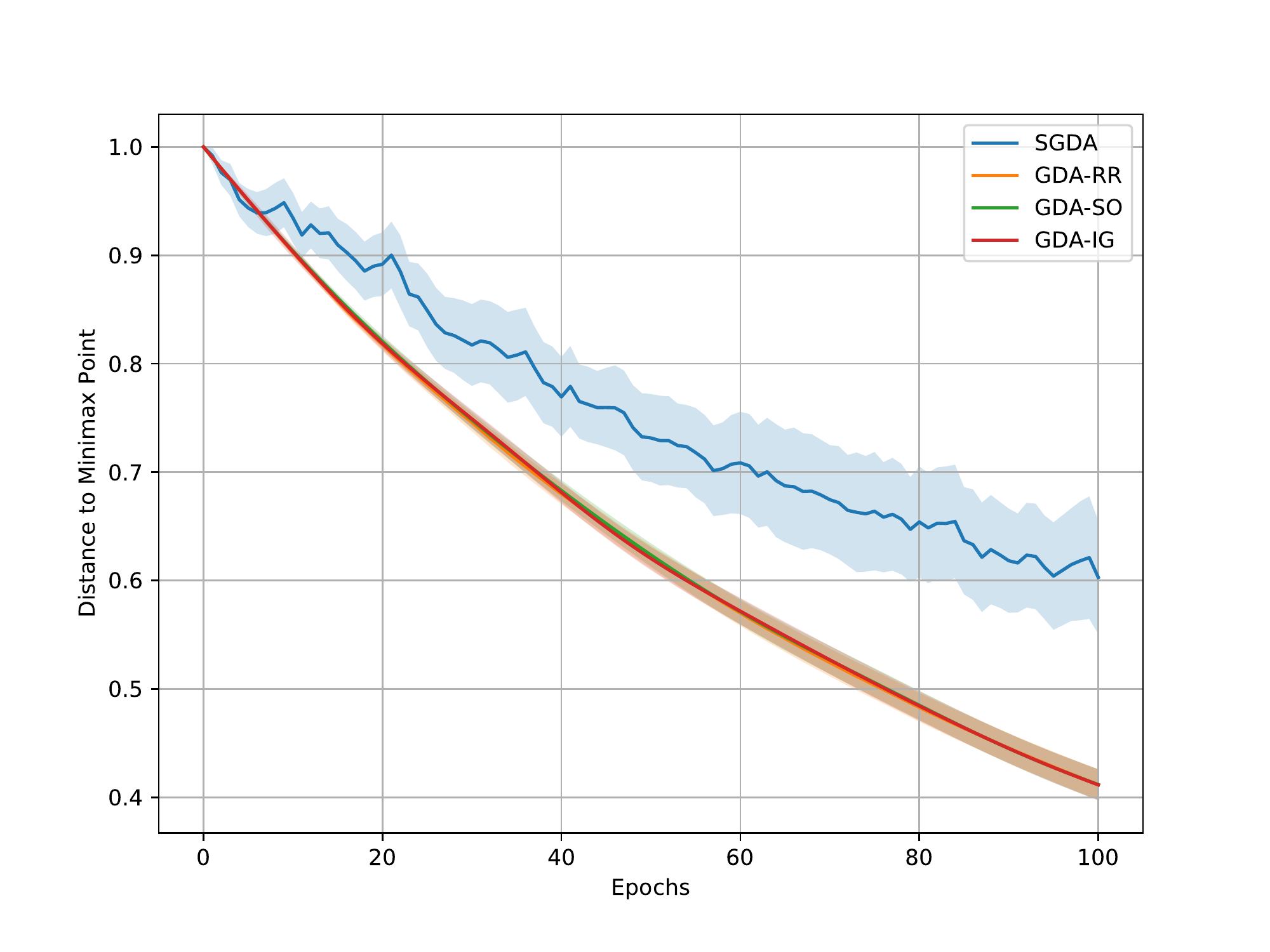}
         \caption{GDA}
         \label{graph:gda}
     \end{subfigure}
     \hspace{-0.045\textwidth}
     \begin{subfigure}[b]{0.372\textwidth}
         \centering
         \includegraphics[width=\textwidth]{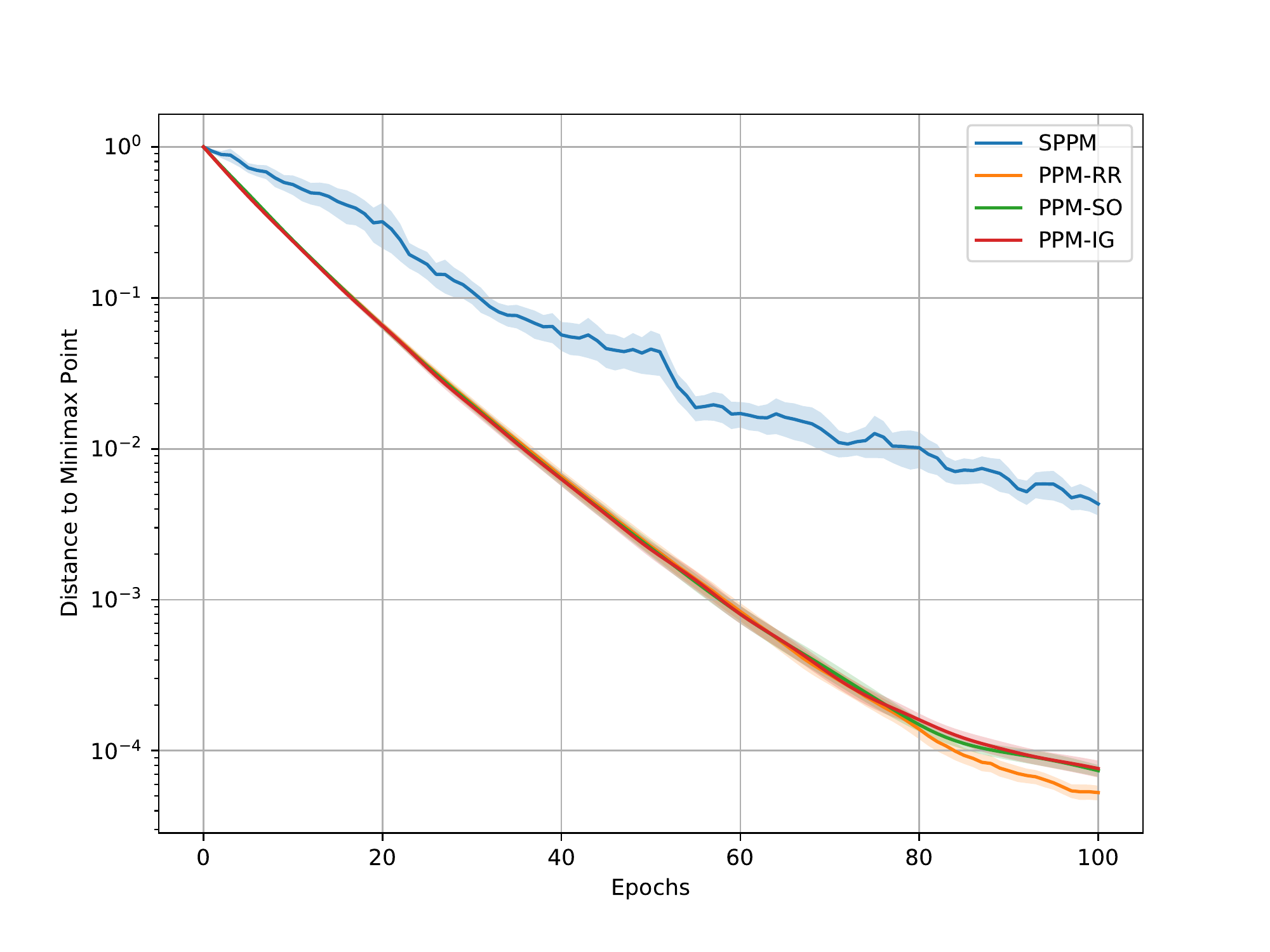}
         \caption{PPM}
         \label{graph:ppm}
     \end{subfigure}
     \hspace{-0.045\textwidth}
     \begin{subfigure}[b]{0.372\textwidth}
         \centering
         \includegraphics[width=\textwidth]{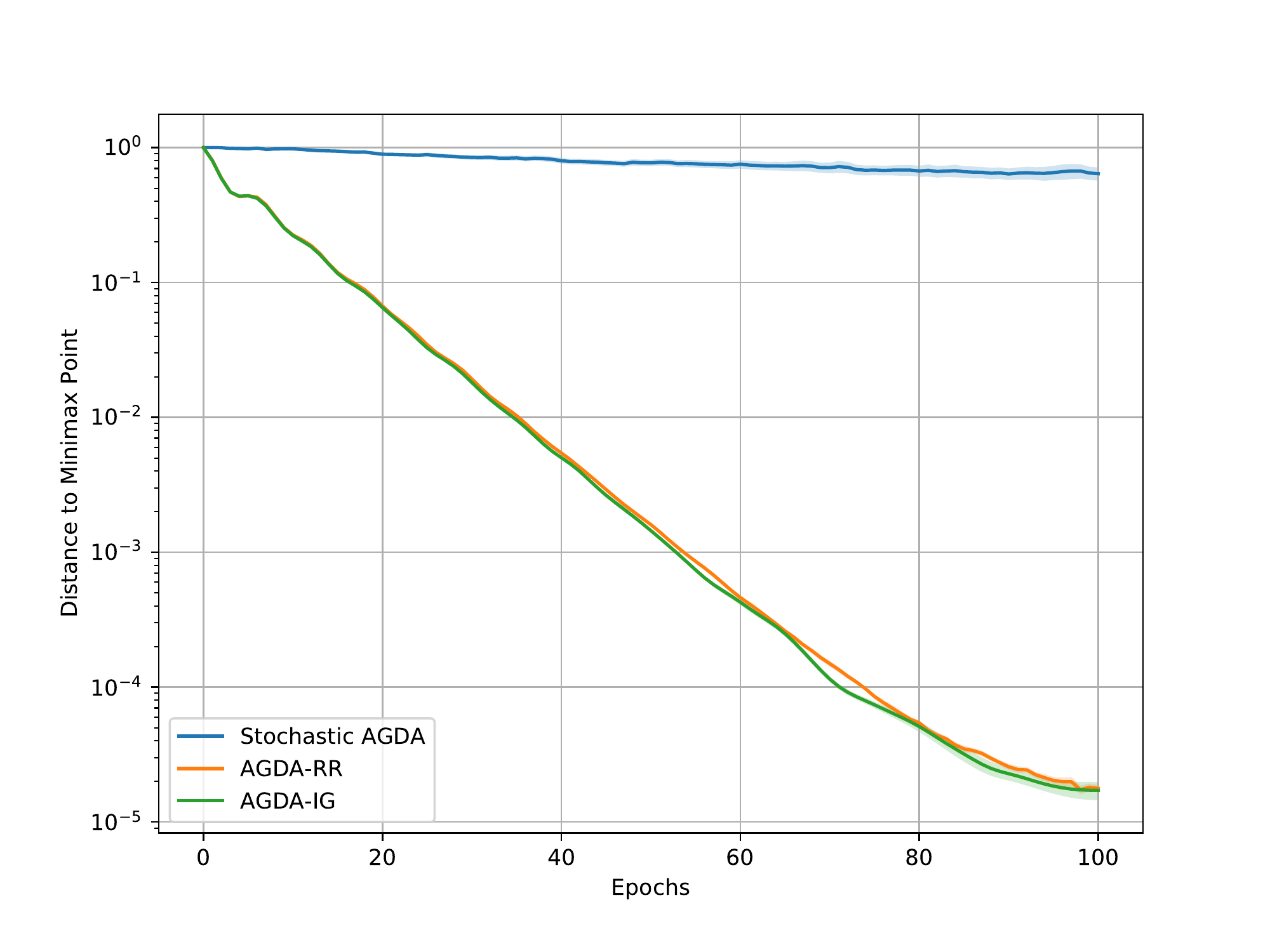}
         \caption{AGDA}
         \label{graph:agda}
     \end{subfigure}
        \caption{Convergence of GDA, PPM and AGDA averaged over 20 random instances. Shaded regions represent 95\% confidence intervals.}
        \label{fig:avg-quad-plots}
\end{figure}
\section{Conclusion}
\label{sec:conclusion}
We derived near optimal convergence rates for several without-replacement stochastic gradient algorithms for finite-sum minimax optimization, and demonstrated that they converge faster than algorithms that use uniform sampling. We considered two problem classes, strongly convex-strongly concave problems (generalized to unconstrained strongly monotone variational inequalities) and nonconvex-nonconcave problems with two-sided P\L{} objectives. We also formally defined \emph{adversarial shuffling}, where an attacker can control the order in which data points are supplied to the optimizer, and analyzed minimax optimization in this regime. Interesting future directions include the analysis of inexact proximal point methods, more general function classes, and time-varying step-sizes.

\section*{Acknowledgements and Disclosure of Funding}
Michael Muehlebach and Bernhard Sch\"{o}lkopf thank the German Research Foundation and the Branco Weiss Fellowship, administered by ETH Zurich, for the generous support.
\bibliography{references}

\clearpage
\section*{Checklist}

\begin{enumerate}

\item For all authors...
\begin{enumerate}
  \item Do the main claims made in the abstract and introduction accurately reflect the paper's contributions and scope?
    \answerYes{}
  \item Did you describe the limitations of your work?
    \answerYes{ Limitations such as structural assumptions on the function class and noise variance are clearly stated wherever appropriate.}
  \item Did you discuss any potential negative societal impacts of your work?
    \answerNo{ The work is purely of a theoretical nature and analyzes existing algorithms that are already widely prevalent.}
  \item Have you read the ethics review guidelines and ensured that your paper conforms to them?
    \answerYes{}
\end{enumerate}

\item If you are including theoretical results...
\begin{enumerate}
  \item Did you state the full set of assumptions of all theoretical results?
    \answerYes{ Assumptions are stated in Sections 2, 3 and 4.}
        \item Did you include complete proofs of all theoretical results?
    \answerYes{ A proof sketch of the main result is presented in Section 3.2 while full proofs are deferred to the appendix.}
\end{enumerate}

\item If you ran experiments...
\begin{enumerate}
  \item Did you include the code, data, and instructions needed to reproduce the main experimental results (either in the supplemental material or as a URL)?
    \answerYes{ They are included in the supplementary material}
  \item Did you specify all the training details (e.g., data splits, hyperparameters, how they were chosen)?
    \answerYes{.Key details are specified in Section 5 and further details are specified in the appendix.}
        \item Did you report error bars (e.g., with respect to the random seed after running experiments multiple times)?
    \answerYes{ All results are reported with 95\% confidence intervals.}
        \item Did you include the total amount of compute and the type of resources used (e.g., type of GPUs, internal cluster, or cloud provider)?
    \answerYes{.These details are provided in the appendix.}
\end{enumerate}

\item If you are using existing assets (e.g., code, data, models) or curating/releasing new assets...
\begin{enumerate}
  \item If your work uses existing assets, did you cite the creators?
    \answerNA{}
  \item Did you mention the license of the assets?
    \answerNA{}
  \item Did you include any new assets either in the supplemental material or as a URL?
    \answerNA{}
  \item Did you discuss whether and how consent was obtained from people whose data you're using/curating?
    \answerNA{}
  \item Did you discuss whether the data you are using/curating contains personally identifiable information or offensive content?
    \answerNA{}
\end{enumerate}

\item If you used crowdsourcing or conducted research with human subjects...
\begin{enumerate}
  \item Did you include the full text of instructions given to participants and screenshots, if applicable?
    \answerNA{}
  \item Did you describe any potential participant risks, with links to Institutional Review Board (IRB) approvals, if applicable?
    \answerNA{}
  \item Did you include the estimated hourly wage paid to participants and the total amount spent on participant compensation?
    \answerNA{}
\end{enumerate}

\end{enumerate}


\newpage
\appendix
\section*{Appendix}
\tableofcontents
\clearpage
\section{Notation}
\label{app-sec:notation}
In addition to the existing notation introduced in Section \ref{sec:notation}, we introduce the matrix product notation $\prod_{i=1}^{n} \vM_i = \vM_1 \vM_2 \ldots \vM_n$ where $\vM_1, \vM_2, \ldots \vM_n$ are square matrices of the same size. Since matrix multiplication is not commutative, the ordering of the indices is important. To this end, we denote the reverse ordered matrix product as follows:
\begin{equation*}
    \revprod_{i=1}^{n} \vM_i = \vM_n \vM_{n-1} \ldots \vM_1.
\end{equation*}
We highlight that the above product is not necessarily equal to $\prod_{i=1}^{n} \vM_i$ due to the noncommutativity of matrix products. Furthermore, we define the empty (reverse) matrix product (i.e. the matrix product or reverse matrix product to be indexed over empty sets) to be the identity matrix $\vI$, where the size of $\vI$ is clear from the context. Finally, we use $\sigma_{\max}(\vA)$ to denote the maximum singular value of the matrix $\vA$ (which is equal to its operator norm $\norm{\vA}$) and $\sigma_{\min}(\vA)$ to denote its minimum singular value.
\section{Useful Lemmas}
\label{app-sec:lemmas}
In this section, we present the proofs of some properties that are important for our convergence analysis. We begin with a lemma on the differentiability properties of Lipschitz and strongly monotone operators, which is repeatedly used in all our proofs. 
\begin{lemma}[Differentiability properties of Lipschitz operators]
\label{app-lem:lip-operator-property}
Let $\nu : \bR^d \rightarrow \bR^d$ be an $l$-Lipschitz continuous operator. Then, $\nu$ is Lebesgue almost everywhere differentiable, i.e., the Jacobian $\grad \nu(\vz)$ is well defined everywhere outside of a set of Lebesgue measure zero and satisfies $\norm{\grad \nu(\vz)} \leq l$, wherever it is defined. Additionally, if $\nu$ is $\mu$-strongly monotone, then $\nu$ has a unique root, and $\vv^T \grad \nu(\vz) \vv \geq \mu|\vv|{}^2$ for any $\vv \in \bR^d$ and any $\vz \in \bR^d$ such that $\grad \nu(\vz)$ is well defined.
\end{lemma}
\begin{proof}
The first statement of this lemma is a direct consequence of Rademacher's Theorem which states that any Lipschitz continuous function defined on $\bR^d$ is differentiable everywhere outside of a set of zero Lebesgue measure. The second statement is proved as follows:

Consider any $\vz \in \bR^d$ such that $\grad \nu(\vz)$ is well defined and let $\vv \in \bR^d$ be arbitrary. By the Lipschitz continuity of $\nu$, it follows that for any $t \in \bR$:
\begin{align*}
    \norm{\nu(\vz + t \vv) - \nu(\vz)} &\leq l \norm{t\vv}.
\end{align*}
Rearranging and taking limits results in,
\begin{align*}
    \lim_{t \to 0} \norm{\frac{\nu(\vz + t \vv) - \nu(\vz)}{t}} &\leq l \norm{\vv}.
\end{align*}
By the continuity of norms and the definition of the Jacobian, it follows that $\norm{\grad \nu(\vz) \vv} \leq l \norm{\vv} \ \forall \ \vv \in \bR^d$, which further implies that $\norm{\grad \nu(\vz)} \leq l$. This concludes the proof of the second statement. 

For the proof of the third statement, we consider the operator $\zeta(\vz) = \vz - \frac{\mu}{l^2}\nu(\vz)$ and observe that any fixed point of $\zeta$ is a root of $\nu$ and vice versa.

From the $l$-smoothness and $\mu$-strong monotonicity of $\nu$, we infer that for any $\vz_1, \vz_2 \in \bR^d$:
\begin{align*}
    \norm{\zeta(\vz_2) - \zeta(\vz_1)}^2 &= \norm{\vz_2 - \vz_1}^2 - \frac{2 \mu}{l^2} \dotprod{\nu(\vz_2) - \nu(\vz_1)}{\vz_2 - \vz_1} + \frac{\mu^2}{l^4}\norm{\nu(\vz_2) - \nu(\vz_1)}^2 \\
    &\leq (1 - \frac{ \mu^2}{l^2})\norm{\vz_2 - \vz_1}^2 < \norm{\vz_2 - \vz_1}^2.
\end{align*}
Thus, $\zeta$ is a contraction mapping, and has a unique fixed point $\vz^*$ by the Banach Fixed Point Theorem, which implies that $\vz^*$ is the unique root of $\nu$.
Finally, consider any $\vz \in \bR^d$ such that $\grad \nu(\vz)$ is well defined and let $\vv \in \bR^d$ be arbitrary. By the strong monotonicity of $\nu$, it follows that for any $t \in \bR$:
\begin{align*}
    \dotprod{\nu(\vz + t \vv) - \nu(\vz)}{t \vv} \geq \mu \norm{t \vv}^2.
\end{align*}
Rearranging and taking limits results in
\begin{align*}
    \dotprod{\vv}{\lim_{t \to 0} \frac{\nu(\vz + t \vv) - \nu(\vz)}{t}} \geq \mu \norm{\vv}^2.
\end{align*}
By definition of the Jacobian and the continuity of inner products, it follows that $\vv^T \grad \nu(\vz) \vv \geq \mu \norm{\vv}^2$.
\end{proof}
We now present a lemma on the properties of the gradient operator of a strongly convex-strongly concave function, which was earlier stated as Lemma \ref{lem:sc-sc-monotone-vi} in Section \ref{sec:sc-sc-setting}. This lemma establishes the equivalence between the minimax optimization problem for a smooth and strongly convex-strongly concave function $F$ and the root finding problem for its corresponding gradient operator $\nu$.
\begin{lemma}[Equivalence of minimax optimization and root finding]
\label{app-lem:minmax-root-find-equivalence}
 Let $F : \bR^{d_{\vx}} \times\bR^{d_{\vy}}  \rightarrow \bR$ be an $l$-smooth and $\mu$-strongly convex-strongly concave function and let $\nu(\vx, \vy) = [\grad_{\vx} F(\vx, \vy), -\grad_{\vy} F(\vx, \vy)]$. Then $\nu$ is $l$-Lipschitz continuous, $\mu$-strongly monotone and has a unique root $\vz^*$, which is also the unique global minimax point of $F$.
\end{lemma}
\begin{proof}
The $l$-Lipschitz continuity of $\nu$ follows by definition of the $l$-smoothness of $F$, as described in Assumption \ref{as:comp-smooth}. To establish $\mu$-strong monotonicity, we proceed as follows:

Consider any $\vz_1 = (\vx_1, \vy_1)$ and $\vz_2 = (\vx_2, \vy_2)$ in $\bR^d$. Since $F$ is $\mu$-strongly convex in $\vx$ and $\mu$-strongly concave in $\vy$, it follows that:
\begin{align*}
    F(\vx_2, \vy_1) - F(\vx_1, \vy_1) &\geq \dotprod{\grad_{\vx} F(\vx_1, \vy_1)}{\vx_2 - \vx_1} + \frac{\mu}{2}\norm{\vx_2 - \vx_1}^2, \\
    F(\vx_1, \vy_2) - F(\vx_2, \vy_2) &\geq \dotprod{\grad_{\vx} F(\vx_2, \vy_2)}{\vx_1 - \vx_2} + \frac{\mu}{2}\norm{\vx_2 - \vx_1}^2, \\
    F(\vx_2, \vy_2) - F(\vx_2, \vy_1) &\geq \dotprod{\grad_{\vy} F(\vx_2, \vy_2)}{\vy_2 - \vy_1} + \frac{\mu}{2}\norm{\vy_2 - \vy_1}^2, \\
    F(\vx_1, \vy_1) - F(\vx_1, \vy_2) &\geq \dotprod{\grad_{\vy} F(\vx_1, \vy_1)}{\vy_1 - \vy_2} + \frac{\mu}{2}\norm{\vy_2 - \vy_1}^2.
\end{align*}
Adding all four inequalities, and substituting the definition of $\nu$, we get,
\begin{align*}
    \dotprod{\nu(\vz_2) - \nu(\vz_1)}{\vz_2 - \vz_1} \geq \mu \norm{\vz_2 - \vz_1}^2.
\end{align*}

To prove the second part of this lemma, we define $\Phi(\vx) = \max_{\vy \in \bR^{d_{\vy}}} F(\vx, \vy)$. Furthermore, we note that $\Phi$ is strongly convex and by Danskin's Theorem, $\grad \Phi(\vx) = \grad_{\vx} F(\vx, \vy^*(\vx))$ where $\vy^*(\vx) \in \bR^{d_{\vy}}$ is the unique solution of $F(\vx, \vy^*(\vx)) = \max_{\vy \in \bR^{d_{\vy}}} F(\vx, \vy)$, with uniqueness being guaranteed by the strong concavity of $F$ in $\vy$. Since $\nu$ is $l$-Lipschitz continuous and $\mu$-strongly monotone, it is guaranteed to have a unique root $\vz^* = (\vx^*, \vy^*)$ by Lemma \ref{app-lem:lip-operator-property}, which implies that $\grad_{\vx} F(\vx^*, \vy^*) = \grad_{\vy} F(\vx^*, \vy^*) = 0$. Since $F$ is strongly concave in $\vy$, this further implies that $\vy^*$ is the unique maximizer of $F(\vx^*, \vy)$. Moreover, by Danskin's Theorem, $\grad \Phi(\vx^*) = \grad_{\vx} F(\vx^*, \vy^*) = 0$ which implies that $\vx^*$ is the unique minimizer of $\Phi$, since $\Phi$ is strongly convex. Thus, by definition of a global minimax point, we conclude that $\vz^*$ is the unique global minimax point of $F$.
\end{proof}
We conclude this section with the proof of a standard result in statistics on the variance of the population mean estimator under simple random sampling without replacement \cite{CochranSamplingWOR, RiceSamplingWOR}.  This property constitutes a central component of our convergence proofs and allows us to sharply bound the noise in the iterates of RR and SO. This lemma has also been used in \citet{MischenkoRR2020} to analyze RR/SO for minimization. 
\begin{lemma}[Properties of without-replacement sample means : Lemma 1 in \citet{MischenkoRR2020}]
\label{app-lem:sample-mean-wor}
Let $\tau$ be a uniformly sampled random permutation of $[n]$, and let $\vv_1, \vv_2, ..., \vv_n$ be arbitrary vectors that are \textbf{independent of $\tau$}, with sample mean $\vm = 1/n\sum_{i=1}^{n} \vv_i$ and sample variance $\sigma^2 = 1/n \sum_{i=1}^{n} \norm{\vv_i - \vm}^2$. For any $i \in [n]$, let the without-replacement sample average be defined as $\hat{\vm}^{(i)}_{\tau} = 1/i\sum_{j=1}^{i} \vv_{\tau(j)}$. Then, $\E{\hat{\vm}^{(i)}_{\tau}} = \vm$ and $\E{|\hat{\vm}^{(i)}_{\tau} - \vm|^2} = \frac{n-i}{n-1}\frac{\sigma^2}{i}$, where the expectation is taken over the uniform random permutation $\tau$.
\end{lemma}
\begin{proof}
Since $\vv_1, \ldots, \vv_n$ is independent of the uniformly sampled permutation $\tau$, it follows that for any $j \in [n]$,
\begin{align*}
    \bE[\vv_{\tau(j)}] = \sum_{k=1}^{n} \mathbb{P}_{\textrm{Uniform}}\left[\tau(j) = k\right] \vv_k = \sum_{k=1}^{n} \frac{(n-1)!}{n!} \vv_k = 1/n \sum_{k=1}^{n} \vv_k = \vm.
\end{align*}
Hence, by linearity of expectations,
\begin{align*}
    \bE[\hat{\vm}^{(i)}_{\tau}] = 1/i\sum_{j=1}^{i} \bE[\vv_{\tau(j)}] = \vm.
\end{align*}
Similarly, we can show that for any $j \in [n]$,
\begin{align*}
    \bE[\norm{\vv_{\tau(j)} - \vm}^2] = \sum_{k=1}^{n} \frac{(n-1)!}{n!} \norm{\vv_{k} - \vm}^2 = 1/n \sum_{k=1}^{n}\norm{\vv_{k} - \vm}^2 = \sigma^2.
\end{align*}
We now consider any $j_1, j_2 \in [n]$ such that $j_1 \neq j_2$ and infer that,
\begin{align*}
    \bE[\dotprod{\vv_{\tau(j_1)} - \vm}{\vv_{\tau(j_2)} - \vm}] &= \sum_{1 \leq k_1 \neq k_2 \leq n} \mathbb{P}_{\textrm{Uniform}}\left[\tau(j_1) = k_1, \tau(j_2) = k_2 \right] \dotprod{\vv_{k_1} - \vm}{\vv_{k_2} - \vm} \\
    &= \frac{(n-2)!}{n!}\sum_{k_1 = 1}^{n} \sum_{1 \leq k_2 \neq k_1 \leq n} \dotprod{\vv_{k_1} - \vm}{\vv_{k_2} - \vm} \\
    &= \frac{1}{n(n-1)}\sum_{k_1=1}^{n}\left[\left(\sum_{k_2=1}^{n} \dotprod{\vv_{k_1} - \vm}{\vv_{k_2} - \vm} \right) - \norm{\vv_{k_1} - \vm}^2 \right] \\
    &= \frac{1}{n(n-1)} \sum_{k_1 = 1}^{n}\dotprod{\vv_{k_1} - \vm}{\sum_{k_2=1}^{n} \vv_{k_2} - \vm} - \frac{1}{n(n-1)}\sum_{k_1=1}^{n}\norm{\vv_{k_1} - \vm}^2 \\
    &= -\frac{\sigma^2}{n-1}.
\end{align*}
From the above identities, we conclude that,
\begin{align*}
    \bE[|\hat{\vm}^{(i)}_{\tau} - \vm|^2] &= \bE[|1/i \sum_{j=1}^{i} \vv_{\tau(j)} - \vm|^2] \\
    &= \nicefrac{1}{i^2} \bE[|\sum_{j=1}^{i} \vv_{\tau(j)} - \vm|^2] \\
    &= \nicefrac{1}{i^2}\sum_{j=1}^{i}\sum_{k=1}^{i} \bE[\dotprod{\vv_{\tau(j)} - \vm}{\vv_{\tau(k)} - \vm}] \\
    &= \nicefrac{1}{i^2} \left( \sum_{j=1}^{i} \bE[\norm{\vv_{\tau(j)} - \vm}^2] + \sum_{1 \leq j \neq k \leq i}\bE[\dotprod{\vv_{\tau(j)} - \vm}{\vv_{\tau(k)} - \vm}] \right) \\
    &= \nicefrac{1}{i^2}[i \sigma^2 - \frac{i(i-1)}{n-1}\sigma^2] \\
    &= \frac{n-i}{n-1}\frac{\sigma^2}{i}.
\end{align*}
\end{proof}
\section{Analysis of GDA without Replacement}
\label{app-sec:gda}
In this section, we present the proofs of convergence for simultaneous Gradient Descent Ascent without replacement (i.e. GDA-RR/SO/AS) for solving the finite-sum strongly monotone root finding problem. Our proofs of the same rely on two key techniques: 1. exploiting the Lipschitz continuity of the components $\omega_i$ to \emph{linearize} the dynamics of GDA without replacement about $\vz^*$ (where $\vz^*$ denotes the unique root of $\nu$), 2. using Lemma \ref{app-lem:sample-mean-wor} to control the noise in the iterates of GDA without replacement. 

For ease of exposition, we first illustrate the linearization technique by presenting a linearization-based proof of the well known exponential convergence rate of full-batch GDA for solving the root finding problem under Lipschitz continuity and strong monotonicity assumptions. This serves as a prelude for our unified proof of convergence of GDA-RR and GDA-SO for solving the root finding problem for finite-sum strongly monotone operators. We conclude this section by presenting a proof of convergence of GDA-AS, which demonstrates how our techniques for analyzing RR and SO can be easily adapted to the adversarial shuffling setup.
\subsection{Analysis of Full-batch GDA by Linearization}
\label{app-sec:gda-batch}
\begin{theorem}[Convergence of full-batch GDA]
\label{app-thm:batch-gda-convergence}
Consider Problem \eqref{p:strong-monotone-vi} for the $l$-Lipschitz and $\mu$-strongly monotone operator $\nu : \bR^d \rightarrow \bR^d$ and let $\vz^*$ denote the unique root of $\nu$. For any step-size $\alpha > 0$, the iterates $\vz_k$ of full-batch GDA satisfy the following recurrence:
\begin{align*}
    \vz_{k+1} - \vz^* = (\vI - \alpha \vM_k) (\vz_{k} - \vz^*),
\end{align*}
where the spectral norm of $\vI - \alpha \vM_k$ is bounded as $\norm{\vI - \alpha \vM_k} \leq (1 - 2\alpha \mu + \alpha^2 l^2)^{1/2}$. Consequently, setting $\alpha = \mu/l^2$ gives us the following last iterate convergence guarantee:
\begin{align*}
    \norm{\vz_{K+1} - \vz^{*}}^2 \leq (1 - 1/\kappa^2)^K \norm{\vz_{0} - \vz^{*}}^2,
\end{align*}
where $\kappa = l/\mu$. Furthermore, the exponential convergence rate of $1 - 1/\kappa^2$ is optimal for GDA up to constant factors.
\end{theorem}
\begin{proof}
The iterates of full-batch GDA with step-size $\alpha$ have the following update rule:
\begin{align*}
    \vz_{k+1} = \vz_k - \alpha \nu(\vz_k).
\end{align*}
Interpreting the time evolution of $\vz_k$ as a discrete-time dynamical system, we analyze the convergence of GDA by applying \emph{linearization}, which is a standard technique in the analysis of discrete and continuous-time dynamical systems. To this end, we proceed as follows.

Let $g : [0, 1] \rightarrow \bR^d$ be defined as $g(t) = \nu(t \vz_k + (1 - t) \vz^*)$. We note that the Lipschitz continuity of $\nu$ implies the Lipschitz continuity of $g$. As a result, the Fundamental Theorem of Calculus for Lebesgue Integrals applies to $g$ (since Lipschitz continuous functions defined on a compact interval are also absolutely continuous). It follows that $g(1) = g(0) + \int_{0}^{1} g^\prime (t) \textrm{d}t$ which results in the following expansion:
\begin{align*}
    \nu(\vz_k) &= \nu(\vz^*) + \int_{0}^{1} \grad \nu(t \vz_k + (1 - t) \vz^*)(\vz_k - \vz^*) \textrm{d}t \\ 
    &= \vM_k(\vz_k - \vz^*),
\end{align*}
where $\nu(\vz^*) = 0$ and $\vM_k = \int_{0}^{1} \grad \nu(t \vz_k + (1 - t) \vz^*)\textrm{d}t$ is well defined because $\grad \nu(\vz)$ is defined Lebesgue-almost everywhere by Lemma \ref{app-lem:lip-operator-property}.

Substituting the above expansion into the update rule of GDA gives us the following:
\begin{equation}
\label{app-eqn:batch-gda-lds}
\vz_{k+1} - \vz^* = (\vI - \alpha \vM_k)(\vz_k - \vz^*).
\end{equation}
We now proceed to obtain a convergence guarantee by upper bounding the spectral norm of $\vI - \alpha \vM_k$. Note that by Lemma \ref{app-lem:lip-operator-property}, $\grad \nu(\vz)$ is well defined Lebesgue-almost everywhere, and wherever defined, satisfies $|\grad \nu(\vz)| \leq l$ and $\vv^{T} \grad \nu(\vz) \vv \geq \mu \norm{\vv}^2 \ \forall \ \vv \in \bR^d$. Using standard properties of Lebesgue integrals and inner products, we conclude that $|\vM_k| \leq l$ and $\vv^T \vM_k \vv \geq \mu \norm{\vv}^2 \ \forall \ \vv \in \bR^d$. Hence, for any $\vv \in \bR^d$:
\begin{align*}
    \norm{(\vI - \alpha \vM_k) \vv}^2 &= \norm{\vv}^2 - \alpha\vv^T \vM_k \vv - \alpha \vv^T \vM^T_k \vv + \norm{\vM_k \vv}^2 \\
    &\leq (1 - 2 \alpha \mu + \alpha^2 l^2) \norm{\vv}^2.
\end{align*}
Hence, $\norm{\vI - \alpha \vM_k} \leq (1 - 2\alpha \mu + \alpha^2 l^2)^{1/2}$ for any $\alpha > 0$. It follows that,
\begin{align*}
    \norm{\vz_{k+1} - \vz^{*}}^2 \leq (1 - 2\alpha \mu + \alpha^2 l^2)  \norm{\vz_{k} - \vz^{*}}^2.
\end{align*}
Setting $\alpha = \mu/l^2$, we obtain the following convergence guarantee:
\begin{align*}
    \norm{\vz_{K+1} - \vz^{*}}^2 \leq (1 - 1/\kappa^2)^K \norm{\vz_{0} - \vz^{*}}^2.
\end{align*}
We now establish the optimality of the $1 - 1/\kappa^2$ exponential convergence rate by presenting a lower bound construction on a two-dimensional quadratic problem. 

Consider any positive constants $l, \mu$ such that $l > \mu$ and the quadratic $f : \bR^2 \rightarrow \bR$ given by $f(x, y) = \frac{\mu}{2}(x^2 - y^2) + (l - \mu)xy$. It is easy to see that $f$ is $l$-smooth and $\mu$-strongly convex strongly concave. Thus, by Lemma \ref{app-lem:minmax-root-find-equivalence}, we infer that $\nu(\vz) = \nu(x, y) = [\grad_x f(x, y), -\grad_y f(x,y)]$ is $l$-smooth and $\mu$-strongly monotone. Furthermore, we also note that  $\nu(\vz) = \vM \vz$ where $\vM$ is given by:
\begin{align*}
    \vM = \begin{bmatrix}
\mu & l - \mu \\
\mu - l & \mu
\end{bmatrix}.
\end{align*}
Thus, it is easy to see that $\vz^* = 0$ is the unique root of $\nu$, where uniqueness is guaranteed by Lemma \ref{app-lem:lip-operator-property}. Moreover, for any step-size $\alpha > 0$, the iterates of GDA on $f$ are given by $\vz_{k+1} = (\vI - \alpha \vM) \vz_k$. Thus, by taking norms on both sides, we conclude that:
\begin{align*}
    \norm{\vz_{k+1}}^2 &= \norm{(\vI - \alpha \vM)\vz_k}^2 = [1 - 2\alpha \mu  + \alpha^2 [\mu^2 + (l-\mu)^2]]\norm{\vz_k}^2 \\
    &\geq (1 - \frac{\mu^2}{\mu^2 + (l - \mu)^2})\norm{\vz_k}^2 =  (1 - \frac{1}{1 + (\kappa - 1)^2})\norm{\vz_k}^2,
\end{align*}
where the inequality is obtained by minimizing the quadratic $1 - 2\alpha \mu  + \alpha^2 [\mu^2 + (l-\mu)^2]$ with respect to $\alpha$.

We note that by Young's inequality, $\kappa^2 = [1 + (\kappa - 1)]^2 \leq 2[1 + (\kappa - 1)^2]$ and thus, 
\begin{align*}
    1 - \frac{1}{1 + (\kappa - 1)^2} \geq 1 - \frac{2}{\kappa^2},
\end{align*}
which leads us to conclude that for any $\alpha > 0$, the iterates $\vz_k$ of GDA on $\nu$ satisfy the following:
\begin{align*}
    \norm{\vz_{k+1} - \vz^*}^2 \geq (1 - 2/\kappa^2) \norm{\vz_{k} - \vz^*}^2.
\end{align*}
This establishes that the convergence rate of $1 - 1/\kappa^2$ is optimal for GDA up to constant factors. We contrast this with the smooth and strongly convex minimization setting where gradient descent enjoys a faster convergence rate of $1 - 1/\kappa$.
\end{proof}
\subsection{A Unified Analysis of GDA-RR and GDA-SO}
\label{app-sec:gda-wor}
\begin{theorem}[Convergence of GDA-RR/SO]
\label{app-thm:gda-wor-convergence}
Consider Problem \eqref{p:strong-monotone-vi} for the $\mu$-strongly monotone operator $\nu(\vz) = \nicefrac{1}{n} \sum_{i=1}^{n} \omega_i(\vz)$ where each $\omega_i$ is $l$-Lipschitz, but not necessarily monotone. Let $\vz^*$ denote the unique root of $\nu$. Then, for any $\alpha \leq \frac{\mu}{5nl^2}$ and $K \geq 1$, the iterates of GDA-RR and GDA-SO satisfy the following:
\begin{equation*}
    \E{\norm{\vz^{K+1}_0 - \vz^*}^2} \leq 2e^{-n\alpha \mu K}\norm{\vz_0 - \vz^*}^2 + \frac{l^2 \sigma_*^2 \alpha^3 n^2 K}{ \mu}.
\end{equation*}
Setting $\alpha = \min \{  \mu/5nl^2, 2 \log (\norm{\nu(\vz_0 )}n^{1/2}K/\mu )/\mu n K \}$ results in the following expected last-iterate convergence guarantee for both GDA-RR and GDA-SO which holds for any $K \geq 1$:

\begin{align*}
\E{|\vz^{K+1}_0 - \vz^*|^2} &\leq  2e^{\nicefrac{-K}{5\kappa^2}}|\vz_0 - \vz^*|^2 + \frac{2\mu^2 + 8 \kappa^2 \sigma^2_* \log^3 (|\nu(\vz_0)|n^{1/2}K/\mu)}{\mu^2 nK^2} \\
&= \Tilde{O}(e^{\nicefrac{-K}{5\kappa^2}} + \nicefrac{1}{nK^2}).
\end{align*}
\end{theorem}

\begin{proof}
Without loss of generality, we express the iterate-level update rule of GDA-RR and GDA-SO jointly as
\begin{equation}
\label{app-eqn:gda-wor-itr}
    \vz^{k}_{i} = \vz^{k}_{i-1} - \alpha \vw_{\tau_k(i)}(\vz^k_{i-1}),
\end{equation}
where $k \in [K]$ and $i \in [n]$. We note that for GDA-RR, $\tau_k \sim \textrm{Uniform}(\mathbb{S}_n)$ for every $k \in [K]$ (i.e. $\tau_k$ is a permutation of $[n]$ that is resampled uniformly in every epoch), whereas for GDA-SO, $\tau_k = \tau \ \forall k \in [K]$ where $\tau \sim \textrm{Uniform}(\mathbb{S}_n)$ (i.e. the permutation $\tau$ is uniformly sampled before the first epoch, then reused for all subsequent epochs). 

As discussed in the proof sketch presented in Section \ref{sec:rr-so}, our proof begins by adapting the linearization technique illustrated in Theorem \ref{app-thm:batch-gda-convergence} to GDA without replacement in order to derive a linearized epoch-level update rule of the form $\vz^{k+1}_0 - \vz^* = \vH_k (\vz^{k}_0 - \vz^*) + \alpha^2 \vr_k$. We then use the Lipschitz continuity and strong monotonicity properties of $\nu$ to upper bound $\norm{\vH_k}$, and use Lemma \ref{app-lem:sample-mean-wor} to upper bound $\bE[|\vr_k|^2]$. Equipped with these bounds, the proof is concluded by unrolling the epoch-level update rule for $K$ epochs and carefully choosing $\alpha$ as stated above.

\paragraph{Step 1: Linearized epoch-level update rule}
Our approach for this step is inspired by the following insight developed by prior works that have analyzed sampling without replacement for minimization: for small enough step-sizes, the epoch iterates $\vz^k_0$ of gradient descent without replacement closely track the iterates of full-batch gradient descent. To this end, we treat GDA without replacement as a noisy version of full-batch GDA by focusing on the time evolution of the epoch iterates $\vz^k_0$. At the same time, we wish to adapt the proof technique of Theorem \ref{app-thm:batch-gda-convergence}, which hinges on linearization about $\vz^*$, to the analysis of GDA without replacement. This motivates us to perform the following decomposition of the stochastic gradients, where we linearize $\omega_{\tau_k(i)}(\vz^k_{i-1})$ about $\vz^k_0$, and subsequently linearize $\omega_{\tau_k(i)}(\vz^k_0)$ about $\vz^*$:
\begin{align*}
    \vw_{\tau_k(i)}(\vz^k_{i-1}) &= \vw_{\tau_k(i)}(\vz^{*}) + [\vw_{\tau_k(i)}(\vz^k_0) - \vw_{\tau_k(i)}(\vz^*)] + [\vw_{\tau_k(i)}(\vz^k_{i-1}) - \vw_{\tau_k(i)}(\vz^k_0)] \\
    &= \vw_{\tau_k(i)}(\vz^{*}) + \int_{0}^{1} \grad \omega_{\tau_k(i)}(t\vz^k_0 + (1 - t) \vz^*)(\vz^k_0 - \vz^*) \textrm{d}t \\
    &+ \int_{0}^{1} \grad \omega_{\tau_k(i)}(t\vz^k_{i-1} + (1 - t) \vz^k_0)(\vz^k_{i-1} - \vz^k_0) \textrm{d}t. 
\end{align*}
We define the matrices $\vM_{\tau_k(i)}$ and $\vJ_{\tau_k(i)}$ as follows:
\begin{align*}
    \vM_{\tau_k(i)} &= \int_{0}^{1} \grad \omega_{\tau_k(i)}(t\vz^k_0 + (1 - t) \vz^*) \textrm{d}t, \\
    \vJ_{\tau_k(i)} &= \int_{0}^{1} \grad \omega_{\tau_k(i)}(t\vz^k_{i-1} + (1 - t) \vz^k_0) \textrm{d}t.
\end{align*}
Repeating the same argument as in Theorem \ref{app-thm:batch-gda-convergence}, we use Lemma \ref{app-lem:lip-operator-property} and the $l$-Lipschitz continuity of $\omega_i(\vz) \ \forall i \in [n]$ to conclude that $\vM_{\tau_k(i)}$ and $\vJ_{\tau_k(i)}$ are well defined and bounded with  $\norm{\vM_{\tau_k(i)}} \leq l$ and $\norm{\vJ_{\tau_k(i)}} \leq l$. It follows that
\begin{equation}
\label{app-eqn:gda-wor-grads}
    \vw_{\tau_k(i)}(\vz^k_{i-1}) = \vw_{\tau_k(i)}(\vz^{*}) + \vM_{\tau_k(i)}(\vz^k_0 - \vz^*) + \vJ_{\tau_k(i)}(\vz^k_{i-1} - \vz^k_0).
\end{equation}
Substituting \eqref{app-eqn:gda-wor-grads} into the update equation \eqref{app-eqn:gda-wor-itr} for $\vz^k_1$ gives us the following:
\begin{align*}
    \vz^k_1 - \vz^* &=
    (\vI - \alpha \vM_{\tau_k(1)})(\vz^k_0 - \vz^*) - \alpha \vJ_{\sigma_k(1)}(\vz^k_0 - \vz^k_0) -  \alpha \vw_{\tau_k(1)}(\vz^*) \\
    &= (\vI - \alpha \vM_{\tau_k(1)})(\vz^k_0 - \vz^*) - \alpha \vw_{\tau_k(1)}(\vz^*).
\end{align*}
Repeating the same for $\vz^k_2$ yields,
\begin{align*}
    \vz^k_2 - \vz^* &= \vz^k_1 - \vz^* - \alpha \vw_{\tau_k(2)}(\vz^*) - \alpha \vM_{\tau_k(2)}(\vz^k_0 - \vz^*) - \alpha \vJ_{\tau_k(2)}(\vz^k_1 - \vz^k_0) \\
    &= (\vI - \alpha \vJ_{\tau_k(2)})(\vz^k_1 - \vz^*) - \alpha (\vM_{\tau_k(2)} - \vJ_{\tau_k(2)})(\vz^k_0 - \vz^*) - \alpha \vw_{\tau_k(2)}(\vz^*) \\
    &= (\vI - \alpha \vJ_{\tau_k(2)})[(\vI - \alpha \vM_{\tau_k(1)})(\vz^k_0 - \vz^*) - \alpha \vw_{\tau_k(1)}(\vz^*)] \\
    &- \alpha (\vM_{\tau_k(2)} - \vJ_{\tau_k(2)})(\vz^k_0 - \vz^*) - \alpha \vw_{\tau_k(2)}(\vz^*) \\ 
    &= [(\vI - \alpha \vJ_{\tau_k(2)})(\vI - \alpha \vM_{\tau_k(1)}) - \alpha (\vM_{\tau_k(2)} - \vJ_{\tau_k(2)})](\vz^k_0 - \vz^*) \\
    & - \alpha [\vw_{\tau_k(2)}(\vz^*) + (\vI - \alpha \vJ_{\tau_k(2)})\vw_{\tau_k(1)}(\vz^*)] \\
    &= [\vI - \alpha \vM_{\tau_k(2)} - \alpha(\vI - \alpha \vJ_{\tau_k(2)})\vM_{\tau_k(1)}](\vz^k_0 - \vz^*) \\
    &- \alpha [\vw_{\tau_k(2)}(\vz^*) + (\vI - \alpha \vJ_{\tau_k(2)})\vw_{\tau_k(1)}(\vz^*)].
\end{align*}
Applying the same process to subsequent iterates, the update rule for $\vz^k_i$ can then be obtained recursively. We propose that, for any $i \in [n]$, 
\begin{align}
\label{app-eqn:gda-wor-step-induction}
\vz^k_i - \vz^* &= [\vI - \alpha \sum_{j=1}^{i} \left(\revprod_{t=j+1}^{i}(\vI - \alpha \vJ_{\tau_k(t)})\right)\vM_{\tau_k(j)}](\vz^k_0 - \vz^*) \\ \nonumber
&- \alpha \sum_{j=1}^{i}\left(\revprod_{t=j+1}^{i}(\vI - \alpha \vJ_{\tau_k(t)})\right)\vw_{\tau_k(j)}(\vz^*).
\end{align}
We clarify that the matrix products in \eqref{app-eqn:gda-wor-step-induction} are in reverse order, and hence reduce to the empty product, which is defined to be $\vI$, when $j=i$.

As verified above, \eqref{app-eqn:gda-wor-step-induction} is satisfied for $i=1,2$. We now prove that it holds for any $i \in [n]$ by induction.

Assume \eqref{app-eqn:gda-wor-step-induction} is true for some $i \in [n]$. Using \eqref{app-eqn:gda-wor-grads} to derive the update of $\vz^k_{i+1}$ yields,
\begin{align*}
    \vz^k_{i+1} - \vz^* &= \vz^k_i - \vz^* - \alpha \vw_{\tau_k(i+1)}(\vz^*) - \alpha \vM_{\tau_k(i+1)}(\vz^k_0 - \vz^*) - \alpha \vJ_{\tau_k(i+1)}(\vz^k_i - \vz^k_0) \\
    &= (\vI - \alpha \vJ_{\tau_k(i+1)})(\vz^k_i - \vz^*) - \alpha (\vM_{\tau_k(i+1)} - \vJ_{\tau_k(i+1)})(\vz^k_0 - \vz^*) - \alpha \vw_{\tau_k(i+1)}(\vz^*) \\
    &= (\vI - \alpha \vJ_{\tau_k(i+1)})[\vI - \alpha \sum_{j=1}^{i} \left(\revprod_{t=j+1}^{i}(\vI - \alpha \vJ_{\tau_k(t)})\right)\vM_{\tau_k(j)}](\vz^k_0 - \vz^*) \\
    &- \alpha [\omega_{\tau_k(i+1)}(\vz^*) + \sum_{j=1}^{i}(I - \alpha \vJ_{\tau_k(i+1)})\left(\revprod_{t=j+1}^{i}(\vI - \alpha \vJ_{\tau_k(t)})\right)\vw_{\tau_k(j)}(\vz^*)] \\
    &- \alpha (\vM_{\tau_k(i+1)} - \vJ_{\tau_k(i+1)})(\vz^k_0 - \vz^*) \\
    &= [\vI - \alpha \vM_{\tau_k(i+1)} - \alpha \sum_{j=1}^{i}(\vI - \alpha \vJ_{\tau_k(i+1)}) \left(\revprod_{t=j+1}^{i}(\vI - \alpha \vJ_{\tau_k(t)})\right)\vM_{\tau_k(j)}](\vz^k_0 - \vz^*) \\
    &- \alpha \sum_{j=1}^{i+1}\left(\revprod_{t=j+1}^{i+1}(\vI - \alpha \vJ_{\tau_k(t)})\right)\vw_{\tau_k(j)}(\vz^*).
\end{align*}
It follows that,
\begin{align*}
\vz^{k}_{i+1} - \vz^* &= [\vI - \alpha \sum_{j=1}^{i+1} \left(\revprod_{t=j+1}^{i+1}(\vI - \alpha \vJ_{\tau_k(t)})\right)\vM_{\tau_k(j)}](\vz^k_0 - \vz^*) \\
&- \alpha \sum_{j=1}^{i+1}\left(\revprod_{t=j+1}^{i+1}(\vI - \alpha \vJ_{\tau_k(t)})\right)\vw_{\tau_k(j)}(\vz^*),
\end{align*}
which is simply \eqref{app-eqn:gda-wor-step-induction} for the iterate $\vz^{k}_{i+1}$. Hence, by induction, \eqref{app-eqn:gda-wor-step-induction} holds for any $i \in [n]$.  As before, the reverse matrix products for the summand $j=i+1$ is empty (in both the sums) and hence, is equal to $\vI$. 

Since \eqref{app-eqn:gda-wor-step-induction} holds for any $i \in [n]$, it holds for $\vz^{k+1}_0 = \vz^{k}_n$ which gives us the following epoch-level update rule
\begin{align}
    \vz^{k+1}_0 - \vz^* &= [\vI - \alpha \sum_{j=1}^{n} \left(\revprod_{t=j+1}^{n}(\vI - \alpha \vJ_{\tau_k(t)})\right)\vM_{\tau_k(j)}](\vz^k_0 - \vz^*) \nonumber\\ 
    &- \alpha \sum_{j=1}^{n}\left(\revprod_{t=j+1}^{n}(\vI - \alpha \vJ_{\tau_k(t)})\right)\vw_{\tau_k(j)}(\vz^*).
\label{app-eqn:gda-wor-incomplete-epoch}
\end{align}
We simplify the second term in the right hand side of \eqref{app-eqn:gda-wor-incomplete-epoch} using the summation by parts identity. To this end, we define $a_j$ and $b_j$ as
\begin{align*}
    a_j &= \revprod_{t=j+1}^{n}(\vI - \alpha \vJ_{\tau_k(t)}), \\
    b_j &= \vw_{\tau_k(j)}(\vz^*),
\end{align*}
and also note that,
\begin{align*}
    \sum_{j=1}^{n}b_j = \sum_{j=1}^{n} \vw_{\tau_k(j)}(\vz^*) = n \nu(\vz^{*}) = 0,
\end{align*}
since $\vz^*$ is the unique root of $\nu$. We now apply the summation by parts identity as follows:
\begin{align*}
    \sum_{j=1}^{n} a_j b_j &= a_n \sum_{j=1}^{n}b_j - \sum_{i=1}^{n-1} (a_{i+1} - a_i)\sum_{j=1}^{i}b_j \\
    &= - \sum_{i=1}^{n-1} \left[\revprod_{t=i+2}^{n}(\vI - \alpha \vJ_{\tau_k(t)}) - \revprod_{t=i+1}^{n}(\vI - \alpha \vJ_{\tau_k(t)}) \right]\sum_{j=1}^{i}\vw_{\tau_k(j)}(\vz^*) \\
    &= - \alpha \sum_{i=1}^{n-1}\left[\revprod_{t=i+2}^{n}(\vI - \alpha \vJ_{\tau_k(t)}) \right]\vJ_{\tau_k(i+1)}\sum_{j=1}^{i}\vw_{\tau_k(j)}(\vz^*).
\end{align*}
Hence, we conclude that
\begin{align*}
    \sum_{j=1}^{n}\left[\revprod_{t=j+1}^{n}(\vI - \alpha \vJ_{\tau_k(t)})\right]\vw_{\tau_k(j)}(\vz^*) = -\alpha \sum_{i=1}^{n-1}\left[\revprod_{t=i+2}^{n}(\vI - \alpha \vJ_{\tau_k(t)}) \right]\vJ_{\tau_k(i+1)}\sum_{j=1}^{i}\vw_{\tau_k(j)}(\vz^*),
\end{align*}
and introduce $\vH_k$ and $\vr_k$ as
\begin{align*}
    \vH_k &= \vI - \alpha \sum_{j=1}^{n} \left(\revprod_{t=j+1}^{n}(\vI - \alpha \vJ_{\tau_k(t)})\right)\vM_{\tau_k(j)}, \\
    \vr_k &= \sum_{i=1}^{n-1}\left[\revprod_{t=i+2}^{n}(\vI - \alpha \vJ_{\tau_k(t)}) \right]\vJ_{\tau_k(i+1)}\sum_{j=1}^{i}\vw_{\tau_k(j)}(\vz^*).
\end{align*}
By substituting the above expressions in \eqref{app-eqn:gda-wor-incomplete-epoch}, we obtain the desired epoch-level update rule for GDA-RR and GDA-SO:
\begin{equation}
\label{app-eqn:gda-wor-epoch-update}
    \vz^{k+1}_0 - \vz^* = \vH_k (\vz^k_0 - \vz^*) + \alpha^2 \vr_k.
\end{equation}
We observe that the linearized epoch-level update rule \eqref{app-eqn:gda-wor-epoch-update} closely resembles the linearized update rule \eqref{app-eqn:batch-gda-lds} of full-batch GDA as derived in Theorem \ref{app-thm:batch-gda-convergence} with an additional noise term $\alpha^2 \vr_k$. In fact, for $n=1$, which corresponds to the full-batch regime, it is easy to see that $\vr_k = 0$ and \eqref{app-eqn:gda-wor-epoch-update} exactly recovers the linearized full-batch GDA update rule \eqref{app-thm:batch-gda-convergence}. As we shall see in Theorem \ref{app-thm:gda-as-convergence}, the epoch iterates of GDA-AS also follow the exact same update. The derivation of such an epoch-level update rule which simultaneously handles GDA-RR, GDA-SO and GDA-AS is a cornerstone of our unified analysis. Intuition behind this strategy dates back to the foundational works of \citet{NedicBertsekas} and \citet{GurbuzRR2019}, where the authors interpret GD with RR as a stochastic approximation of full-batch GDA, and then obtain asymptotic $O(1/K^2)$ rates using Chung's Lemma. Our linearization strategy (which we also use for analyzing PPM and AGDA) essentially formalizes those insights for minimax optimization in a manner that enables the derivation of non-asymptotic rates.

Similar to the procedure followed in Theorem \ref{app-thm:batch-gda-convergence}, we now upper bound $\norm{\vH_k}$ as follows:
\paragraph{Step 2: Upper bounding $\norm{\vH_k}$} We define the matrix $\vM$ as,
\begin{align*}
    \vM &= 1/n\sum_{j=1}^{n}\vM_{\tau_k(j)} = 1/n \sum_{j=1}^{n} \int_{0}^{1} \grad \omega_{\tau_k(j)}(t\vz^k_0 + (1-t)\vz^*)\textrm{d}t = \int_{0}^{1} \grad \nu(t\vz^k_0 + (1-t)\vz^*)\textrm{d}t.
\end{align*}
Following the same arguments as in Theorem \ref{app-thm:batch-gda-convergence}, we note that the $l$-smoothness and $\mu$-strong monotonicity of $\nu$ implies $\norm{\vM} \leq l$ and $\vv^T \vM \vv \geq \mu \norm{\vv}^2 \ \forall \vv \in \bR^d$.  

We now simplify $\vH_k$ by unrolling the sum, beginning with the last index $j=n$:
\begin{align*}
    \vH_k &= \vI - \alpha \sum_{j=1}^{n} \left(\revprod_{t=j+1}^{n}(\vI - \alpha \vJ_{\tau_k(t)})\right)\vM_{\tau_k(j)} \\
    &= \vI - \alpha \vM_{\tau_k(n)} - \alpha (\vI - \alpha \vJ_{\tau_k(n)})\vM_{\tau_k(n-1)} - \alpha (\vI - \alpha \vJ_{\tau_k(n)}) (\vI - \alpha \vJ_{\tau_k(n-1)})\vM_{\tau_k(n-2)} \\
    &- \ldots -\alpha \left(\revprod_{t=2}^{n}(\vI - \alpha \vJ_{\tau_k(t)}) \right)\vM_{\tau_k(1)}.
\end{align*}
Expanding each product and grouping by powers of $\alpha$, we obtain the following:
\begin{align*}
    \vH_k &= \vI - \alpha \sum_{t_1 = 1}^{n} \vM_{\tau_k(t_1)} + \alpha^2 \sum_{1 \leq t_1 < t_2 \leq n} \vJ_{\tau_k(t_2)} \vM_{\tau_k(t_1)} - \alpha^3 \! \! \sum_{1 \leq t_1 < t_2 < t_3 \leq n} \! \! \vJ_{\tau_k(t_3)}\vJ_{\tau_k(t_2)} \vM_{\tau_k(t_1)} \\
    &+ \ldots + (-1)^n \alpha^n \sum_{1 \leq t_1 < \ldots < t_n \leq n} \vJ_{\tau_k(t_n)} \vJ_{\tau_k(t_{n-1})} \ldots \vJ_{\tau_k(t_2)}\vM_{\tau_k(t_1)}.
\end{align*}
We note that by definition of $\vM$, $\sum_{t_1 = 1}^{n}\vM_{\tau_k(t_1)} = n \vM$. Regrouping the remaining terms, we obtain:
\begin{equation*}
    \vH_k = \vI - n\alpha \vM + \sum_{j=2}^{n}(-1)^j\alpha^j\sum_{1 \leq t_1 < t_2 < \ldots < t_j \leq n} \vJ_{\tau_k(t_j)} \vJ_{\tau_k(t_{j-1})} \ldots \vJ_{\tau_k(t_2)}\vM_{\tau_k(t_1)}.
\end{equation*}
Hence, by the triangle inequality, it can be inferred that
\begin{equation}
\label{app-eqn:gda-wor-incomplete-spec-bound}
    \norm{\vH_k} \leq \norm{\vI - n \alpha \vM} + \sum_{j=2}^{n}\alpha^j\sum_{1 \leq t_1 < t_2 < \ldots < t_j \leq n} \norm{\vJ_{\tau_k(t_j)}} \norm{\vJ_{\tau_k(t_{j-1})}} \ldots \norm{\vJ_{\tau_k(t_2)}}\norm{\vM_{\tau_k(t_1)}}.
\end{equation}
We begin by bounding $\norm{\vI - n \alpha \vM}$ in a procedure similar to that of Theorem \ref{app-thm:batch-gda-convergence}. Observe that for any $\vv \in \bR^d$
\begin{align*}
    \norm{(\vI - n\alpha \vM) \vv}^2 &= \norm{\vv}^2 - n\alpha\vv^T \vM \vv - n\alpha \vv^T \vM \vv + n^2\norm{\vM \vv}^2 \\
    &\leq (1 - 2 n \alpha \mu + n^2 \alpha^2 l^2) \norm{\vv}^2,
\end{align*}
where the last inequality follows from $\norm{\vM \vv} \leq \norm{\vM}\norm{\vv} \leq l\norm{\vv}$ and $\vv^T \vM^T \vv = \vv^T \vM \vv \geq \mu \norm{\vv}^2$. As the above holds true for any $\vv \in \bR^d$, it follows that
\begin{align*}
    \norm{\vI - n\alpha \vM} &\leq \sqrt{1 - 2n \alpha \mu + n^2 \alpha^2 l^2} 
    \leq \sqrt{1 - \frac{9n\alpha\mu}{5}}
    \leq 1-\frac{9n\alpha\mu}{10},
\end{align*}
where we use $\alpha \leq \frac{\mu}{5nl^2}$ in the second and last inequalities.

It now remains to bound the right hand side of \eqref{app-eqn:gda-wor-incomplete-spec-bound}. Note that the constraint $1 \leq t_1 < t_2 < \ldots < t_j \leq n$ implies that there exist only $\binom{n}{j}$ possible choices for the $t_i$'s. Furthermore, as stated in Step 1, the $l$-smoothness of the components $\vw_i(\vz)$ implies that $\norm{\vM_{\tau_k(i)}}, \norm{\vJ_{\tau_k(i)}} \leq l, \forall i \in [n]$. Hence,
\begin{align*}
    \sum_{j=2}^{n}\alpha^j\sum_{1 \leq t_1 < t_2 < \ldots < t_j \leq n} \norm{\vJ_{\tau_k(t_j)}} \norm{\vJ_{\tau_k(t_{j-1})}} \ldots \norm{\vJ_{\tau_k(t_2)}}\norm{\vM_{\tau_k(t_1)}} &\leq \sum_{j=2}^{n}\alpha^j \binom{n}{j}l^j.
\end{align*}
Combining the previous inequalities allows us to bound $\norm{\vH_k}$ as follows: 
\begin{align*}
    \norm{\vH_k} &\leq 1 - \frac{9n\alpha\mu}{10} + \sum_{j=2}^{n}(\alpha l)^j \binom{n}{j} \leq 1 - \frac{9n\alpha\mu}{10} + \sum_{j=2}^{n} (\alpha n l)^j \leq 1 - \frac{9n\alpha\mu}{10} + \frac{\alpha^2 n^2 l^2}{1 - \alpha n l} \\
    &\leq 1 - \frac{9n\alpha\mu}{10} + \frac{5\alpha^2 n^2 l^2}{4} \leq 1 - \frac{9n\alpha  \mu}{10} + \frac{n\alpha \mu}{4} \leq 1 - n\alpha \mu/2,
\end{align*}
where we substitute $\alpha \leq \frac{\mu}{5nl^2} \leq 1/5$ wherever required. 

We highlight that our analysis so far holds for any permutation $\tau_k \in \mathbb{S}_n$. Consequently, Steps 1 and 2 directly generalize to the adversarial shuffling regime without any modifications. We use this as a starting point for our analysis of GDA-AS in Theorem \ref{app-thm:gda-as-convergence}. 

We now proceed to control the magnitude of the noise term $\vr_k$ by bounding its expected squared norm $\E{\norm{\vr_k}^2}$, where the expectation is taken over the single uniform permutation $\tau$ for GDA-SO and over independent uniform permutations $\tau_1, \tau_2, \ldots, \tau_K$ for GDA-RR. 

\paragraph{Step 3: Upper bounding $\bE[\norm{\vr_k}^2]$}
Recall that, as per Step 1, $\norm{\vJ_{\tau_k(t)}} \leq l$. Hence, applying the triangle inequality yields,
\begin{align*}
    \norm{\vr_k} &\leq \sum_{i=1}^{n-1}\left[\revprod_{t=i+2}^{n}\norm{\vI - \alpha \vJ_{\tau_k(t)}} \right]\norm{\vJ_{\tau_k(i+1)}}|\sum_{j=1}^{i}\vw_{\tau_k(j)}(\vz^*)| \\
    &\leq l(1 + \alpha l)^{n}\sum_{i=1}^{n-1}|\sum_{j=1}^{i}\vw_{\tau_k(j)}(\vz^*)| \\
    &\leq l(1 + \frac{\mu}{5nl})^{n}\sum_{i=1}^{n-1}|\sum_{j=1}^{i}\vw_{\tau_k(j)}(\vz^*)| 
    \leq le^{1/5}\sum_{i=1}^{n-1}|\sum_{j=1}^{i}\vw_{\tau_k(j)}(\vz^*)|.
\end{align*}
To bound $\norm{\vr_k}^2$, we use the inequality $(\sum_{i=1}^{n}x_i)^2 \leq n \sum_{i=1}^{n} x_i^2$ which follows from Young's inequality. This yields:
\begin{align*}
    \norm{\vr_k}^2 &\leq l^2 e^{2/5}\left(\sum_{i=1}^{n-1}|\sum_{j=1}^{i}\vw_{\tau_k(j)}(\vz^*)|\right)^2 \leq l^2 e^{2/5}\left(\sum_{i=1}^{n-1}i|1/i\sum_{j=1}^{i}\vw_{\tau_k(j)}(\vz^*)|\right)^2 \\ 
    &\leq l^2 e^{2/5}(n-1)\sum_{i=1}^{n-1}i^2 |1/i\sum_{j=1}^{i}\vw_{\tau_k(j)}(\vz^*)|^2.
\end{align*}
We now convert the above inequality into an upper bound for $\E{\norm{\vr_k}^2}$. For the sake of clarity, we illustrate the case of GDA-RR and GDA-SO separately. However, we remark that the proof for both cases is practically identical and relies on an application of Lemma \ref{app-lem:sample-mean-wor} to the term $|1/i\sum_{j=1}^{i}\vw_{\tau_k(j)}(\vz^*)|^2$.

\paragraph{Step 3A: $\bE[\norm{\vr_k}^2]$ for GDA-RR}
We write $\bE[\norm{\vr_k}^2] = \bE_{\tau_1, \ldots, \tau_K}[\norm{\vr_k}^2]$ to explicitly denote the dependence of the expectation on the uniform random permutations $\tau_1, \ldots, \tau_K$. We then use Lemma \ref{app-lem:sample-mean-wor} to bound $\bE[\norm{\vr_k}^2]$ as follows:
\begin{align*}
    \bE[\norm{\vr_k}^2] &\leq l^2 e^{2/5}(n-1)\sum_{i=1}^{n-1}i^2 \bE_{\tau_1, \ldots, \tau_K}[|1/i\sum_{j=1}^{i}\vw_{\tau_k(j)}(\vz^*)|^2] \\
    &\leq l^2 e^{2/5}(n-1)\sum_{i=1}^{n-1}i^2 \bE_{\tau_k}[|1/i\sum_{j=1}^{i}\vw_{\tau_k(j)}(\vz^*)|^2] \\
    &\leq l^2 e^{2/5}(n-1) \sum_{i=1}^{n-1}i^2 \frac{n-i}{i(n-1)}\sigma_*^2 = l^2 e^{2/5} \sigma_*^2 \sum_{i=1}^{n-1}i(n-i) \\
    &\leq l^2 n^3 \sigma_*^{2}/4,
\end{align*}
where $\sigma_*^{2}$ denotes the gradient variance at the minimax point $\vz^*$, defined as $\sigma_*^2 = 1/n\sum_{i=1}^{n}\norm{\vw_{i}(\vz^*)}^2$. We highlight that the second inequality follows from the independence of the uniformly sampled permutations $\tau_1, \ldots, \tau_K$ and the third inequality follows from Lemma \ref{app-lem:sample-mean-wor} (recall that $1/n\sum_{j=1}^{n} \vw(\vz^*) = \nu(\vz^*) = 0$).

\paragraph{Step 3B: $\bE[\norm{\vr_k}^2]$ for GDA-SO}
For GDA-SO, $\tau_1 = \ldots = \tau_K = \tau$ where $\tau$ is a uniformly sampled random permutation. Hence, this yields the following:  
\begin{align*}
    \norm{\vr_k}^2 \leq l^2 e^{2/5}(n-1)\sum_{i=1}^{n-1}i^2 |1/i\sum_{j=1}^{i}\vw_{\tau(j)}(\vz^*)|^2.
\end{align*}
Proceeding along the same lines as Step 3A, we obtain
\begin{align*}
    \E{\norm{\vr_k}^2} &\leq l^2 e^{2/5}(n-1)\sum_{i=1}^{n-1}i^2 \bE_{\tau}[ |1/i\sum_{j=1}^{i}\vw_{\tau(j)}(\vz^*)|^2] \\
    &\leq l^2 e^{2/5}(n-1) \sum_{i=1}^{n-1}i^2 \frac{n-i}{i(n-1)}\sigma_*^2 \\
    &\leq l^2 n^3 \sigma_*^{2}/4.
\end{align*}
Hence, for both GDA-RR and GDA-SO, $\vr_k$ is bounded in expectation as $\E{\norm{\vr_k}^2} \leq l^2 n^3 \sigma_*^{2}/4$.

Having bounded both $\norm{\vH_k}$ and $\bE[\norm{\vr_k}^2]$, we now proceed to present a complete unified proof of convergence for GDA-RR and GDA-SO.
\paragraph{Step 4: Convergence analysis}
Unrolling \eqref{app-eqn:gda-wor-epoch-update} for $K$ epochs and setting $\vz^1_0 = \vz_0$ gives us the following:
\begin{align*}
    \vz^{K+1}_0 - \vz^* = \left[\revprod_{k=1}^{K}\vH_k\right](\vz_0 - \vz^*) + \alpha^2 \sum_{k=1}^{K}\left[\revprod_{j=k+1}^{K}\vH_j\right]\vr_k.
\end{align*}
Using the triangle inequality and the bound $\norm{\vH_k} \leq 1 - n\alpha \mu / 2$ results in:
\begin{align*}
    \norm{\vz^{K+1}_0 - \vz^*} &\leq  \left[\revprod_{k=1}^{K}\norm{\vH_k}\right]\norm{\vz_0 - \vz^*} + \alpha^2 \sum_{k=1}^{K}\left[\revprod_{j=k+1}^{K}\norm{\vH_j}\right]\norm{\vr_k} \\
    &\leq (1 - n \alpha \mu /2)^K \norm{\vz_0 - \vz^*} + \alpha^2 \sum_{k=1}^{K}(1 - n\alpha\mu/2)^{K-k}\norm{\vr_k}.
\end{align*}
By applying Young's inequality, we conclude the following: 
\begin{align*}
    \norm{\vz^{K+1}_0 - \vz^*}^2 &\leq 2(1 - n\alpha\mu/2)^{2K}\norm{\vz_0 - \vz^*}^2 + 2 \alpha^4 K \sum_{k=1}^{K} (1-n\alpha \mu/2)^{2(K-k)}|\vr_k|^2 \\
    &\leq 2e^{-n\alpha \mu K}\norm{\vz_0 - \vz^*}^2 + 2 \alpha^4 K \sum_{k=1}^{K} (1 - n\alpha \mu/2)^{2(K-k)} |\vr_k|^2.
\end{align*}
Taking expectations (with respect to the uniform random permutations $\tau_1, \ldots, \tau_K$ for RR and $\tau$ for SO respectively) on both sides of the above inequality and substituting the upper bound on $\bE[\norm{\vr_k}^2]$ as derived in Step 3 gives us the following:
\begin{align*}
    \E{\norm{\vz^{K+1}_0 - \vz^*}^2} &\leq 2e^{-n\alpha \mu K}\norm{\vz_0 - \vz^*}^2 + 2 \alpha^4 K \sum_{k=1}^K (1 - n \alpha \mu/2)^{2(K-k)} \E{\norm{\vr_k}^2} \\
    &\leq 2e^{-n\alpha \mu K}\norm{\vz_0 - \vz^*}^2 + \left(l^2 n^3 \sigma_*^2 / 2 \right) \alpha^4 K \sum_{k=1}^K (1 - n \alpha \mu/2)^{2(K-k)}.
\end{align*}
By simplifying the last term, we obtain the following guarantee that holds for GDA-RR and GDA-SO whenever $\alpha \leq \frac{\mu}{5nl^2}$,
\begin{equation}
\label{app-eqn:gda-wor-alpha-rate}
    \E{\norm{\vz^{K+1}_0 - \vz^*}^2} \leq 2e^{-n\alpha \mu K}\norm{\vz_0 - \vz^*}^2 + \frac{l^2 \sigma_*^2 \alpha^3 n^2 K}{ \mu}.
\end{equation}
To complete the proof, we substitute $\alpha = \min \{  \mu/5nl^2, 2 \log (\norm{\nu(\vz_0 )}n^{1/2}K/\mu )/\mu n K \}$ in \eqref{app-eqn:gda-wor-alpha-rate}.

Under this choice of step-size, using $\alpha \leq \frac{2 \log \left(\norm{\nu(\vz_0 )}n^{1/2}K/\mu \right)}{\mu n K}$, we bound the second term in the right hand side of \eqref{app-eqn:gda-wor-alpha-rate} as
\begin{align*}
    \frac{l^2 \sigma_*^2 \alpha^3 n^2 K}{\mu} \leq \frac{8l^2 \sigma_*^2}{\mu^4}\frac{\log^3\left(\norm{\nu(\vz_0 )}n^{1/2}K/\mu \right)}{nK^2}.
\end{align*}
By substituting the above expression into \eqref{app-eqn:gda-wor-alpha-rate}, we obtain the following:
\begin{equation}
\label{app-eqn:gda-wor-rate-eqn-1}
\E{\norm{\vz^{K+1}_0 - \vz^*}^2} \leq 2e^{-n\alpha \mu K}\norm{\vz_0 - \vz^*}^2 + \frac{8\kappa^2 \sigma_*^2}{\mu^2}\frac{\log^3\left(\norm{\nu(\vz_0 )}n^{1/2}K/\mu \right)}{nK^2}.
\end{equation}
We now consider the following cases:
\paragraph{Case 1: $\frac{\mu}{5nl^2} \leq \frac{2 \log \left(\norm{\nu(\vz_0 )}n^{1/2}K/\mu \right)}{\mu n K}$} By our choice of the step-size, $\alpha = \frac{\mu}{5nl^2}$, which implies that
\begin{align*}
    2e^{-n\alpha\mu K}\norm{\vz_0-\vz^*}^2 = 2e^{-\frac{K}{5\kappa^2}}\norm{\vz_0-\vz^*}^2.
\end{align*}
By substituting into \eqref{app-eqn:gda-wor-rate-eqn-1}, we obtain the following rate for Case 1, 
\begin{equation}
\label{app-eqn:gda-wor-convergence-rate-case1}
\E{\norm{\vz^{K+1}_0 - \vz^*}^2} \leq 2e^{-\frac{K}{5\kappa^2}}\norm{\vz_0 - \vz^*}^2 + \frac{8\kappa^2 \sigma_*^2}{\mu^2}\frac{\log^3\left(\norm{\nu(\vz_0 )}n^{1/2}K/\mu \right)}{nK^2}.
\end{equation}
\paragraph{Case 2: $\frac{2 \log \left(\norm{\nu(\vz_0 )}n^{1/2}K/\mu \right)}{\mu n K} \leq \frac{\mu}{5nl^2}$}
By our choice of the step-size, $\alpha = \frac{2 \log \left(\norm{\nu(\vz_0 )}n^{1/2}K/\mu \right)}{\mu n K}$. It follows that:
\begin{align*}
    2e^{-n\alpha\mu K}\norm{\vz_0-\vz^*}^2 &= 2e^{-2 \log \left(\norm{\nu(\vz_0 )}n^{1/2}K/\mu \right)} 
    \leq \frac{\mu \norm{\vz_0 - \vz^*}^2}{\norm{\nu(\vz_0)}^2}\frac{2}{nK^2} \leq \frac{2}{nK^2},
\end{align*}
where the last inequality follows from the $\mu$-strong monotonicity of $\nu$.

Substituting the above inequality into \eqref{app-eqn:gda-wor-rate-eqn-1}, we obtain the following rate for Case 2, 
\begin{equation}
\label{app-eqn:gda-wor-convergence-rate-case2}
\E{\norm{\vz^{K+1}_0 - \vz^*}^2} \leq \frac{2\mu^2 + 8\kappa^2 \sigma_*^2\log^3\left(\norm{\nu(\vz_0 )}n^{1/2}K/\mu \right)}{\mu^2 n K^2}.
\end{equation}
Taking the maximum of the right hand side of \eqref{app-eqn:gda-wor-convergence-rate-case1} and \eqref{app-eqn:gda-wor-convergence-rate-case2}, we finally obtain the desired last-iterate convergence rate which holds for both cases:

\begin{equation}
\label{app-eqn:gda-wor-constant-convergence-rate}
\E{\norm{\vz^{K+1}_0 - \vz^*}^2} \leq 2e^{-\frac{K}{5\kappa^2}}\norm{\vz_0 - \vz^*}^2 +  \frac{2\mu^2 + 8\kappa^2 \sigma_*^2\log^3\left(\norm{\nu(\vz_0 )}n^{1/2}K/\mu \right)}{\mu^2 n K^2}.
\end{equation}
Suppressing constant terms and logarithmic factors, we obtain the following:
\begin{equation}
\label{app-eqn:gda-wor-convergence-rate}
\E{\norm{\vz^{K+1}_0 - \vz^*}^2} = \Tilde{O}(e^{\nicefrac{-K}{5\kappa^2}} + \nicefrac{1}{nK^2}).
\end{equation}
We highlight that the obtained last-iterate convergence guarantee holds for both GDA-RR and GDA-SO, and is applicable for any number of epochs $K \geq 1$. 
\end{proof}
\paragraph{Recovering the rates of full-batch GDA} We note that when $n=1$, i.e., in the non-stochastic regime, GDA-RR/SO reduces to full-batch GDA and $\sigma^2_* = 0$. Substituting this into \eqref{app-eqn:gda-wor-constant-convergence-rate}, we note that Theorem \ref{app-thm:gda-wor-convergence} implies a last-iterate convergence guarantee of $O(\exp(-K/5\kappa^2))$, which matches the tight convergence rate of full-batch GDA as established in Theorem \ref{app-thm:batch-gda-convergence}. Thus, our analysis of GDA-RR/SO automatically recovers the rates of full-batch GDA and the exponential dependence of $-K/5\kappa^2$ is optimal up to constant factors. We highlight that the above argument holds even when $n > 1$ as long as $\sigma^2_* = 0$ (which occurs whenever $\vz^*$ is the common root of all the operators $\omega_i$). 

\paragraph{Comparison with lower bounds and uniform sampling} As discussed earlier in Section \ref{sec:rr-so}, it is easy to see that the obtained convergence rate for GDA-RR/SO is nearly tight, i.e., differs from the lower bound only by an exponentially decaying term $\exp(-K/5\kappa^2)$. In fact, Theorem \ref{app-thm:gda-wor-convergence} implies an $\lO(1/nK^2)$ convergence rate (suppressing constants such as $\kappa, \mu, \sigma_*$ as well as logarithmic factors) when $\exp(-K/5\kappa^2) + 1/nK^2 = O(1/nK^2)$, which holds when $K$ satisfies an epoch requirement of the form $K \geq 10\kappa^2 \log(n^{1/2}K)$. Thus, as per Theorem \ref{app-thm:gda-wor-convergence}, GDA-RR/SO improves upon SGDA with replacement and matches the lower bound (modulo logarithmic factors) when $K \geq 10\kappa^2 \log(n^{1/2}K)$. Furthermore, since our analysis of GDA-RR/SO recovers the rates of full-batch GDA, one can infer that the dependence of this epoch requirement on $\kappa$ cannot be improved for constant step-sizes (i.e., using constant step-sizes, one cannot obtain $\Tilde{O}(1/nK^2)$ rates by imposing an epoch requirement of the form $K \geq C\kappa^{a} \log(n^{1/2}K)$ with $a < 2$), as it would contradict the optimality of the $-K/5\kappa^2$ exponential dependence. We remark that it might be possible to relax or even remove this epoch requirement for matching the lower bound, by either assuming strong monotonicity of the components $\omega_i$ or by using time-varying step-sizes, as was done in \citet{MischenkoRR2020} and \citet{AhnRR2020} respectively, in the minimization setting. 
\subsection{Analysis of GDA-AS}
\label{app-sec:gda-as}
\begin{theorem}[Convergence of GDA-AS]
\label{app-thm:gda-as-convergence}
Consider Problem \eqref{p:strong-monotone-vi} for the $\mu$-strongly monotone operator $\nu(\vz) = \nicefrac{1}{n} \sum_{i=1}^{n} \omega_i(\vz)$ where each $\omega_i$ is $l$-Lipschitz, but not necessarily monotone. Let $\vz^*$ denote the unique root of $\nu$. Then, for any $\alpha \leq \frac{\mu}{5nl^2}$ and $K \geq 1$, the iterates of GDA-AS satisfy the following:
\begin{equation*}
    \max_{\tau_1, \ldots, \tau_K \in \mathbb{S}_n} \norm{\vz^{K+1}_0 - \vz^*}^2 \leq 2e^{-n\alpha \mu K}\norm{\vz_0 - \vz^*}^2 + \frac{3l^2 \sigma_*^2 \alpha^3 n^3 K}{ \mu}.
\end{equation*}
Setting $\alpha = \min \{  \mu/5nl^2, 2 \log (\norm{\nu(\vz_0 )}K/\mu )/\mu n K \}$ results in the following expected last-iterate convergence guarantee for GDA-AS, which holds for any $K \geq 1$:
\begin{align*}
\max_{\tau_1, \ldots, \tau_K \in \mathbb{S}_n} \norm{\vz^{K+1}_0 - \vz^*}^2 &\leq 2e^{\nicefrac{-K}{5\kappa^2}}|\vz_0 - \vz^*|^2  +  \frac{2\mu^2  +  24 \kappa^2 \sigma^2_* \log^3 (|\nu(\vz_0)|K/\mu)}{\mu^2 K^2} \\
&= \Tilde{O}(e^{\nicefrac{-K}{5\kappa^2}}  +  \nicefrac{1}{K^2}).
\end{align*}
\end{theorem}
\begin{proof}
The iterate-level update rule of GDA-AS can be expressed as
\begin{equation}
\label{app-eqn:gda-as-itr}
    \vz^{k}_{i} = \vz^{k}_{i-1} - \alpha \vw_{\tau_k(i)}(\vz^k_{i-1}),
\end{equation}
where $k \in [K]$, $i \in [n]$ and $\tau_k \in \mathbb{S}_n$ denotes an arbitrary permutation chosen by the adversary at the start of epoch $k$. Recall that Steps 1 and 2 of Theorem \ref{app-thm:gda-wor-convergence} hold for any permutation $\tau_k \in \mathbb{S}_n$, and consequently, they are directly applicable to GDA-AS. As a result, we can conclude that the epoch-level update rule of GDA-AS is the same as that of GDA-RR/SO, and is given by:

\begin{equation}
\label{app-eqn:gda-as-epoch-update}
    \vz^{k+1}_0 - \vz^* = \vH_k (\vz^k_0 - \vz^*) + \alpha^2 \vr_k.
\end{equation}
Note that $\vH_k$ and $\vr_k$ are as defined in Theorem \ref{app-thm:gda-wor-convergence}, and are given by:
\begin{align*}
    \vH_k &= \vI - \alpha \sum_{j=1}^{n} \left(\revprod_{t=j+1}^{n}(\vI - \alpha \vJ_{\tau_k(t)})\right)\vM_{\tau_k(j)}, \\
    \vr_k &= \sum_{i=1}^{n-1}\left[\revprod_{t=i+2}^{n}(\vI - \alpha \vJ_{\tau_k(t)}) \right]\vJ_{\tau_k(i+1)}\sum_{j=1}^{i}\vw_{\tau_k(j)}(\vz^*),
\end{align*}
where, as before, $\vM_{\tau_k(i)} = \int_{0}^{1} \grad \omega_{\tau_k(i)}(t\vz^k_0 + (1 - t) \vz^*) \textrm{d}t$ and $\vJ_{\tau_k(i)} = \int_{0}^{1} \grad \omega_{\tau_k(i)}(t\vz^k_i + (1 - t) \vz^k_0) \textrm{d}t$. Furthermore, by Step 2 of Theorem \ref{app-thm:gda-wor-convergence}, $\norm{\vH_k} \leq 1 - n \alpha \mu / 2$.

Our analysis now proceeds by controlling the influence of the noise term $\vr_k$. Recall that this was achieved in our analysis of GDA-RR/SO by using Lemma \ref{app-lem:sample-mean-wor} to upper bound $\bE[\norm{\vr_k}^2]$. However, the same strategy does not apply to GDA-AS since the permutations $\tau_1, \ldots, \tau_K$ are chosen by the adversary as per an arbitrary strategy that is unknown to us. We circumvent this by instead deriving an upper bound for $\norm{\vr_k}^2$ that holds uniformly for any permutation $\tau_k \in \mathbb{S}_n$.

We begin by repeating some initial steps from Step 3 of Theorem \ref{app-thm:gda-wor-convergence}. Applying $\norm{\vJ_{\tau_k(t)}} \leq l$ gives us the following:
\begin{align*}
    \norm{\vr_k} &\leq \sum_{i=1}^{n-1}\left[\revprod_{t=i+2}^{n}\norm{\vI - \alpha \vJ_{\tau_k(t)}} \right]\norm{\vJ_{\tau_k(i+1)}}|\sum_{j=1}^{i}\vw_{\tau_k(j)}(\vz^*)| \leq l(1 + \alpha l)^{n}\sum_{i=1}^{n-1}|\sum_{j=1}^{i}\vw_{\tau_k(j)}(\vz^*)| \\
    &\leq l(1 + \frac{\mu}{5nl})^{n}\sum_{i=1}^{n-1}|\sum_{j=1}^{i}\vw_{\tau_k(j)}(\vz^*)| \leq le^{1/5}\sum_{i=1}^{n-1}|\sum_{j=1}^{i}\vw_{\tau_k(j)}(\vz^*)|.
\end{align*}
By Young's inequality, we conclude,
\begin{align*}
    \norm{\vr_k}^2 &\leq l^2 e^{2/5}\left(\sum_{i=1}^{n-1}|\sum_{j=1}^{i}\vw_{\tau_k(j)}(\vz^*)|\right)^2  \\ 
    &\leq l^2 e^{2/5}(n-1)\sum_{i=1}^{n-1}|\sum_{j=1}^{i}\vw_{\tau_k(j)}(\vz^*)|^2 \\
    &\leq l^2 e^{2/5}(n-1)\sum_{i=1}^{n-1} i \sum_{j=1}^{i}|\vw_{\tau_k(j)}(\vz^*)|^2 \\
    &\leq l^2 e^{2/5}(n-1)\sum_{i=1}^{n-1} i \sum_{j=1}^{n}|\vw_{\tau_k(j)}(\vz^*)|^2 = l^2 e^{2/5}(n-1)\sum_{i=1}^{n-1} i \sum_{j=1}^{n}|\vw_{j}(\vz^*)|^2 \\
    &= l^2 e^{2/5}(n-1)(n \sigma^2_*)\sum_{i=1}^{n-1} i
    \leq 3 l^2 n^4 \sigma_*^2 / 4.
\end{align*}
We note that the above sequence of inequalities holds for any permutation $\tau_k$, and hence, the upper bound $\norm{\vr_k}^2 \leq 3 l^2 n^4 \sigma_*^2 / 4$ holds for any $\tau_k \in \mathbb{S}_n, \ k \in [K]$.

Having bounded both $\norm{\vH_k}$ and $\norm{\vr_k}$, we now proceed to derive the last-iterate convergence guarantee.

Unrolling \eqref{app-eqn:gda-as-epoch-update} for $K$ epochs and setting $\vz^1_0 = \vz_0$ gives us the following:
\begin{align*}
    \vz^{K+1}_0 - \vz^* = \left[\revprod_{k=1}^{K}\vH_k\right](\vz_0 - \vz^*) + \alpha^2 \sum_{k=1}^{K}\left[\revprod_{j=k+1}^{K}\vH_k\right]\vr_k.
\end{align*}
By the triangle inequality and the bound $\norm{\vH_k} \leq 1 - n\alpha \mu / 2$, we infer,
\begin{align*}
    \norm{\vz^{K+1}_0 - \vz^*} &\leq  \left[\revprod_{k=1}^{K}\norm{\vH_k}\right]\norm{\vz_0 - \vz^*} + \alpha^2 \sum_{k=1}^{K}\left[\revprod_{j=k+1}^{K}\norm{\vH_k}\right]\norm{\vr_k} \\
    &\leq (1 - n \alpha \mu /2)^K \norm{\vz_0 - \vz^*} + \alpha^2 \sum_{k=1}^{K}(1 - n\alpha\mu/2)^{K-k}\norm{\vr_k}.
\end{align*}
By applying Young's inequality, we can conclude that,
\begin{align*}
    \norm{\vz^{K+1}_0 - \vz^*}^2 &\leq 2(1 - n\alpha\mu/2)^{2K}\norm{\vz_0 - \vz^*}^2 + 2 \alpha^4 K \sum_{k=1}^{K} (1-n\alpha \mu/2)^{2(K-k)}|\vr_k|^2 \\
    &\leq 2e^{-n\alpha \mu K}\norm{\vz_0 - \vz^*}^2 + 2 \alpha^4 K \sum_{k=1}^{K} (1 - n\alpha \mu/2)^{2(K-k)} |\vr_k|^2 \\
    &\leq 2e^{-n\alpha \mu K}\norm{\vz_0 - \vz^*}^2 + (1.5 l^2 n^4 \sigma^2_*) \alpha^4 K \sum_{k=1}^{K} (1 - n\alpha \mu/2)^{2(K-k)} \\
    &\leq 2e^{-n\alpha \mu K}\norm{\vz_0 - \vz^*}^2 + \frac{3 l^2 \sigma^2_* \alpha^3 n^3 K}{\mu}.
\end{align*}
Since the above inequality holds for any sequence of permutations $\tau_1, \ldots, \tau_K$ chosen by the adversary, we obtain the following uniform guarantee for GDA-AS, whenever $\alpha \leq \frac{\mu}{5nl^2}$,
\begin{equation}
\label{app-eqn:gda-as-alpha-rate}
    \max_{\tau_1, \ldots, \tau_K \in \mathbb{S}_n} \norm{\vz^{K+1}_0 - \vz^*}^2 \leq 2e^{-n\alpha \mu K}\norm{\vz_0 - \vz^*}^2 + \frac{3l^2 \sigma_*^2 \alpha^3 n^3 K}{\mu}.
\end{equation}
To complete the proof, we substitute $\alpha = \min \{  \mu/5nl^2, 2 \log (\norm{\nu(\vz_0 )}K/\mu )/\mu n K \}$ in \eqref{app-eqn:gda-as-alpha-rate} and proceed in a manner similar to Theorem \ref{app-thm:gda-wor-convergence}.

Under this choice of step-size, using $\alpha \leq \frac{2 \log \left(\norm{\nu(\vz_0 )}K/\mu \right)}{\mu n K}$, we bound the second term in the right hand side of \eqref{app-eqn:gda-as-alpha-rate} as
\begin{align*}
    \frac{3l^2 \sigma_*^2 \alpha^3 n^3 K}{2 \mu} \leq \frac{24l^2 \sigma_*^2}{\mu^4}\frac{\log^3\left(\norm{\nu(\vz_0 )}K/\mu \right)}{K^2}.
\end{align*}
Substituting the above inequality into \eqref{app-eqn:gda-as-alpha-rate}, we obtain,
\begin{equation}
\label{app-eqn:gda-as-rate-eqn-1}
\max_{\tau_1, \ldots, \tau_K \in \mathbb{S}_n} \norm{\vz^{K+1}_0 - \vz^*}^2 \leq 2e^{-n\alpha \mu K}\norm{\vz_0 - \vz^*}^2 + \frac{24\kappa^2 \sigma_*^2}{\mu^2}\frac{\log^3\left(\norm{\nu(\vz_0 )}K/\mu \right)}{K^2}.
\end{equation}
We now consider the following cases:
\paragraph{Case 1: $\frac{\mu}{5nl^2} \leq \frac{2 \log \left(\norm{\nu(\vz_0 )}K/\mu \right)}{\mu n K}$} By our choice of the step-size, $\alpha = \frac{\mu}{5nl^2}$. It follows that:
\begin{align*}
    2e^{-n\alpha\mu K}\norm{\vz_0-\vz^*}^2 = 2e^{-K/(5\kappa^2)}\norm{\vz_0-\vz^*}^2.
\end{align*}
Substituting into \eqref{app-eqn:gda-as-rate-eqn-1}, we obtain the following rate for Case 1, 
\begin{equation}
\label{app-eqn:gda-as-convergence-rate-case1}
\max_{\tau_1, \ldots, \tau_K \in \mathbb{S}_n} \norm{\vz^{K+1}_0 - \vz^*}^2 \leq 2e^{\frac{-K}{5\kappa^2}}\norm{\vz_0 - \vz^*}^2 + \frac{24\kappa^2 \sigma_*^2}{\mu^2}\frac{\log^3\left(\norm{\nu(\vz_0 )}K/\mu \right)}{K^2}.
\end{equation}
\paragraph{Case 2: $\frac{2 \log \left(\norm{\nu(\vz_0 )}K/\mu \right)}{\mu n K} \leq \frac{\mu}{5nl^2}$}
By our choice of the step-size, $\alpha = \frac{2 \log \left(\norm{\nu(\vz_0 )}K/\mu \right)}{\mu n K}$. It follows that:
\begin{align*}
    2e^{-n\alpha\mu K}\norm{\vz_0-\vz^*}^2 &= 2e^{-2 \log \left(\norm{\nu(\vz_0 )}K/\mu \right)} 
    \leq \frac{\mu \norm{\vz_0 - \vz^*}^2}{\norm{\nu(\vz_0)}^2}\frac{2}{K^2} \leq \frac{2}{K^2},
\end{align*}
where the last inequality follows from the $\mu$-strong monotonicity of $\nu$.

Substituting the above inequality into \eqref{app-eqn:gda-as-rate-eqn-1}, we obtain the following rate for Case 2, 
\begin{equation}
\label{app-eqn:gda-as-convergence-rate-case2}
\max_{\tau_1, \ldots, \tau_K \in \mathbb{S}_n} \norm{\vz^{K+1}_0 - \vz^*}^2 \leq \frac{2\mu^2 + 24 \kappa^2 \sigma^2_* \log^3 (|\nu(\vz_0)|K/\mu)}{\mu^2 K^2}.
\end{equation}
By taking the maximum of the right hand side of \eqref{app-eqn:gda-as-convergence-rate-case1} and \eqref{app-eqn:gda-as-convergence-rate-case2}, we finally obtain the desired last-iterate convergence rate.
\begin{equation}
\label{app-eqn:gda-as-const-convergence-rate}
\max_{\tau_1, \ldots, \tau_K \in \mathbb{S}_n} \norm{\vz^{K+1}_0 - \vz^*}^2 \leq 2e^{\nicefrac{-K}{5\kappa^2}}|\vz_0 - \vz^*|^2  +  \frac{2\mu^2  +  24 \kappa^2 \sigma^2_* \log^3 (|\nu(\vz_0)|K/\mu)}{\mu^2 K^2}.
\end{equation}
Suppressing constant terms and logarithmic factors, we get,
\begin{equation}
\label{app-eqn:gda-as-convergence-rate}
\max_{\tau_1, \ldots, \tau_K \in \mathbb{S}_n} \norm{\vz^{K+1}_0 - \vz^*}^2 = \Tilde{O}(e^{\nicefrac{-K}{5\kappa^2}}  +  \nicefrac{1}{K^2}).
\end{equation}
\end{proof}
\paragraph{Comparison with lower bounds} By repeating the same arguments as those used for GDA-RR/SO, it is easy to see that our analysis of GDA-AS recovers the tight rates of full-batch GDA in the deterministic, i.e., $n=1$ regime. Furthermore, Theorem \ref{app-thm:gda-wor-convergence} implies an $\lO(1/K^2)$ convergence rate, which matches the lower bound modulo logarithmic factors, when $e^{\nicefrac{-K}{5\kappa^2}}  +  \nicefrac{1}{K^2} = O(1/K^2)$, i.e., when $K$ satisfies an epoch requirement of the form $K \geq 10 \kappa^2 \log(K)$. As before, since our analysis of GDA-AS recovers the rates of full-batch GDA, the $\kappa^2$ dependence of the epoch requirement cannot be improved for constant step-sizes.
\section{Analysis of PPM without Replacement}
\label{app-sec:ppm}
We now present the convergence proofs of the proximal point method without replacement (i.e. PPM-RR/SO/AS). The structure of these results largely mirror that of our analysis of GDA without replacement, since the key techniques used in our analysis is an adaptation of the linearization technique to PPM and the application of Lemma \ref{app-lem:sample-mean-wor} to control the influence of noise in PPM without replacement. To this end, we begin by presenting a linearization-based proof of convergence of full batch PPM, followed by a unified proof of convergence of PPM-RR and PPM-SO, and conclude with the extension of our proof techniques to the adversarial shuffling regime by analyzing PPM-AS.
\subsection{Analysis of Full-batch PPM by Linearization}
\label{app-sec:ppm-batch}
\begin{theorem}[Convergence of full-batch PPM]
\label{app-thm:batch-ppm-convergence}
Consider the $l$-Lipschitz and $\mu$-strongly monotone operator $\nu : \bR^d \rightarrow \bR^d$ and let $\vz^*$ denote the unique root of $\nu$. For any step-size $0 < \alpha < 1/l$, the iterates $\vz_{k}$ of full-batch PPM satisfy the following recurrence:
\begin{align*}
    \vz_{k+1} - \vz^* = (\vI + \alpha \vM_{k+1})^{-1} (\vz_{k} - \vz^*),
\end{align*}
where $\vI + \alpha \vM_{k+1}$ is invertible and its singular values are lower bounded as $\sigma_{\min}(\vI + \alpha \vM_{k+1}) \geq 1 + \alpha \mu$. Consequently, setting $\alpha = 1/2l$ gives us the following last iterate convergence guarantee:
\begin{align*}
    \norm{\vz_{K+1} - \vz^{*}}^2 \leq (1 + 1/2\kappa)^{-2K} \norm{\vz_{0} - \vz^{*}}^2.
\end{align*}
\end{theorem}
\begin{proof}
The iterates of full-batch PPM with step-size $\alpha$ satisfy
\begin{align*}
    \vz_{k+1} = \vz_k - \alpha \nu(\vz_{k+1}).
\end{align*}
From the above equation, we conclude that $\vz_{k+1}$ is a fixed point of the operator $\zeta(\vz) = \vz_{k} - \alpha \nu(\vz)$. Since $\alpha < 1/l$, $\zeta$ is a contraction mapping. Hence, by the Banach Fixed Point Theorem, we conclude that $\vz_{k+1}$ exists and is uniquely defined. This allows us to proceed in a manner similar to Theorem \ref{app-thm:batch-gda-convergence} and linearize $\nu(\vz_{k+1})$ about $\vz^*$. Thus, by the Lipschitz continuity of $\nu$ and the Fundamental Theorem of Calculus for Lebesgue Integrals:
\begin{align*}
    \nu(\vz_{k+1}) &= \nu(\vz^*) + \int_{0}^{1} \grad \nu(t \vz_{k+1} + (1 - t) \vz^*)(\vz_{k+1} - \vz^*) \textrm{d}t \\
    &= \vM_{k+1}(\vz_{k+1} - \vz^*),
\end{align*}
where $\vM_{k+1} = \int_{0}^{1} \grad \nu(t \vz_{k+1} + (1 - t) \vz^*)\textrm{d}t$ is well defined by Lemma \ref{app-lem:lip-operator-property}. Substituting this expansion into the PPM update rule gives us the following:
\begin{equation}
\label{app-eqn:batch-ppm-lds}
\vz_{k} - \vz^* = (\vI + \alpha \vM_{k+1})(\vz_{k+1} - \vz^*).
\end{equation}
Using the same arguments as in Theorem \ref{app-thm:batch-gda-convergence}, we conclude that the $\mu$-strong monotonicity of $\nu$ implies $\vv^T \vM_{k+1} \vv \geq \mu \norm{\vv}^2 \ \forall \vv \in \bR^d$. Using this, we lower bound the singular values of $\vI + \alpha \vM_{k+1}$ as follows:

Note that for any $\vv \in \bR^{d}$, 
\begin{align*}
    \norm{(\vI + \alpha \vM_{k+1}) \vv}^2 &= \norm{\vv}^2 + \alpha\vv^T \vM_{k+1} \vv + \alpha \vv^T \vM^T_{k+1} \vv + \alpha^2 \norm{\vM_{k+1} \vv}^2 \\
    &\geq (1 + 2 \alpha \mu + \alpha^2 \mu^2) \norm{\vv}^2,
\end{align*}
where the last inequality follows from the Cauchy-Schwarz inequality since $\norm{\vv}\norm{\vM_{k+1}\vv} \geq \vv^T \vM_{k+1} \vv \geq \mu \norm{\vv}^2$.

Thus, the singular values of $\vI + \alpha \vM_{k+1}$ are lower bounded as $\sigma_{\min}(\vI + \alpha \vM_{k+1}) \geq 1 + \alpha \mu > 0$. Since $\vI + \vM_{k+1}$ is a square matrix, this implies that $\vI + \alpha \vM_{k+1}$ is invertible and
\begin{align*}
    \norm{(\vI + \alpha \vM_{k+1})^{-1}} = \sigma_{\max}((\vI + \alpha \vM_{k+1})^{-1}) = 1/\sigma_{\min}(\vI + \alpha \vM_{k+1}) \leq (1 + \alpha \mu)^{-1}.
\end{align*}

It follows that,
\begin{align*}
    \vz_{k+1} - \vz^{*} &= (\vI + \alpha \vM_{k+1})^{-1}(\vz_k - \vz^*) \\
    \implies \norm{\vz_{k+1} - \vz^{*}}^2 &\leq (1 + \alpha \mu)^{-2}  \norm{\vz_{k} - \vz^{*}}^2.
\end{align*}
Thus, by setting $\alpha = 1/2l$, we obtain the following convergence guarantee
\begin{align*}
    \norm{\vz_{K+1} - \vz^{*}}^2 \leq (1 + 1/2\kappa)^{-2K} \norm{\vz_{0} - \vz^{*}}^2.
\end{align*}
\end{proof}
\subsection{A Unified Analysis of PPM-RR and PPM-SO}
\label{app-sec:ppm-wor}
\begin{theorem}[Convergence of PPM-RR/SO]
\label{app-thm:ppm-wor-convergence}
Consider Problem \eqref{p:strong-monotone-vi} for the $\mu$-strongly monotone operator $\nu(\vz) = \nicefrac{1}{n} \sum_{i=1}^{n} \omega_i(\vz)$ where each $\omega_i$ is $l$-Lipschitz, but not necessarily monotone. Let $\vz^*$ denote the unique root of $\nu$. Then, for any $\alpha \leq \frac{\mu}{5nl^2}$ and $K \geq 1$, the iterates of PPM-RR and PPM-SO satisfy the following:
\begin{equation*}
    \E{\norm{\vz^{K+1}_0 - \vz^*}^2} \leq 2e^{-n\alpha \mu K}\norm{\vz_0 - \vz^*}^2 + \frac{l^2 \sigma_*^2 \alpha^3 n^2 K}{\mu}.
\end{equation*}
Setting $\alpha = \min \{  \mu/5nl^2, 2 \log (\norm{\nu(\vz_0 )}n^{1/2}K/\mu )/\mu n K \}$ results in the following expected last-iterate convergence guarantee for both PPM-RR and PPM-SO, which holds for any $K \geq 1$:
\begin{align*}
\E{|\vz^{K+1}_0 - \vz^*|^2} &\leq  2e^{\nicefrac{-K}{5\kappa^2}}|\vz_0 - \vz^*|^2 + \frac{2\mu^2 + 8 \kappa^2 \sigma^2_* \log^3 (|\nu(\vz_0)|n^{1/2}K/\mu)}{\mu^2 nK^2} \\
&= \Tilde{O}(e^{\nicefrac{-K}{5\kappa^2}} + \nicefrac{1}{nK^2}).
\end{align*}
\end{theorem}

\begin{proof}
Without loss of generality, we express the iterate-level update rule of PPM-RR and PPM-SO jointly as
\begin{equation}
\label{app-eqn:ppm-wor-itr}
    \vz^{k}_{i} = \vz^{k}_{i-1} - \alpha \vw_{\tau_k(i)}(\vz^k_{i}),
\end{equation}
where $k \in [K]$ and $i \in [n]$. Similar to Theorem \ref{app-thm:batch-ppm-convergence}, the Banach Fixed Point Theorem guarantees that $\vz^k_{i}$ is well defined and unique since $\alpha \leq \mu/5nl^2 < 1/l$.

We note that $\tau_k \sim \textrm{Uniform}(\mathbb{S}_n)$ for every $k \in [K]$ for PPM-RR (i.e., $\tau_k$ is a permutation of $[n]$ that is resampled uniformly at every epoch) whereas for PPM-SO, $\tau_k = \tau \ \forall k \in [K]$ where $\tau \sim \textrm{Uniform}(\mathbb{S}_n)$ (i.e., the permutation $\tau$ is uniformly sampled before the first epoch, then reused for all subsequent epochs). We also define the transposed permutation $\ttau_k$ as $\ttau_k(i) = \tau_k(n+1-i)$ for any $i \in [n]$. 

The remainder of our proof follows the structure of Theorem \ref{app-thm:gda-wor-convergence}. In particular, we apply the linearization technique as illustrated in Theorem \ref{app-thm:batch-ppm-convergence} to PPM without replacement and derive a linearized epoch-level update rule of the form $\vz^{k}_0 - \vz^* = \vH_k (\vz^{k+1}_0 - \vz^*) + \alpha^2 \vr_k$. This is followed by lower bounding $\sigma_{\min}(\vH_k)$ using the Lipschitz continuity and strong monotonicity of $\mu$ and upper bounding $\bE[\norm{\vr_k}^2]$ using Lemma \ref{app-lem:sample-mean-wor}. The proof is completed by unrolling the update rule for $K$ epochs and carefully choosing the step-size.

\paragraph{Step 1: Linearized epoch-level update rule}
The approach for this step mirrors that of Theorem \ref{app-thm:gda-wor-convergence} and hinges on the insight that for small enough step-sizes, the dynamics of the iterates $\vz^{k+1}_0$ of PPM without replacement can be treated as a noisy version of full-batch PPM. Complementing this insight with the linearization-based analysis of full-batch PPM in Theorem \ref{app-thm:batch-ppm-convergence} motivates the following decomposition:
\begin{align*}
    \vw_{\tau_k(i)}(\vz^k_{i}) &= \vw_{\tau_k(i)}(\vz^{*}) + [\vw_{\tau_k(i)}(\vz^k_n) - \vw_{\tau_k(i)}(\vz^*)] + [\vw_{\tau_k(i)}(\vz^k_{i}) - \vw_{\tau_k(i)}(\vz^k_n)] \\
    &= \vw_{\tau_k(i)}(\vz^{*}) + \int_{0}^{1} \grad \omega_{\tau_k(i)}(t\vz^k_n + (1 - t) \vz^*)(\vz^k_n - \vz^*) \textrm{d}t \\
    &+ \int_{0}^{1} \grad \omega_{\tau_k(i)}(t\vz^k_{i} + (1 - t) \vz^k_n)(\vz^k_{i} - \vz^k_n) \textrm{d}t. 
\end{align*}
We further define $\vM_{\tau_k(i)}$ and $\vJ_{\tau_k(i)}$ as follows:
\begin{align*}
    \vM_{\tau_k(i)} &= \int_{0}^{1} \grad \omega_{\tau_k(i)}(t\vz^k_n + (1 - t) \vz^*) \textrm{d}t, \\
    \vJ_{\tau_k(i)} &= \int_{0}^{1} \grad \omega_{\tau_k(i)}(t\vz^k_i + (1 - t) \vz^k_n) \textrm{d}t.
\end{align*}
By repeating the same arguments as in Theorem \ref{app-thm:batch-ppm-convergence}, we note that the $l$-Lipschitz continuity of $\omega_i(\vz) \ \forall i \in [n]$ implies that $\vM_{\tau_k(i)}$ and $\vJ_{\tau_k(i)}$ are well defined and bounded as  $\norm{\vM_{\tau_k(i)}} \leq l$ and $\norm{\vJ_{\tau_k(i)}} \leq l$. It follows that
\begin{equation}
\label{app-eqn:ppm-wor-grads}
    \vw_{\tau_k(i)}(\vz^k_{i}) = \vw_{\tau_k(i)}(\vz^{*}) + \vM_{\tau_k(i)}(\vz^k_n - \vz^*) + \vJ_{\tau_k(i)}(\vz^k_{i} - \vz^k_n).
\end{equation}
Substituting \eqref{app-eqn:ppm-wor-grads} into the update equation \eqref{app-eqn:gda-wor-itr} for $\vz^k_{n}$ yields,
\begin{align*}
    \vz^k_{n-1} - \vz^* &=
    (\vI + \alpha \vM_{\tau_k(n)})(\vz^k_n - \vz^*) + \alpha \vJ_{\sigma_k(n)}(\vz^k_n - \vz^k_n) +  \alpha \vw_{\tau_k(n)}(\vz^*) \\
    &= (\vI + \alpha \vM_{\ttau_k(1)})(\vz^k_n - \vz^*) + \alpha \vw_{\ttau_k(1)}(\vz^*).
\end{align*}
Repeating the same for $\vz^k_{n-1}$ results in the following:
\begin{align*}
    \vz^k_{n-2} - \vz^* &= \vz^k_{n-1} - \vz^* + \alpha \vw_{\tau_k(n-1)}(\vz^*) + \alpha \vM_{\tau_k(n-1)}(\vz^k_n - \vz^*) + \alpha \vJ_{\tau_k(n-1)}(\vz^k_{n-1} - \vz^k_n) \\
    &= \vz^k_{n-1} - \vz^* + \alpha \vw_{\ttau_k(2)}(\vz^*) + \alpha \vM_{\ttau_k(2)}(\vz^k_n - \vz^*) + \alpha \vJ_{\ttau_k(2)}(\vz^k_{n-1} - \vz^k_n) \\
    &= (\vI + \alpha \vJ_{\ttau_k(2)})(\vz^k_{n-1} - \vz^*) + \alpha (\vM_{\ttau_k(2)} - \vJ_{\ttau_k(2)})(\vz^k_n - \vz^*) + \alpha \vw_{\ttau_k(2)}(\vz^*) \\
    &= (\vI + \alpha \vJ_{\ttau_k(2)})[(\vI + \alpha \vM_{\ttau_k(1)})(\vz^k_n - \vz^*) + \alpha \vw_{\ttau_k(1)}(\vz^*)] \\
    &+ \alpha (\vM_{\ttau_k(2)} - \vJ_{\ttau_k(2)})(\vz^k_n - \vz^*) + \alpha \vw_{\ttau_k(2)}(\vz^*) \\
    &= [(\vI + \alpha \vJ_{\ttau_k(2)})(\vI + \alpha \vM_{\ttau_k(1)}) + \alpha (\vM_{\ttau_k(2)} - \vJ_{\ttau_k(2)})](\vz^k_n - \vz^*) \\
    &+ \alpha [\vw_{\ttau_k(2)}(\vz^*) + (\vI + \alpha \vJ_{\ttau_k(2)})\vw_{\ttau_k(1)}(\vz^*)] \\
    &= [\vI + \alpha \vM_{\ttau_k(2)} + \alpha(\vI + \alpha \vJ_{\ttau_k(2)})\vM_{\ttau_k(1)}](\vz^k_n - \vz^*) \\
    &+ \alpha [\vw_{\ttau_k(2)}(\vz^*) + (\vI + \alpha \vJ_{\ttau_k(2)})\vw_{\ttau_k(1)}(\vz^*)].
\end{align*}
By applying the same process for preceeding iterates and substituting $\vz^k_n = \vz^{k+1}_0$, we obtain the following epoch-level update rule for PPM-RR and PPM-SO
\begin{align}
    \vz^{k}_0 - \vz^* &= [\vI + \alpha \sum_{j=1}^{n} \left(\revprod_{t=j+1}^{n}(\vI + \alpha \vJ_{\ttau_k(t)})\right)\vM_{\ttau_k(j)}](\vz^{k+1}_0 - \vz^*) \nonumber \\
    &+ \alpha \sum_{j=1}^{n}\left(\revprod_{t=j+1}^{n}(\vI + \alpha \vJ_{\ttau_k(t)})\right)\vw_{\ttau_k(j)}(\vz^*).
\label{app-eqn:ppm-wor-incomplete-epoch}
\end{align}
We clarify that the matrix products in \eqref{app-eqn:ppm-wor-incomplete-epoch} are in reverse order, and hence reduce to the empty product, which is defined to be $\vI$, when $j=i$. Furthermore $\ttau_k$ denotes the transposed permutation $\ttau_k(i) = \tau_k(n+1-i)$.

We simplify the second term in the right hand side of \eqref{app-eqn:ppm-wor-incomplete-epoch} using the summation by parts identity. To this end, we define $a_j$ and $b_j$ as
\begin{align*}
    a_j &= \revprod_{t=j+1}^{n}(\vI + \alpha \vJ_{\ttau_k(t)}), \\
    b_j &= \vw_{\ttau_k(j)}(\vz^*),
\end{align*}
and observe that,
\begin{align*}
    \sum_{j=1}^{n}b_j = \sum_{j=1}^{n} \vw_{\ttau_k(j)}(\vz^*) = n \nu(\vz^{*}) = 0,
\end{align*}
since $\vz^*$ is the unique root of $\nu$. We now apply the summation by parts identity to obtain the following:
\begin{align*}
    \sum_{j=1}^{n} a_j b_j &= a_n \sum_{j=1}^{n}b_j - \sum_{i=1}^{n-1} (a_{i+1} - a_i)\sum_{j=1}^{i}b_j \\
    &= - \sum_{i=1}^{n-1} \left[\revprod_{t=i+2}^{n}(\vI + \alpha \vJ_{\ttau_k(t)}) - \revprod_{t=i+1}^{n}(\vI + \alpha \vJ_{\ttau_k(t)}) \right]\sum_{j=1}^{i}\vw_{\ttau_k(j)}(\vz^*) \\
    &= \alpha \sum_{i=1}^{n-1}\left[\revprod_{t=i+2}^{n}(\vI + \alpha \vJ_{\ttau_k(t)}) \right]\vJ_{\ttau_k(i+1)}\sum_{j=1}^{i}\vw_{\ttau_k(j)}(\vz^*).
\end{align*}
Hence, we conclude that:
\begin{align*}
    \sum_{j=1}^{n}\left[\revprod_{t=j+1}^{n}(\vI + \alpha \vJ_{\ttau_k(t)})\right]\vw_{\ttau_k(j)}(\vz^*) = \alpha \sum_{i=1}^{n-1}\left[\revprod_{t=i+2}^{n}(\vI + \alpha \vJ_{\ttau_k(t)}) \right]\vJ_{\ttau_k(i+1)}\sum_{j=1}^{i}\vw_{\ttau_k(j)}(\vz^*).
\end{align*}
We define $\vH_k$ and $\vr_k$ as
\begin{align*}
    \vH_k &= \vI + \alpha \sum_{j=1}^{n} \left(\revprod_{t=j+1}^{n}(\vI + \alpha \vJ_{\ttau_k(t)})\right)\vM_{\ttau_k(j)}, \\
    \vr_k &= \sum_{i=1}^{n-1}\left[\revprod_{t=i+2}^{n}(\vI + \alpha \vJ_{\ttau_k(t)}) \right]\vJ_{\ttau_k(i+1)}\sum_{j=1}^{i}\vw_{\ttau_k(j)}(\vz^*),
\end{align*}
and substitute the above expressions in \eqref{app-eqn:ppm-wor-incomplete-epoch}. This gives us the following noisy linearized update rule for the epoch iterates $\vz^k_0$ of PPM-RR and PPM-SO:
\begin{equation}
\label{app-eqn:ppm-wor-epoch-update}
    \vz^{k}_0 - \vz^* = \vH_k (\vz^{k+1}_0 - \vz^*) + \alpha^2 \vr_k.
\end{equation}
We observe that \eqref{app-eqn:ppm-wor-epoch-update} closely resembles the linearized update rule \eqref{app-eqn:batch-ppm-lds} of full-batch PPM as derived in Theorem \ref{app-thm:batch-ppm-convergence} with an additive noise term $\alpha^2 \vr_k$. In fact, for $n=1$, which corresponds to the full-batch regime, it is easy to see that $\vr_k = 0$ and \eqref{app-eqn:ppm-wor-epoch-update} reduces to \eqref{app-eqn:batch-ppm-lds}. As we shall see in Theorem \ref{app-thm:ppm-as-convergence}, the same update rule also applies to PPM-AS. Similar to our analysis of GDA, the derivation of an epoch-level update rule that simultaneously handles PPM-RR, PPM-SO and PPM-AS is a key component of our analysis. 

Thus, we proceed in a manner similar to Theorem \ref{app-thm:batch-ppm-convergence} by lower bounding the minimum singular value of $\vH_k$.

\paragraph{Step 2: Lower bounding $\sigma_{\min}(\vH_k)$} We define the matrix $\vM$ as,
\begin{align*}
    \vM &= 1/n\sum_{j=1}^{n}\vM_{\ttau_k(j)} = 1/n \sum_{j=1}^{n} \int_{0}^{1} \grad \omega_{\ttau_k(j)}(t\vz^k_n + (1-t)\vz^*)\textrm{d}t = \int_{0}^{1} \grad \nu(t\vz^k_n + (1-t)\vz^*)\textrm{d}t.
\end{align*}
By following the same arguments as Theorem \ref{app-thm:batch-ppm-convergence}, we note that the $l$-smoothness and $\mu$-strong monotonicity of $\nu$ implies that $\norm{\vM} \leq l$ and $\vv^T \vM \vv \geq \mu \norm{\vv}^2 \ \forall \vv \in \bR^d$.  

By expanding the product terms in $\vH_k$, we obtain the following:
\begin{align*}
    \vH_k &= \vI + \alpha \sum_{j=1}^{n} \left(\revprod_{t=j+1}^{n}(\vI + \alpha \vJ_{\ttau_k(t)})\right)\vM_{\ttau_k(j)} \\
    &= \vI + n\alpha \vM + \sum_{j=2}^{n}\alpha^j\sum_{1 \leq t_1 < t_2 < \ldots < t_j \leq n} \vJ_{\ttau_k(t_j)} \vJ_{\ttau_k(t_{j-1})} \ldots \vJ_{\ttau_k(t_2)}\vM_{\ttau_k(t_1)}.
\end{align*}
We now apply Weyl's inequality for singular value perturbations, which gives us the following bound,
\begin{equation}
\label{app-eqn:ppm-wor-incomplete-spec-bound}
    \sigma_{\min}(\vH_k) \geq \sigma_{\min}(\vI + n \alpha \vM) - \sum_{j=2}^{n}\alpha^j\sum_{1 \leq t_1 < t_2 < \ldots < t_j \leq n} \norm{\vJ_{\ttau_k(t_j)}} \norm{\vJ_{\ttau_k(t_{j-1})}} \ldots \norm{\vJ_{\ttau_k(t_2)}}\norm{\vM_{\ttau_k(t_1)}}.
\end{equation}
Following the same steps as Theorem \ref{app-thm:batch-ppm-convergence}, we conclude that $\vv^T \vM \vv \geq \mu \norm{\vv}^2$ implies $\sigma_{\min}(\vI + n \alpha \vM) \geq 1 + n \alpha \mu$. Furthermore, we observe that,
\begin{align*}
    \sum_{j=2}^{n}\alpha^j\sum_{1 \leq t_1 < t_2 < \ldots < t_j \leq n} \norm{\vJ_{\ttau_k(t_j)}} \norm{\vJ_{\ttau_k(t_{j-1})}} \ldots \norm{\vJ_{\ttau_k(t_2)}}\norm{\vM_{\ttau_k(t_1)}} &\leq \sum_{j=2}^{n}\alpha^j \binom{n}{j}l^j.
\end{align*}
Substituting the above bound in \eqref{app-eqn:ppm-wor-incomplete-spec-bound} gives us the following:
\begin{align*}
    \sigma_{\min}(\vH_k) &\geq 1 + n\alpha\mu - \sum_{j=2}^{n}(\alpha l)^j \binom{n}{j} \geq 1 + n\alpha\mu - \sum_{j=2}^{n} (\alpha n l)^j \geq 1 + n\alpha\mu - \frac{\alpha^2 n^2 l^2}{1 - \alpha n l} \\
    &\geq 1 + n\alpha \mu - \frac{5\alpha^2 n^2 l^2}{4} \geq 1 + 3n\alpha \mu/4,
\end{align*}
where we substitute $\alpha \leq \frac{\mu}{5nl^2} \leq 1/5$ wherever required. Furthermore, since $\sigma_{\min}(\vH_k) \geq 1 + 3n\alpha \mu/4 > 0$, $\vH_k$ is invertible and $\norm{\vH^{-1}_k} \leq (1 + 3n \alpha \mu/4)^{-1} \leq (1 - n \alpha \mu / 2)$, which holds since $n \alpha \mu \leq 1/5$.

As before, our analysis so far holds for any permutation $\tau_k \in \mathbb{S}_n$. Consequently, Steps 1 and 2 directly generalize to the adversarial shuffling regime without any modifications. We use this as a starting point for our analysis of PPM-AS in Theorem \ref{app-thm:ppm-as-convergence}. 

We now control the magnitude of the noise term $\vr_k$ by bounding $\E{\norm{\vr_k}^2}$ similar to Theorem \ref{app-thm:gda-wor-convergence}. We note that the expectation is taken over the single uniform permutation $\tau$ for PPM-SO and over independent uniform permutations $\tau_1, \tau_2, \ldots, \tau_K$ for PPM-RR. 

\paragraph{Step 3: Upper bounding $\bE[\norm{\vr_k}^2]$}
We follow the same procedure as in Step 3 of Theorem \ref{app-thm:gda-wor-convergence} and obtain the following bound by successively applying $\norm{\vJ_{\ttau_k(i)}} \leq l$ and Young's inequality:
\begin{align*}
    \norm{\vr_k}^2 &\leq l^2 e^{2/5}(n-1)\sum_{i=1}^{n-1}i^2 |1/i\sum_{j=1}^{i}\vw_{\ttau_k(j)}(\vz^*)|^2.
\end{align*}
We now convert the above inequality into an upper bound for $\E{\norm{\vr_k}^2}$ by applying Lemma \ref{app-lem:sample-mean-wor} to the term $|1/i\sum_{j=1}^{i}\vw_{\ttau_k(j)}(\vz^*)|^2$. The case of PPM-RR and PPM-SO are discussed separately.

\paragraph{Step 3A: $\bE[\norm{\vr_k}^2]$ for PPM-RR}
We write $\bE[\norm{\vr_k}^2] = \bE_{\tau_1, \ldots, \tau_K}[\norm{\vr_k}^2]$ to explicitly denote the dependence of the expectation on the uniformly sampled random permutations $\tau_1, \ldots, \tau_K$, and then proceed as follows:
\begin{align*}
    \bE[\norm{\vr_k}^2]
    &\leq l^2 e^{2/5}(n-1)\sum_{i=1}^{n-1}i^2 \bE_{\tau_k}[|1/i\sum_{j=1}^{i}\vw_{\ttau_k(j)}(\vz^*)|^2].
\end{align*}
Since $\tau_k$ is a uniformly sampled random permutation and $\mathbb{S}_n$ is a finite group with $n!$ elements, it follows that for any $\pi \in \mathbb{S}_n$,
\begin{align*}
    \mathbb{P}_{\textrm{Uniform}}\left[\tau_k = \pi\right] = \mathbb{P}_{\textrm{Uniform}}\left[\ttau_k = \pi\right] = 1/n! \ .
\end{align*}
Hence, we infer that,
\begin{align*}
    \bE_{\tau_k}[|1/i\sum_{j=1}^{i}\vw_{\ttau_k(j)}(\vz^*)|^2] = \bE_{\ttau_k}[|1/i\sum_{j=1}^{i}\vw_{\ttau_k(j)}(\vz^*)|^2] = \frac{n-i}{i(n-1)}\sigma_*^2,
\end{align*}
where the last equality follows from Lemma \ref{app-lem:sample-mean-wor}. Note that $\sigma_*^{2}$ denotes the gradient variance at the minimax point $\vz^*$, defined as $\sigma_*^2 = 1/n\sum_{i=1}^{n}\norm{\vw_{i}(\vz^*)}^2$. Substituting the above into the upper bound for $\bE[\norm{\vr_k}^2]$ yields,
\begin{align*}
    \bE[\norm{\vr_k}^2]
    &\leq l^2 e^{2/5}(n-1)\sum_{i=1}^{n-1}i^2 \bE_{\tau_k}[|1/i\sum_{j=1}^{i}\vw_{\ttau_k(j)}(\vz^*)|^2] \\
    &\leq l^2 e^{2/5}(n-1) \sum_{i=1}^{n-1}i^2 \frac{n-i}{i(n-1)}\sigma_*^2 = l^2 e^{2/5} \sigma_*^2 \sum_{i=1}^{n-1}i(n-i) \\
    &\leq l^2 n^3 \sigma_*^{2}/4.
\end{align*}

\paragraph{Step 3B: $\bE[\norm{\vr_k}^2]$ for PPM-SO}
For PPM-SO, $\tau_1 = \ldots = \tau_K = \tau$ where $\tau$ is a uniformly sampled random permutation. Consequently, $\ttau_1 = \ldots = \ttau_K = \ttau$ where $\ttau$ is the transposed permutation of $\tau$. Hence, we conclude that,
\begin{align*}
    \norm{\vr_k}^2 \leq l^2 e^{2/5}(n-1)\sum_{i=1}^{n-1}i^2 |1/i\sum_{j=1}^{i}\vw_{\ttau(j)}(\vz^*)|^2.
\end{align*}
Proceeding along the same lines as Step 3A, we use the fact that $\tau$ is a uniformly sampled random permutation to obtain the following:
\begin{align*}
    \E{\norm{\vr_k}^2} &\leq l^2 e^{2/5}(n-1)\sum_{i=1}^{n-1}i^2 \bE_{\tau}[ |1/i\sum_{j=1}^{i}\vw_{\ttau(j)}(\vz^*)|^2] \\
    &= l^2 e^{2/5}(n-1)\sum_{i=1}^{n-1}i^2 \bE_{\ttau}[ |1/i\sum_{j=1}^{i}\vw_{\ttau(j)}(\vz^*)|^2] \\
    &\leq l^2 e^{2/5}(n-1) \sum_{i=1}^{n-1}i^2 \frac{n-i}{i(n-1)}\sigma_*^2 \\
    &\leq l^2 n^3 \sigma_*^{2}/4.
\end{align*}
Hence, for both PPM-RR and PPM-SO, $\norm{\vr_k}$ is bounded in expectation as $\E{\norm{\vr_k}^2} \leq l^2 n^3 \sigma_*^{2}/4$.

\paragraph{Step 4: Convergence analysis}
Since $\vH_k$ is invertible, we write \eqref{app-eqn:ppm-wor-epoch-update} as
\begin{align*}
    \vz^{k+1}_0 - \vz^* = \vH^{-1}_k[\vz^k_0 - \vz^* - \alpha^2 \vr_k].
\end{align*}

Unrolling for $K$ epochs and setting $\vz^1_0 = \vz_0$ yields,
\begin{align*}
    \vz^{K+1}_0 - \vz^* = \left[\revprod_{k=1}^{K}\vH^{-1}_k\right](\vz_0 - \vz^*) - \alpha^2 \sum_{k=1}^{K}\left[\revprod_{j=k}^{K}\vH^{-1}_j\right]\vr_k.
\end{align*}
By the triangle inequality and the bound $\norm{\vH^{-1}_k} \leq (1 - n \alpha \mu / 2)$, we obtain:
\begin{align*}
    \norm{\vz^{K+1}_0 - \vz^*} &\leq  \left[\revprod_{k=1}^{K}\norm{\vH^{-1}_k}\right]\norm{\vz_0 - \vz^*} + \alpha^2 \sum_{k=1}^{K}\left[\revprod_{j=k}^{K}\norm{\vH^{-1}_j}\right]\norm{\vr_k} \\
    &\leq (1 - n \alpha \mu /2)^K \norm{\vz_0 - \vz^*} + \alpha^2 \sum_{k=1}^{K}(1 - n\alpha\mu/2)^{K-k}\norm{\vr_k},
\end{align*}
and as a result of Young's inequality, 
\begin{align*}
    \norm{\vz^{K+1}_0 - \vz^*}^2 &\leq 2(1 - n\alpha\mu/2)^{2K}\norm{\vz_0 - \vz^*}^2 + 2 \alpha^4 K \sum_{k=1}^{K} (1-n\alpha \mu/2)^{2(K-k)}|\vr_k|^2.
\end{align*}
By taking expectations (with respect to the uniform random permutations $\tau_1, \ldots, \tau_K$ for RR and $\tau$ for SO respectively) on both sides of the above inequality and substituting the upper bound for $\bE[|\vr_k|{}^2]$ for PPM-RR and PPM-SO as derived in Step 3, we obtain
\begin{align*}
    \E{\norm{\vz^{K+1}_0 - \vz^*}^2} &\leq 2e^{-n\alpha \mu K}\norm{\vz_0 - \vz^*}^2 + 2 \alpha^4 K \sum_{k=1}^K (1 - n \alpha \mu/2)^{2(K-k)} \E{\norm{\vr_k}^2} \\
    &\leq 2e^{-n\alpha \mu K}\norm{\vz_0 - \vz^*}^2 + 0.5l^2 n^3 \sigma_*^2  \alpha^4 K \sum_{k=1}^K (1 - n \alpha \mu/2)^{2(K-k)}.
\end{align*}
This yields the following guarantee that holds for PPM-RR and PPM-SO whenever $\alpha \leq \frac{\mu}{5nl^2}$,
\begin{equation}
\label{app-eqn:ppm-wor-alpha-rate}
    \E{\norm{\vz^{K+1}_0 - \vz^*}^2} \leq 2e^{-n\alpha \mu K}\norm{\vz_0 - \vz^*}^2 + \frac{l^2 \sigma_*^2 \alpha^3 n^2 K}{\mu}.
\end{equation}
We note that the above inequality is identical to the inequality \eqref{app-eqn:gda-wor-alpha-rate} obtained in Theorem \ref{app-thm:gda-wor-convergence}. Hence, we substitute $\alpha = \min \{  \mu/5nl^2, 2 \log (\norm{\nu(\vz_0 )}n^{1/2}K/\mu )/\mu n K \}$ in \eqref{app-eqn:ppm-wor-alpha-rate} and follow the same steps as in Theorem \ref{app-thm:gda-wor-convergence} to obtain the following last-iterate convergence guarantee for both PPM-RR and PPM-SO, which holds for any $K \geq 1$:
\begin{equation}
\label{app-eqn:ppm-wor-const-convergence-rate}
\E{\norm{\vz^{K+1}_0 - \vz^*}^2} \leq 2e^{-\frac{K}{5\kappa^2}}\norm{\vz_0 - \vz^*}^2 +  \frac{2\mu^2 + 8\kappa^2 \sigma_*^2\log^3\left(\norm{\nu(\vz_0 )}n^{1/2}K/\mu \right)}{\mu^2 n K^2}.
\end{equation}
Suppressing constant terms and logarithmic factors, we obtain the following:
\begin{equation}
\label{app-eqn:ppm-wor-convergence-rate}
\E{\norm{\vz^{K+1}_0 - \vz^*}^2} = \Tilde{O}(e^{\nicefrac{-K}{5\kappa^2}} + \nicefrac{1}{nK^2}).
\end{equation}
\end{proof}
\subsection{Analysis of PPM-AS}
\label{app-sec:ppm-as}
\begin{theorem}[Convergence of PPM-AS]
\label{app-thm:ppm-as-convergence}
Consider Problem \eqref{p:strong-monotone-vi} for the $\mu$-strongly monotone operator $\nu(\vz) = \nicefrac{1}{n} \sum_{i=1}^{n} \omega_i(\vz)$ where each $\omega_i$ is $l$-Lipschitz, but not necessarily monotone. Let $\vz^*$ denote the unique root of $\nu$. Then, for any $\alpha \leq \frac{\mu}{5nl^2}$ and $K \geq 1$, the iterates of PPM-AS satisfy the following:
\begin{equation*}
    \max_{\tau_1, \ldots, \tau_K \in \mathbb{S}_n} \norm{\vz^{K+1}_0 - \vz^*}^2 \leq 2e^{-n\alpha \mu K}\norm{\vz_0 - \vz^*}^2 + \frac{3l^2 \sigma_*^2 \alpha^3 n^3 K}{ \mu}.
\end{equation*}
Setting $\alpha = \min \{  \mu/5nl^2, 2 \log (\norm{\nu(\vz_0 )}K/\mu )/\mu n K \}$ results in the following last-iterate convergence guarantee for PPM-AS, which holds for any $K \geq 1$:
\begin{align*}
\max_{\tau_1, \ldots, \tau_K \in \mathbb{S}_n} \norm{\vz^{K+1}_0 - \vz^*}^2 &\leq 2e^{\nicefrac{-K}{5\kappa^2}}|\vz_0 - \vz^*|^2  +  \frac{2\mu^2  +  24 \kappa^2 \sigma^2_* \log^3 (|\nu(\vz_0)|K/\mu)}{\mu^2 K^2} \\
&= \Tilde{O}(e^{\nicefrac{-K}{5\kappa^2}}  +  \nicefrac{1}{K^2}).
\end{align*}
\end{theorem}
\begin{proof}
The iterate-level update rule of PPM-AS can be expressed as
\begin{equation}
\label{app-eqn:ppm-as-itr}
    \vz^{k}_{i} = \vz^{k}_{i-1} - \alpha \vw_{\tau_k(i)}(\vz^k_{i}),
\end{equation}
where $k \in [K]$, $i \in [n]$ and $\tau_k$ denotes an arbitrary permutation of $[n]$ which is chosen by the adversary at the start of epoch $k$. For any permutation $\tau_k$ chosen by the adversary, we denote the transposed permutation $\ttau_k$ as $\ttau_k(i) = \tau_k(n+1-i)$. Furthermore $\vz^k_i$ is well defined and unique due to the Banach Fixed Point Theorem as $\alpha \leq \mu/5nl^2 < 1/l$.

We recall that Steps 1 and 2 of Theorem \ref{app-thm:ppm-wor-convergence} are applicable for any permutation $\tau_k \in \mathbb{S}_n$, and hence, can be applied directly to PPM-AS. Consequently, the epoch-level update rule of PPM-AS is the same as that of PPM-RR/SO and is given by:
\begin{equation}
\label{app-eqn:ppm-as-epoch-update}
    \vz^{k}_0 - \vz^* = \vH_k (\vz^{k+1}_0 - \vz^*) + \alpha^2 \vr_k,
\end{equation}
where $\vH_k$ and $\vr_k$ are as defined in Theorem \ref{app-thm:ppm-wor-convergence}, and are given by:
\begin{align*}
    \vH_k &= \vI + \alpha \sum_{j=1}^{n} \left(\revprod_{t=j+1}^{n}(\vI + \alpha \vJ_{\ttau_k(t)})\right)\vM_{\ttau_k(j)}, \\
    \vr_k &= \sum_{i=1}^{n-1}\left[\revprod_{t=i+2}^{n}(\vI + \alpha \vJ_{\ttau_k(t)}) \right]\vJ_{\ttau_k(i+1)}\sum_{j=1}^{i}\vw_{\ttau_k(j)}(\vz^*),
\end{align*}
where, as before, $\vM_{\tau_k(i)} = \int_{0}^{1} \grad \omega_{\tau_k(i)}(t\vz^{k+1}_0 + (1 - t) \vz^*) \textrm{d}t$ and $\vJ_{\tau_k(i)} = \int_{0}^{1} \grad \omega_{\tau_k(i)}(t\vz^k_i + (1 - t) \vz^{k+1}_{0}) \textrm{d}t$. Furthermore, by Step 2 of Theorem \ref{app-thm:ppm-wor-convergence}, $\vH_k$ is invertible and $\norm{\vH^{-1}_{k}} \leq 1 - n\alpha \mu / 2$.

We now proceed to control the noise term $\vr_k$ by deriving an upper bound for $\norm{\vr_k}^2$ that holds uniformly over all permutations $\tau_1, \ldots, \tau_K \in \mathbb{S}_n$. This is done by following the same steps as that of Theorem \ref{app-thm:gda-as-convergence}, which gives us the following bound by successive applications of $\norm{\vJ_{\ttau_k(i)}} \leq l$ and Young's inequality:

\begin{align*}
    \norm{\vr_k}^2 &\leq l^2 e^{2/5}\left(\sum_{i=1}^{n-1}|\sum_{j=1}^{i}\vw_{\ttau_k(j)}(\vz^*)|\right)^2  \\ 
    &\leq l^2 e^{2/5}(n-1)\sum_{i=1}^{n-1}|\sum_{j=1}^{i}\vw_{\ttau_k(j)}(\vz^*)|^2 \\
    &\leq l^2 e^{2/5}(n-1)\sum_{i=1}^{n-1} i \sum_{j=1}^{i}|\vw_{\ttau_k(j)}(\vz^*)|^2 \\
    &\leq l^2 e^{2/5}(n-1)\sum_{i=1}^{n-1} i \sum_{j=1}^{n}|\vw_{\ttau_k(j)}(\vz^*)|^2 = l^2 e^{2/5}(n-1)\sum_{i=1}^{n-1} i \sum_{j=1}^{n}|\vw_{j}(\vz^*)|^2 \\
    &= l^2 e^{2/5}(n-1)(n \sigma^2_*)\sum_{i=1}^{n-1} i
    \leq 3 l^2 n^4 \sigma_*^2 / 4.
\end{align*}
We note that the above sequence of inequalities hold for any permutation $\tau_k$, and hence, the upper bound $\norm{\vr_k}^2 \leq 3 l^2 n^4 \sigma_*^2 / 4$ holds for any $\tau_k \in \mathbb{S}_n, \ k \in [K]$.

Since $\vH_k$ is invertible, we write \eqref{app-eqn:ppm-as-epoch-update} as
\begin{align*}
    \vz^{k+1}_0 - \vz^* = \vH^{-1}_k[\vz^k_0 - \vz^* - \alpha^2 \vr_k].
\end{align*}

By unrolling for $K$ epochs and setting $\vz^1_0 = \vz_0$, we obtain:
\begin{align*}
    \vz^{K+1}_0 - \vz^* = \left[\revprod_{k=1}^{K}\vH^{-1}_k\right](\vz_0 - \vz^*) - \alpha^2 \sum_{k=1}^{K}\left[\revprod_{j=k}^{K}\vH^{-1}_j\right]\vr_k.
\end{align*}
By using the triangle inequality and the bound $\norm{\vH^{-1}_k} \leq (1 - n \alpha \mu / 2)$, we get:
\begin{align*}
    \norm{\vz^{K+1}_0 - \vz^*} &\leq  \left[\revprod_{k=1}^{K}\norm{\vH^{-1}_k}\right]\norm{\vz_0 - \vz^*} + \alpha^2 \sum_{k=1}^{K}\left[\revprod_{j=k}^{K}\norm{\vH^{-1}_j}\right]\norm{\vr_k} \\
    &\leq (1 - n \alpha \mu /2)^K \norm{\vz_0 - \vz^*} + \alpha^2 \sum_{k=1}^{K}(1 - n\alpha\mu/2)^{K-k}\norm{\vr_k}.
\end{align*}

Applying Young's inequality yields the following: 
\begin{align*}
    \norm{\vz^{K+1}_0 - \vz^*}^2 &\leq 2(1 - n\alpha\mu/2)^{2K}\norm{\vz_0 - \vz^*}^2 + 2 \alpha^4 K \sum_{k=1}^{K} (1-n\alpha \mu/2)^{2(K-k)}|\vr_k|^2 \\
    &\leq 2e^{-n\alpha \mu K}\norm{\vz_0 - \vz^*}^2 + 2 \alpha^4 K \sum_{k=1}^{K} (1 - n\alpha \mu/2)^{2(K-k)} |\vr_k|^2 \\
    &\leq 2e^{-n\alpha \mu K}\norm{\vz_0 - \vz^*}^2 + (1.5 l^2 n^4 \sigma^2_*) \alpha^4 K \sum_{k=1}^{K} (1 - n\alpha \mu/2)^{2(K-k)} \\
    &\leq 2e^{-n\alpha \mu K}\norm{\vz_0 - \vz^*}^2 + \frac{3 l^2 \sigma^2_* n^3 \alpha^3 K}{\mu}.
\end{align*}
Since the above inequality holds for any sequence of permutations $\tau_1, \ldots, \tau_K$ chosen by the adversary, we obtain the following uniform guarantee for GDA-AS whenever $\alpha \leq \frac{\mu}{5nl^2}$,
\begin{equation}
\label{app-eqn:ppm-as-alpha-rate}
    \max_{\tau_1, \ldots, \tau_K \in \mathbb{S}_n} \norm{\vz^{K+1}_0 - \vz^*}^2 \leq 2e^{-n\alpha \mu K}\norm{\vz_0 - \vz^*}^2 + \frac{3l^2 \sigma_*^2 \alpha^3 n^3 K}{ \mu}.
\end{equation}
We note that the above inequality is exactly identical to the inequality \eqref{app-eqn:gda-as-alpha-rate} obtained in Theorem \ref{app-thm:gda-as-convergence}. Hence, we substitute $\alpha = \min \{  \mu/5nl^2, 2 \log (\norm{\nu(\vz_0 )}K/\mu )/\mu n K \}$ in \eqref{app-eqn:ppm-as-alpha-rate} and follow the same steps as in Theorem \ref{app-thm:gda-as-convergence} to obtain the following last-iterate convergence guarantee for PPM-AS, which holds for any $K \geq 1$:
\begin{equation}
\label{app-eqn:ppm-as-const-convergence-rate}
\max_{\tau_1, \ldots, \tau_K \in \mathbb{S}_n} \norm{\vz^{K+1}_0 - \vz^*}^2 \leq 2e^{\nicefrac{-K}{5\kappa^2}}|\vz_0 - \vz^*|^2  +  \frac{2\mu^2  +  24 \kappa^2 \sigma^2_* \log^3 (|\nu(\vz_0)|K/\mu)}{\mu^2 K^2}.
\end{equation}
Suppressing constant terms and logarithmic factors, we get,
\begin{equation}
\label{app-eqn:ppm-as-convergence-rate}
\max_{\tau_1, \ldots, \tau_K \in \mathbb{S}_n} \norm{\vz^{K+1}_0 - \vz^*}^2 = \Tilde{O}(e^{\nicefrac{-K}{5\kappa^2}}  +  \nicefrac{1}{K^2}).
\end{equation}
\end{proof}
\section{Analysis of AGDA-RR and AGDA-AS}
\label{app-sec:agda}
In this section, we present the convergence analysis of AGDA-RR and AGDA-AS for smooth finite-sum objectives that satisfy a two-sided P\L{} inequality. We begin by stating some useful properties of two-sided P\L{} functions which have been established in prior works and are crucial to our analysis. We then present our convergence analysis of AGDA-RR and AGDA-AS, which follows the same broad structure as our analysis of GDA-RR/AS. 
\subsection{Properties of Two-sided P\L{} Functions}
\label{app-sec:agda-lemmas}
We begin by presenting three important properties of two-sided P\L{} functions that are used throughout our analysis of AGDA-RR/AS
\begin{lemma}[Equivalent optimality conditions for two-sided P\L{} functions]
\label{app-lem:2pl-eq-optimal}
Let $F : \bR^{d_\vx} \times \bR^{d_\vy} \rightarrow \bR$ satisfy Assumption \ref{as:2pl} and let $(\vx^*, \vy^*)$ be any point in $\bR^{d_\vx} \times \bR^{d_\vy}$. Then the following conditions are equivalent:
\begin{itemize}
    \item $(\vx^*, \vy^*)$ is a global minimax point, i.e., $\vx^* \in \arg \min_{\vx \in \bR^{d_{\vx}}} \Phi(\vx)$ where $\Phi$ is the best response function defined by $\Phi(\vx) = \max_{\vy \in \bR^{d_{\vy}}} F(\vx, \vy)$, and $\vy^* \in \arg \max_{\vy \in \bR^{d_{\vy}}} F(\vx^*, \vy)$.

    \item  $(\vx^*, \vy^*)$ is a saddle point, i.e., 
    \begin{align*}
        F(\vx^{*}, \vy) \leq F(\vx^{*}, \vy^{*}) \leq F(\vx, \vy^{*}), \ \forall \vx \in \bR^{d_\vx}, \vy \in \bR^{d_\vy}.
    \end{align*}
    \item $(\vx^*, \vy^*)$ is a stationary point, i.e., 
    \begin{align*}
        \grad_\vx F(\vx^*, \vy^*) = \grad_\vy F(\vx^*, \vy^*) = 0.
    \end{align*}
\end{itemize}
\end{lemma}
\begin{proof}
See Lemma 2.1 of \citet{YangAGDA2020}.
\end{proof}
\begin{lemma}[Properties of the best response]\label{app-lem:agda-phi-prop}
Let $F : \bR^{d_\vx} \times \bR^{d_\vy} \rightarrow \bR$ be an $l$-smooth function satisfying Assumption \ref{as:2pl} (with constants $\mu_1$ and $\mu_2$) and let $\Phi : \bR^{d_\vx} \rightarrow \bR$ be the best response function defined as $\Phi(\vx) = \max_{\vy \in \bR^{d_\vy}} F(\vx, \vy)$. Then $\Phi$ satisfies the following properties:
\begin{itemize}
    \item $\Phi$ is differentiable and $L$-smooth where $L = l + l^2/\mu_2$.
    \item $\grad \Phi(\vx) = \grad_\vx F(\vx, \vy^*(\vx))$ where $\vy^{*}(\vx)$ is any arbitrary point in $\arg \max_{\vy \in \bR^{d_\vy}} F(\vx, \vy)$. 
    \item $\Phi$ satisfies the (one-sided) P\L{} inequality with constant $\mu_1$, i.e., 
    \begin{align*}
        \norm{\grad \Phi(\vx)}^2 \geq 2 \mu_1 [\Phi(\vx) - \Phi^*],
    \end{align*}
    where $\Phi^* = \min_{\vx \in \bR^{d_{\vx}}} \Phi(\vx)$.
\end{itemize}
\end{lemma}
\begin{proof}
See Lemma A.2 and Lemma A.3 of \citet{YangAGDA2020}.
\end{proof}
\begin{lemma}[Quadratic growth properties of P\L{} functions]
\label{as:pl-quad-growth}
Let $f : \bR^{d_{\vx}} \rightarrow \bR$ be a P\L{} function, i.e., there exists a positive constant $\mu$ such that $\norm{\grad f(\vx)}^2 \geq 2 \mu [f(\vx) - f^*]$, where $f^* = \min_{\vx} f(\vx)$. Then, $f$ satisfies the quadratic growth property with constant $\mu$, i.e., for any $\vx \in \bR^{d_{\vx}}$, $f(\vx) - f^* \geq \frac{\mu}{2}\norm{\vx^* - \vx}^2$ where $\vx^*$ is the projection of $\vx$ on the set $\arg \min_{\vx} f(\vx)$.
\end{lemma}
\begin{proof}
See Lemma A.1 of \citet{YangAGDA2020}.
\end{proof}
Finally, we motivate the necessity of the bounded gradient variance assumption by presenting the construction of a two-sided P\L{} function whose set of global minimax points is unbounded.
\begin{lemma}[2P\L{} function with unbounded set of global minimax points]
Consider the function $f : \bR^2 \times \bR^2 \rightarrow \bR$ given by $f(\vx, \vy) = (\vx_1 + \vx_2)^2/2 - (\vy_1 + \vy_2)^2/2$. Then, $f$ is a two-sided P\L{} function with constants $\mu_1 = \mu_2 = 2$. Furthermore, the set of global minimax points of $F$ is an unbounded proper subset of $\bR^4$.
\end{lemma}
\begin{proof}
We define the functions $\Phi : \bR^2 \rightarrow \bR$ and $\Psi : \bR^2 \rightarrow \bR$ as follows:
\begin{align*}
\Phi(\vx) &= \max_{\vy \in \bR^2} f(\vx, \vy) = \frac{(\vx_1 + \vx_2)^2}{2}, \\
\Psi(\vy) &= \min_{\vx \in \bR^2} f(\vx, \vy) = - \frac{(\vy_1 + \vy_2)^2}{2}.
\end{align*}
Furthermore, we note that:
\begin{align*}
\grad_{\vx} f(\vx, \vy) &= [\vx_1 + \vx_2, \vx_1 + \vx_2], \\
\grad_{\vy} f(\vx, \vy) &= -[\vy_1 + \vy_2, \vy_1 + \vy_2].
\end{align*}
From the above inequalities, it follows that:
\begin{align*}
\norm{\grad_{\vx} f(\vx, \vy)}^2 &= 4[f(\vx, \vy) - \Psi(\vy)] = 4[f(\vx, \vy) - \min_{\vx \in \bR^2} f(\vx, \vy)], \\
\norm{\grad_{\vy} f(\vx, \vy)}^2 &= 4[\Phi(\vx) - f(\vx, \vy)] = 4[\max_{\vy \in \bR^2} f(\vx, \vy) - f(\vx, \vy)].
\end{align*}
Thus, $f$ satisfies the two-sided P\L{} inequality with constants $\mu_1 = \mu_2 = 2$. Consequently, Lemma \ref{app-lem:2pl-eq-optimal} implies that the set of minimax points are exactly the ones satisfying $\grad_{\vx} f(\vx, \vy) = \grad_{\vy} f(\vx, \vy) = 0$, i.e., the set of global minimax points is $\{ (\vx, \vy) \in \bR^4 \ | \ \vx_1 + \vx_2 = 0, \vy_1 + \vy_2 = 0 \}$, which is an unbounded proper subset of $\bR^4$.
\end{proof}
\subsection{Analysis of AGDA-RR}
\label{app-sec:agda-rr}
\begin{theorem}[Convergence of AGDA-RR]
\label{app-thm:agda-rr-convergence}
Let Assumptions \ref{as:comp-smooth}, \ref{as:2pl}, and \ref{as:bgv} be satisfied and let $\eta = \nicefrac{73l^2}{2\mu^2_2}$. Then, for any $\alpha \leq \nicefrac{1}{5\eta n l}$,  $\beta = \eta \alpha$, and $K \geq 1$, the iterates of AGDA-RR satisfy the following:
$$\bE[V_{\lambda}(\vz^{K+1}_0)] \leq (1 - \mu_1 n \alpha/2)^K V_{\lambda}(\vz_0) + 0.42 n \alpha^2 l^2 \sigma^2(3 + \eta^3)/\mu_1.$$
Setting $\alpha = \min \{ 1/5\eta n l, 4 \log(V_{\lambda}(\vz_0)n^{1/2}K)/\mu_1 nK \}$ results in the following expected last iterate convergence guarantee, which holds for any $K \geq 1$:
\begin{align*}
    \mathbb{E}[V_{\lambda}(\vz^{K+1}_0)] &\leq e^{\nicefrac{-K}{365\kappa^3}}V_{\lambda}(\vz_0) + \frac{\mu_1 + c\kappa^8\sigma^2 \log^2(V_{\lambda}(\vz_0)n^{1/2}K)}{\mu_1 nK^2} \\
    &= \Tilde{O}(e^{\nicefrac{-K}{365\kappa^3}} + \nicefrac{1}{nK^2}),
\end{align*}
where $\kappa = \max \{ \nicefrac{l}{\mu_1}, \nicefrac{l}{\mu_2} \}$ and $c > 0$ is a numerical constant that is independent of $l, \mu_1, \mu_2$ and $\sigma^2$.
\end{theorem}
\begin{proof}
Before presenting the convergence analysis of AGDA-RR, we introduce some additional notation. For any $k \in [K]$ and any random variable $Y$, we use $\bE_k[Y]$ to denote the expectation of $Y$ conditioned on all previous epoch iterates $(\vx^j_0, \vy^j_0)_{j=1}^{k}$, i.e., $\bE_k[Y] = \bE[Y \ | \ \vx^1_0, \vy^1_0, \ldots, \vx^k_0, \vy^k_0 ]$. Similarly, for any two random variables $Y, Z$, we use $\bE_k[Y \ | \ Z]$ to denote the expectation of $Y$ conditioned on all previous epoch iterates $(\vx^j_0, \vy^j_0)_{j=1}^{k}$ and $Z$, i.e. $\bE_k[Y \ | \ Z] = \bE[Y \ | \ \vx^1_0, \vy^1_0, \ldots, \vx^k_0, \vy^k_0, Z ]$.

Our analysis follows a structure similar to that of GDA-RR/SO and consists of three major steps: 1. obtaining an epoch level update rule for AGDA-RR that resembles the update of deterministic AGDA with noise, 2. bounding the noise in the update rule using Lemma \ref{app-lem:sample-mean-wor},  3. presenting a convergence analysis with respect to the Lyapunov function $V_{\lambda}$.

\textbf{Step 1: Epoch level update rule}
We begin with the observation that the $\vx$-iterates of AGDA-RR are the same as that of gradient descent without replacement on $F(\vx, \vy^k_0)$ and thus, can be treated as a noisy version of full-batch gradient descent. To this end, we linearize the stochastic gradients $\grad_{\vx} f_{\tau_k(i)}(\vx^k_{i-1}, \vy^k_0)$ about $(\vx^k_0, \vy^k_0)$ and obtain the following decomposition:
\begin{align*}
    \grad_{\vx} f_{\tau_k(i)}(\vx^{k}_{i-1}, \vy^k_0) &= \grad_{\vx} f_{\tau_k(i)}(\vx^{k}_{0}, \vy^{k}_{0}) + \left[\int_{0}^{1} \grad^2_{\vx \vx} f_{\tau_k(i)}(\vx^k_0 + t(\vx^{k}_{i-1} - \vx^k_0), \vy^k_0) \textrm{d}t\right](\vx^{k}_{i-1} - \vx^{k}_0) \\
    &= \grad_{\vx} f_{\tau_k(i)}(\vx^{k}_{0}, \vy^{k}_{0}) + \vH_{\tau_k(i)}(\vx^{k}_{i-1} - \vx^{k}_0),
\end{align*}
where $\vH_{\tau_k(i)} = \int_{0}^{1} \grad^2_{\vx \vx} f_{\tau_k(i)}(\vx^k_0 + t(\vx^{k}_{i-1} - \vx^k_0), \vy^k_0) \textrm{d}t$. As before, by Rademacher's Theorem, the $l$-smoothness of $f_{\tau_k(i)}$ guarantees that $\vH_{\tau_k(i)}$ is well defined and satisfies $\norm{\vH_{\tau_k(i)}} \leq l$.

When substituting the above linearization into the $\vx$-iterate updates, we observe the following:
\begin{align*}
    \vx^{k}_1 &= \vx^{k}_0 - \alpha \grad_{\vx}f_{\tau_k(1)}(\vx^{k}_{0}, \vy^{k}_{0}), \\
    \vx^{k}_2 &= \vx^{k}_1 - \alpha \grad_{\vx}f_{\tau_k(2)}(\vx^{k}_{1}, \vy^{k}_{0}) = \vx^{k}_1 - \alpha \grad_{\vx}f_{\tau_k(2)}(\vx^k_0, \vy^k_0) - \alpha \vH_{\tau_k(2)}(\vx^k_1 - \vx^k_0) \\
    &= \vx^k_0 - \alpha \grad_{\vx} f_{\tau_k(2)}(\vx^k_0, \vy^k_0) - \alpha \grad_{\vx}f_{\tau_k(1)}(\vx^k_0, \vy^k_0) + \alpha^2 \vH_{\tau_k(2)} \grad_{\vx}f_{\tau_k(1)}(\vx^k_0, \vy^k_0) \\
    &= \vx^k_0 - \alpha (\vI - \alpha \vH_{\tau_k(2)})\grad_{\vx}f_{\tau_k(1)}(\vx^k_0, \vy^k_0) - \alpha \grad_{\vx}f_{\tau_k(2)}(\vx^k_0, \vy^k_0).
\end{align*}
Repeating the above steps for subsequent iterates $\vx^k_3, \ldots, \vx^k_n$, we obtain the following epoch-level update rule for $\vx^{k+1}_0 = \vx^k_n$:
\begin{equation}
\label{app-eqn:agda-rr-x-unroll}
    \vx^{k+1}_0 = \vx^k_0 - \alpha \sum_{j=1}^{n} \left( \revprod_{t=j+1}^{n} (\vI - \alpha \vH_{\tau_k(t)}) \right) \grad_{\vx}f_{\tau_k(j)}(\vx^k_0, \vy^k_0).
\end{equation}
We further simplify \eqref{app-eqn:agda-rr-x-unroll} using the summation by parts identity. To this end, we define $c_j$ and $e_j$ as follows:
\begin{align*}
    c_j &= \revprod_{t=j+1}^{n} (\vI - \alpha \vH_{\tau_k(t)}), \\
    e_j &= \grad_{\vx}f_{\tau_k(j)}(\vx^k_0, \vy^k_0).
\end{align*}
We note that $c_n = \vI$, since $\revprod_{t=n+1}^{n} (\vI - \alpha \vH_{\tau_k(t)})$ denotes the empty reverse product. Hence, we conclude,
\begin{align*}
    c_n \sum_{j=1}^{n}e_j &= \sum_{j=1}^{n}\grad_{\vx}f_{\tau_k(j)}(\vx^k_0, \vy^k_0) = n \grad_{\vx} F(\vx^k_0, \vy^k_0), \\
    \sum_{i=1}^{n-1}(c_{i+1}-c_{i})\sum_{j=1}^{i}e_j &= \sum_{i=1}^{n-1} \left[\revprod_{t=i+2}^{n} (\vI - \alpha \vH_{\tau_k(t)}) - \revprod_{t=i+1}^{n} (\vI - \alpha \vH_{\tau_k(t)}) \right]\sum_{j=1}^{i} \grad_{\vx} f_{\tau_k(j)}(\vx^k_0, \vy^k_0) \\
    &= \alpha \sum_{i=1}^{n-1}\left[\revprod_{t=i+2}^{n} (\vI - \alpha \vH_{\tau_k(t)}) \right] \vH_{\tau_k(i+1)}\sum_{j=1}^{i} \grad_{\vx} f_{\tau_k(j)}(\vx^k_0, \vy^k_0).
\end{align*}
We define $\vr_k$ as,
\begin{equation}
\label{app-eqn:agda-rr-rk}
    \vr_k = \sum_{i=1}^{n-1}\left[\revprod_{t=i+2}^{n} (\vI - \alpha \vH_{\tau_k(t)}) \right] \vH_{\tau_k(i+1)}\sum_{j=1}^{i} \grad_{\vx} f_{\tau_k(j)}(\vx^k_0, \vy^k_0).
\end{equation}
From the summation by parts identity $\sum_{i=1}^{n} c_i e_i = c_n \sum_{j=1}^{n}e_j - \sum_{i=1}^{n-1}(c_{i+1}-c_{i})\sum_{j=1}^{i}e_j$, we obtain the following:
\begin{equation*}
    \sum_{j=1}^{n} \left[ \revprod_{t=j+1}^{n} [\vI - \alpha \vH_{\tau_k(t)}] \right] \grad_{\vx}f_{\tau_k(j)}(\vx^k_0, \vy^k_0) = n \grad_{\vx} F(\vx^k_0, \vy^k_0) - \alpha \vr_k.
\end{equation*}
By substituting the above equation into  \eqref{app-eqn:agda-rr-x-unroll}, we obtain the following update rule for the epoch iterates $\vx^k_0$:
\begin{equation}
\label{app-eqn:agda-rr-x-update}
    \vx^{k+1}_0 = \vx^k_0 - n\alpha \grad_{\vx} F(\vx^k_0, \vy^k_0) + \alpha^2 \vr_k.
\end{equation}
The epoch-level update rule for the $\vy$-iterate is derived in a similar fashion. We begin by linearizing $\grad_{\vy} f_{\pi_k(i)}(\vx^{k+1}_0, \vy^k_{i-1})$ as:

\begin{align*}
    \grad_{\vy} f_{\pi_k(i)}(\vx^{k+1}_0, \vy^k_{i-1}) = \grad_{\vy} f(\vx^{k+1}_0, \vy^k_{0}) + \vJ_{\pi_k(i)}(\vy^k_{i-1}-\vy^k_0), 
\end{align*}
where $\vJ_{\pi_k(i)} = \int_{0}^{1} \grad^2_{\vy \vy} f_{\pi_k(i)}(\vx^{k+1}_0,\vy^k_0 + t(\vy^{k}_{i-1} - \vy^k_0)) \textrm{d}t$ satisfies $\norm{\vJ_{\pi_k(i)}} \leq l$. As before, by substituting into the $\vy$ iterate updates and unrolling, we obtain,
\begin{equation}
\label{app-eqn:agda-rr-y-unroll}
    \vy^{k+1}_0 = \vy^k_0 + \beta \sum_{j=1}^{n} \left[ \revprod_{t=j+1}^{n} [\vI + \beta \vJ_{\pi_k(t)}] \right] \grad_{\vy}f_{\pi_k(j)}(\vx^{k+1}_0, \vy^k_0).
\end{equation}
Once again, applying the summation by parts technique gives us the following update rule for the epoch iterates $\vy^{k}_0$:
\begin{equation}
\label{app-eqn:agda-rr-y-update}
\vy^{k+1}_0 = \vy^k_0 + n \beta \grad_{\vy} F(\vx^{k+1}_0, \vy^k_0) + \beta^2 \vs_k,
\end{equation}
where $\vs_k$ is given by,
\begin{equation}
\label{app-eqn:agda-rr-sk}
    \vs_k = \sum_{i=1}^{n-1}\left[\revprod_{t=i+2}^{n} (\vI + \beta \vJ_{\pi_k(t)}) \right] \vJ_{\pi_k(i+1)}\sum_{j=1}^{i} \grad_{\vy} f_{\pi_k(j)}(\vx^{k+1}_0, \vy^k_0).
\end{equation}
\textbf{Step 2: Controlling the noise terms}
By applying the triangle inequality on \eqref{app-eqn:agda-rr-rk} and substituting $\norm{\vH_{\tau_k}(i)} \leq l$ and $\alpha \leq 1/5nl$ wherever required, we bound $\norm{\vr_k}$ as follows:
\begin{align*}
    \norm{\vr_k} &\leq \sum_{i=1}^{n-1}\left[\revprod_{t=i+2}^{n} \norm{\vI - \alpha \vH_{\tau_k(t)}} \right] \norm{\vH_{\tau_k(i+1)}}|\sum_{j=1}^{i} \grad_{\vx} f_{\tau_k(j)}(\vx^k_0, \vy^k_0)| \\
    &\leq l(1 + \alpha l)^n \sum_{i=1}^{n-1} |\sum_{j=1}^{i} \grad_{\vx} f_{\tau_k(j)}(\vx^k_0, \vy^k_0)| \leq le^{1/5} \sum_{i=1}^{n-1} |\sum_{j=1}^{i} \grad_{\vx} f_{\tau_k(j)}(\vx^k_0, \vy^k_0)| \\
    &\leq le^{1/5} \sum_{i=1}^{n-1}i |\grad_{\vx}F(\vx^k_0, \vy^k_0) - \grad_{\vx}F(\vx^k_0, \vy^k_0) +  1/i\sum_{j=1}^{i} \grad_{\vx} f_{\tau_k(j)}(\vx^k_0, \vy^k_0)| \\
    &\leq le^{1/5} \left[ n^2 \norm{\grad_{\vx}F(\vx^k_0, \vy^k_0)}/2 + \sum_{i=1}^{n-1} i |\grad_{\vx}F(\vx^k_0, \vy^k_0) -  1/i\sum_{j=1}^{i} \grad_{\vx} f_{\tau_k(j)}(\vx^k_0, \vy^k_0)|  \right].
\end{align*}
By defining $d_{k,i} = |\grad_{\vx}F(\vx^k_0, \vy^k_0) -  1/i\sum_{j=1}^{i} \grad_{\vx} f_{\tau_k(j)}(\vx^k_0, \vy^k_0)|$ and substituting into the above inequality, we obtain,
\begin{equation}
\label{app-eqn:agda-rr-rk-mid}
\norm{\vr_k} \leq le^{1/5} \left[ n^2 \norm{\grad_{\vx}F(\vx^k_0, \vy^k_0)}/2 + \sum_{i=1}^{n-1} i d_{k,i}  \right].
\end{equation}
By applying Young's inequality to \eqref{app-eqn:agda-rr-rk-mid}, we conclude:
\begin{equation}
\label{app-eqn:agda-rr-rk-sq-mid}
\norm{\vr_k}^2 \leq 2l^2e^{2/5} \left[ n^4 \norm{\grad_{\vx}F(\vx^k_0, \vy^k_0)}^2/4 + (n-1)\sum_{i=1}^{n-1} i^2 d^2_{k,i}  \right].
\end{equation}
By Lemma \ref{app-lem:sample-mean-wor} and the bounded gradient variance assumption, we bound $\bE[d^2_{k,i}]$ as follows:
\begin{equation}
\label{app-eqn:agda-rr-dki-sq}
    \bE[d^2_{k,i}] = \frac{n-i}{i(n-1)}\left[1/n\sum_{i=1}^{n}|\grad_{\vx}f_i(\vx^k_0, \vy^k_0) - \grad_{\vx}F(\vx^k_0, \vy^k_0)|^2\right] \leq \frac{n-i}{i(n-1)} \sigma^2,
\end{equation}
where the expectation is taken over the uniformly sampled random permutation $\tau_k$. Furthermore, applying Jensen's inequality to \eqref{app-eqn:agda-rr-dki-sq} gives us the following bound:
\begin{align}
\label{app-eqn:agda-rr-dki}
    \bE[d_{k,i}] \leq \sqrt{\bE[d^2_{k,i}]} \leq \sigma \sqrt{\frac{n-i}{i(n-1)}}.
\end{align}
From \eqref{app-eqn:agda-rr-dki}, it follows that,
\begin{align*}
    \sum_{i=1}^{n-1} i \bE[d_{k,i}] &\leq \sigma \sum_{i=1}^{n-1} \sqrt{i} \sqrt{\frac{n-i}{n-1}} \leq \sigma \sum_{i=1}^{n-1} \sqrt{i} \leq \frac{2\sigma n^{3/2}}{3}.
\end{align*}
By taking expectations of \eqref{app-eqn:agda-rr-rk-mid} and substituting the above, we get,
\begin{equation}
\label{app-eqn:agda-rr-rk-bound}
    \bE[\norm{\vr_k}] \leq \frac{2n^2l}{3}\norm{\grad_{\vx} F(\vx^k_0, \vy^k_0)} +  \sigma ln^{3/2}.
\end{equation}
Similarly, from \eqref{app-eqn:agda-rr-dki-sq}, it follows that,
\begin{align*}
    (n-1)\sum_{i=1}^{n-1} i^2 \bE[d_{k, i}^2] \leq \sigma^2 \sum_{i=1}^{n-1} i(n-i) \leq \frac{\sigma^2 n^3}{6}.
\end{align*}
By taking expectations of \eqref{app-eqn:agda-rr-rk-sq-mid} and substituting the above, we obtain,
\begin{equation}
\label{app-eqn:agda-rr-rk-sq-bound}
\bE[\norm{\vr_k}^2] \leq \frac{3n^4 l^2}{4}\norm{\grad_{\vx} F(\vx^k_0, \vy^k_0)}^2 + \frac{\sigma^2 l^2 n^3}{2}.
\end{equation}
We bound $\bE[\norm{\vs_k}]$ and $\bE[\norm{\vs_k}^2]$ by following a similar strategy and obtain:
\begin{equation}
\label{app-eqn:agda-rr-sk-bound}
    \bE[\norm{\vs_k}] \leq \frac{2n^2l}{3}\norm{\grad_{\vy} F(\vx^{k+1}_0, \vy^k_0)} +  \sigma ln^{3/2},
\end{equation}
\begin{equation}
\label{app-eqn:agda-rr-sk-sq-bound}
\bE[\norm{\vs_k}^2] \leq \frac{3n^4 l^2}{4}\norm{\grad_{\vy} F(\vx^{k+1}_0, \vy^k_0)}^2 + \frac{\sigma^2 l^2 n^3}{2}.
\end{equation}

\textbf{Step 3: Obtaining a convergence guarantee} We begin by defining $a_k$ and $b_k$ as:
\begin{align*}
    a_k &= \Phi(\vx^k_0) - \Phi^*, \\
    b_k &= \Phi(\vx^k_0) - F(\vx^k_0, \vy^k_0).
\end{align*}
By definition of $V_{\lambda}$, we conclude that $V_{\lambda}(\vz^k_0) = a_k + \lambda b_k$. Furthermore, we note that by Lemma \ref{app-lem:agda-phi-prop}, $\Phi$ is $L$-smooth where $L = l + l^2/\mu_2$. Since $\alpha \leq 1/5\eta n l$, it follows that $n \alpha L \leq 1$. Using these properties, we obtain the following:
\begin{align*}
    \Phi(\vx^{k+1}_0) - \Phi(\vx^{k}_0) &\leq \dotprod{\grad \Phi(\vx^{k}_0)}{\vx^{k+1}_0 - \vx^{k}_0} + \frac{L}{2}\norm{\vx^{k+1}_0 - \vx^{k}_0}^{2} \\ 
    &\leq \dotprod{\nabla \Phi(\vx^{k}_0)}{-n \alpha \grad_\vx F(\vx^{k}_0, \vy^{k}_0) + \alpha^2 \vr_k} + \frac{L}{2}\norm{- n \alpha \grad_\vx F(\vx^{k}_0, \vy^{k}_0) + \alpha^2 \vr_k}^2 \\ 
    &\leq -n \alpha \dotprod{\nabla \Phi(\vx^{k}_0)}{\grad_\vx F(\vx^{k}_0, \vy^{k}_0) - \frac{\alpha}{n} \vr_k} + \frac{L}{2}n^2 \alpha^2 \norm{\grad_\vx F(\vx^{k}_0, \vy^{k}_0) - \frac{\alpha}{n} \vr_k}^2 \\
    &\leq \frac{n \alpha}{2} \norm{\grad_\vx F(\vx^{k}_0, \vy^{k}_0) - \nabla \Phi(\vx^{k}_0) - \frac{\alpha}{n} \vr_k}^2 - \frac{n \alpha}{2} \norm{\nabla \Phi(\vx^{k}_0)}^2 \\
    &\leq n \alpha \norm{\grad_\vx F(\vx^{k}_0, \vy^{k}_0) - \nabla \Phi(\vx^{k}_0)}^2 - \frac{n \alpha}{2} \norm{\nabla \Phi(\vx^{k}_0)}^2  + \frac{\alpha^3}{n} \norm{\vr_k}^2.
\end{align*}
By taking expectations conditioned on the epoch iterates $(\vx^j_0, \vy^j_0)_{j=1}^{k}$ and substituting the bound \eqref{app-eqn:agda-rr-rk-sq-bound}, we obtain the following:
\begin{align}
\label{app-eqn:agda-rr-phi-update}
\bE_k[\Phi(\vx^{k+1}_0)] - \Phi(\vx^{k}_0) &\leq n \alpha \norm{\grad_\vx F(\vx^{k}_0, \vy^{k}_0) - \nabla \Phi(\vx^{k}_0)}^2 - \frac{n \alpha}{2} \norm{\nabla \Phi(\vx^{k}_0)}^2   \nonumber\\
    &+ \frac{3 n^3 \alpha^3 l^2}{4} \norm{\grad_\vx F(\vx_{0}^{k}, \vy_{0}^{k})}^2 + \frac{n^2 \alpha^3 l^2 \sigma^2}{2}.
\end{align}
We proceed by applying the $l$-smoothness of $-F(\vx,\vy)$ in $\vx$, to obtain the following:
\begin{align}
    F(\vx^{k}_0, \vy^{k}_0) - F(\vx^{k+1}_0, \vy^{k}_0) &\leq -\dotprod{\grad_\vx  F(\vx^{k}_0, \vy^{k}_0)}{\vx^{k+1}_0 - \vx^{k}_0} + \frac{l}{2}\norm{\vx^{k+1}_0 - \vx^{k}_0}^2 \nonumber \\
    &\leq \dotprod{\grad_\vx F(\vx^{k}_0, \vy^{k}_0)}{n \alpha \grad_\vx F(\vx^{k}_0, \vy^{k}_0) - \alpha^2 \vr_k} \nonumber\\ 
    &+  \frac{l}{2} \norm{n \alpha \grad_\vx F(\vx^{k}_0, \vy^{k}_0) - \alpha^2 \vr_k}^2 \nonumber \\
    &\leq (n \alpha + \frac{l}{2}n^2 \alpha^2)\norm{\grad_\vx F(\vx^{k}_0, \vy^{k}_0)}^2 + \frac{l}{2}\norm{\alpha^2 \vr_k}^2 \nonumber \\
    &-(\alpha^2 + ln\alpha^3)\dotprod{ \grad_\vx F(\vx^{k}_0, \vy^{k}_0)}{ \vr_k} \nonumber \\
    &\leq  (n \alpha + \frac{l}{2}n^2 \alpha^2)\norm{\grad_\vx F(\vx^{k}_0, \vy^{k}_0)}^2 + \frac{l \alpha^4}{2} \norm{\vr_k}^2 \nonumber \\
    &+ \frac{6 \alpha^2}{5}\norm{\grad_\vx F(\vx^{k}_0, \vy^{k}_0)} \norm{\vr_k}. 
\label{app-eqn:agda-rr-fx-mid}
\end{align}
From \eqref{app-eqn:agda-rr-rk-bound} and \eqref{app-eqn:agda-rr-rk-sq-bound}, we infer that:
\begin{align*}
    \frac{l \alpha^4}{2} \bE[\norm{\vr_k}^2] &\leq \frac{3n^4 \alpha^4 l^3}{8}\norm{\grad_{\vx} F(\vx^k_0, \vy^k_0)}^2 + \frac{\sigma^2 l^3 \alpha^4 n^3}{4} \\
    &\leq \frac{3n^4 \alpha^4 l^3}{8}\norm{\grad_{\vx} F(\vx^k_0, \vy^k_0)}^2 + \frac{\sigma^2 l^2 \alpha^3 n^2}{20}, \\
     \frac{6\alpha^2}{5}\norm{\grad_{\vx}F(\vx_{0}^{k}, \vy_{0}^{k})}\bE[\norm{\vr_k}] &\leq \frac{6\alpha^2n^{\frac{3}{2}}\sigma l}{5} \norm{\grad_\vx F(\vx_{0}^{k}, \vy_{0}^{k})} + \frac{4 n^2 \alpha^2 l}{5}\norm{\grad_{\vx}F(\vx_{0}^{k}, \vy_{0}^{k})}^2 \\
    &= (\frac{6\alpha^{\frac{3}{2}}nl \sigma}{5})( \alpha^{\frac{1}{2}}n^{\frac{1}{2}}\norm{\grad_{\vx}F(x_{0}^{k}, y_{0}^{k})}) + \frac{4 n^2 \alpha^2 l}{5}\norm{\grad_{\vx}F(\vx_{0}^{k}, \vy_{0}^{k})}^2 \\
    &\leq (\frac{n \alpha}{2} + \frac{4 n^2 \alpha^2 l}{5})\norm{\grad_{\vx}F(\vx_{0}^{k}, \vy_{0}^{k})}^2 + \frac{18\alpha^3 n^2 l^2 \sigma^2}{25}.
\end{align*}
Taking expectations on both sides of \eqref{app-eqn:agda-rr-fx-mid} by conditioning on the epoch iterates $(\vx^j_0, \vy^j_0)_{j=1}^{k}$ and substituting the above inequalities, we obtain:
\begin{equation}
\label{app-eqn:agda-rr-fx-bound}
    F(\vx^{k}_0, \vy^{k}_0) - \bE_k[F(\vx^{k+1}_0, \vy^{k}_0)] \leq (\frac{3 n \alpha}{2} + \frac{13}{10}l n^2 \alpha^2 + \frac{3 n^4 \alpha^4 l^3}{8})\norm{\grad_\vx F(\vx_{0}^{k}, \vy_{0}^{k})}^2 + \frac{4\alpha^3 n^2 l^2 \sigma^2}{5}.
\end{equation}
We now apply the $l$-smoothness of $-F(\vx,\vy)$ in $\vy$ to obtain the following:
\begin{align}
    F(\vx^{k+1}_0, \vy^{k}_0) -F(\vx^{k+1}_0, \vy^{k+1}_0) &\leq  - \dotprod{\grad_{\vy} F(\vx^{k+1}_0, \vy^{k}_0)}{ \vy^{k+1}_0 - \vy^{k}_0} + \frac{l}{2}\norm{\vy^{k+1}_0 - \vy^{k}_0}^{2} \nonumber \\
    &\leq - \dotprod{\grad_{\vy} F(\vx^{k+1}_0, \vy^{k}_0)}{n \beta \grad_{\vy} F(\vx^{k+1}_0, \vy^{k}_0) + \beta^2 \vs_k} \nonumber \\
    &+ \frac{l}{2}\norm{n \beta \grad_{\vy} F(\vx^{k+1}_0, \vy^{k}_0) + \beta^2 \vs_k}^2 \nonumber \\
    &\leq (-n\beta + l n^2 \beta^2)\norm{\grad_{\vy} F(\vx^{k+1}_0, \vy^{k}_0)}^2 \nonumber \\
    &+ \beta^2 \norm{\grad_{\vy} F(\vx^{k+1}_0, \vy^{k}_0)} \norm{\vs_k} + l \beta^4\norm{\vs_k}^2.
\label{app-eqn:agda-rr-fy-mid}
\end{align}
From \eqref{app-eqn:agda-rr-sk-bound} and \eqref{app-eqn:agda-rr-sk-sq-bound}, we conclude,
\begin{align*}
    l \beta^4 \bE[\norm{\vs_k}^2] &\leq \frac{3n^4 l^3 \beta^4}{4}\norm{\grad_{\vy} F(\vx^{k+1}_0, \vy^k_0)}^2 + \frac{\sigma^2 l^3 n^3 \beta^4}{2} \\
    &\leq \frac{3n^4 l^3 \beta^4}{4}\norm{\grad_{\vy} F(\vx^{k+1}_0, \vy^k_0)}^2 + \frac{\sigma^2 l^2 n^2 \beta^3}{10}, \\
    \beta^2 \norm{\grad_{\vy} F(\vx^{k+1}_0, \vy^{k}_0)} \bE[\norm{\vs_k}] &\leq \sigma l n^{3/2} \beta^2 \norm{\grad_{\vy} F(\vx^{k+1}_0, \vy^{k}_0)} + \frac{2 n^2 \beta^2 l}{3}\norm{\grad_{\vy} F(\vx^{k+1}_0, \vy^{k}_0)}^2 \\
    &\leq (\frac{n^{\frac{1}{2}}\beta^{\frac{1}{2}}}{2}\norm{\grad_{\vy} F(\vx^{k+1}_0, \vy^{k}_0)})(2n\beta^{\frac{3}{2}}l\sigma) +  \frac{2 n^2 \beta^2 l}{3}\norm{\grad_{\vy} F(\vx^{k+1}_0, \vy^{k}_0)}^2 \\
    &\leq (\frac{n \beta}{8} + \frac{2 n^2 \beta^2 l}{3})\norm{\grad_{\vy} F(\vx^{k+1}_0, \vy^{k}_0)}^2 + 2n^2 \beta^3 l^2 \sigma^2.
\end{align*}
We proceed by taking expectations on both sides of \eqref{app-eqn:agda-rr-fy-mid} by conditioning on the epoch iterates $(\vx^j_0, \vy^j_0)_{j=1}^{k}$  and $\vx^{k+1}_0$ and substituting the above bounds. Consequently, we obtain the following inequality:
\begin{align*}
    F(\vx^{k+1}_0, \vy^{k}_0) - \bE_k[F(\vx^{k+1}_0, \vy^{k+1}_0) | \vx^{k+1}_0] &\leq (-\frac{7n\beta}{8} + \frac{5 n^2 \beta^2 l}{3} + \frac{3n^4 \beta^4 l^3}{4})\norm{\grad_{\vy} F(\vx^{k+1}_0, \vy^{k}_0)}^2 \\
    &+ \frac{21n^2 \beta^3 l^2 \sigma^2}{10} \\
    &\leq -\frac{n \beta}{2}\norm{\grad_{\vy} F(\vx^{k+1}_0, \vy^{k}_0)}^2 + \frac{21n^2 \beta^3 l^2 \sigma^2}{10}.
\end{align*}
The two-sided P\L{} condition on $F$ implies that,
\begin{align*}
    -\frac{n \beta}{2} \norm{\grad_{\vy} F(\vx^{k+1}_0, \vy^{k}_0)}^2 &\leq -n \beta \mu_2 [\Phi(\vx^{k+1}_0) - F(\vx^{k+1}_0, \vy^{k}_0)],
\end{align*}
which gives us the following:
\begin{equation*}
\bE_k[\Phi(\vx^{k+1}_0) -F(\vx^{k+1}_0, \vy^{k+1}_0) | \vx^{k+1}_0] \leq (1 - n \beta \mu_2) [\Phi(\vx^{k+1}_0) - F(\vx^{k+1}_0, \vy^{k}_0)] + \frac{21n^2 \beta^3 l^2 \sigma^2}{10}.
\end{equation*}
Taking expectations of the above with respect to $\vx^{k+1}_0$ and applying the law of iterated expectation, we obtain:
\begin{equation}
\label{app-eqn:agda-rr-fy-bound}
\bE_k[\Phi(\vx^{k+1}_0) -F(\vx^{k+1}_0, \vy^{k+1}_0)] \leq (1 - n \beta \mu_2) \bE_k[\Phi(\vx^{k+1}_0) - F(\vx^{k+1}_0, \vy^{k}_0)] + \frac{21n^2 \beta^3 l^2 \sigma^2}{10}.
\end{equation}
From \eqref{app-eqn:agda-rr-phi-update}, we obtain the following:
\begin{multline}
\bE_k[a_{k+1}] \leq a_k + n \alpha \norm{\grad_\vx F(\vx^{k}_0, \vy^{k}_0) - \nabla \Phi(\vx^{k}_0)}^2 - \frac{n \alpha}{2} \norm{\nabla \Phi(\vx^{k}_0)}^2 \\ + \frac{3 n \alpha}{100} \norm{\grad_\vx F(\vx_{0}^{k}, \vy_{0}^{k})}^2 + \frac{n^2 \alpha^3 l^2 \sigma^2}{2}.
\label{app-eqn:agda-rr-a-update}
\end{multline}
Furthermore, we note that,
\begin{align*}
    \Phi(\vx^{k+1}_0) - F(\vx^{k+1}_0, \vy^{k}_0) = b_k + F(\vx^{k}_0, \vy^{k}_0) - F(\vx^{k+1}_0, \vy^{k}_0) + \Phi(\vx^{k+1}_0) - \Phi(\vx^{k}_0).
\end{align*}
By substituting the above into \eqref{app-eqn:agda-rr-fy-bound}, we obtain the following:
\begin{multline*}
    \bE_k[b_{k+1}] \leq (1 - n\beta \mu_2)b_k + (1 - n \beta \mu_2)\bE_k[F(\vx^k_0, \vy^k_0) - F(\vx^{k+1}_0, \vy^{k}_0)] \\
    + (1 - n\beta \mu_2)\bE_k[\Phi(\vx^{k+1}_0) - \Phi(\vx^k_0)] + \frac{21 n^2 \beta^3 l^2 \sigma^2}{10}.
\end{multline*}
By substituting \eqref{app-eqn:agda-rr-fx-bound} and \eqref{app-eqn:agda-rr-phi-update} into the above, we obtain:
\begin{align}
    \bE_k[b_{k+1}] &\leq (1 - n \beta \mu_2)b_k + (1 - n \beta \mu_2) n \alpha \norm{\grad_\vx F(\vx^{k}_0, \vy^{k}_0) - \grad \Phi(\vx^{k}_0)}^2 \nonumber \\ 
    &- (1 - n \beta \mu_2)\frac{n \alpha}{2} \norm{\grad \Phi(\vx^{k}_0)}^2  + (1 - n \beta \mu_2)\frac{9 n \alpha}{5}\norm{\grad_{\vx}F(\vx^{k}_0, \vy^{k}_0)}^2 \nonumber \\
    &+ (1 - n \beta \mu_2)\frac{13n^2 \alpha^3 l^2 \sigma^2}{10}  + \frac{21n^2 \beta^3 l^2 \sigma^2}{10}.
\label{app-eqn:agda-rr-b-update}
\end{align}
From \eqref{app-eqn:agda-rr-a-update} and \eqref{app-eqn:agda-rr-b-update}, we get the following by substituting $V_{\lambda}(\vz^{k+1}_0) = a_{k+1} + \lambda b_{k+1}$,
\begin{multline}
    \bE_k[V_{\lambda}(\vz^{k+1}_0)] \leq a_k + (1 - n \beta \mu_2) \lambda b_k + [n\alpha + \lambda n \alpha(1 - n \beta \mu_2)] \norm{\grad_\vx F(\vx^{k}_0, \vy^{k}_0) - \nabla \Phi(\vx^{k}_{0})}^2 + C_3  \\
    - [\frac{n \alpha}{2} + \frac{\lambda n \alpha}{2}(1 - n \beta \mu_2)]\norm{\nabla \Phi(\vx^{k}_{0})}^2 + [\frac{3n\alpha}{100}+\lambda(1 - n\beta \mu_2)\frac{9 n \alpha}{5}] \norm{\grad_{\vx}F(\vx^{k}_{0}, \vy^{k}_{0})}^2,
\label{app-eqn:agda-rr-v-partial-update}
\end{multline}
where $C_3$ is defined as follows:
\begin{align*}
    C_3 =  [\frac{1}{2} + \frac{13\lambda}{10}(1 - n \beta \mu_2)]n^2 \alpha^3 l^2 \sigma^2  + \frac{21\lambda n^2 \beta^3 l^2 \sigma^2}{10}.
\end{align*}
By Young's inequality, the following holds for any $\varepsilon \in (0,1]$:
\begin{align*}
    \norm{\grad_{\vx}F(\vx^{k}_0, \vy^{k}_0)}^2 &\leq (1 + \varepsilon)\norm{\nabla \Phi(\vx^{k}_0)}^2 + (1 + 1/\varepsilon)\norm{\nabla_\vx F(\vx^{k}_0, \vy^{k}_0) - \nabla \Phi(\vx^{k}_0)}^2. 
\end{align*}
Substituting this into \eqref{app-eqn:agda-rr-v-partial-update}, we obtain the following:
\begin{multline}
    \bE_k[V_{\lambda}(\vz^{k+1}_0)] \leq a_k + (1 - n \beta \mu_2) \lambda b_k + C_3  \\
    - [ \frac{n \alpha}{2} - (1 + \varepsilon)\frac{3n \alpha}{100} - \lambda(1 - n \beta \mu_2)[(1 + \varepsilon)\frac{9 n \alpha}{5} - \frac{n \alpha}{2}] ]\norm{\grad \Phi(\vx^{k}_0)}^2  \\
    + [n \alpha + (1 + 1/\varepsilon)\frac{3n \alpha}{100} + \lambda(1 - n \beta \mu_2)[n \alpha + (1 + 1/\varepsilon)\frac{9n \alpha}{5}]]\norm{\grad_{\vx} F(\vx^{k}_0, \vy^{k}_0) - \nabla \Phi(\vx^{k}_{0})}^2.
\label{app-eqn:agda-rr-v-quarter-update}
\end{multline}
From Lemma \ref{app-lem:agda-phi-prop}, we know that $\Phi(\vx)$ satisfies the P\L{} inequality with constant $\mu_1$. Thus, we conclude that
\begin{equation}
\label{app-eqn:phi-pl-ak}
    \norm{\nabla \Phi(\vx^k_0)}^2 \geq 2\mu_1(\Phi(\vx^k_0) - \Phi^*) = 2\mu_1 a_k.
\end{equation}
Let $\vy^{*}(\vx^k_0)$ denote the projection of $\vy^k_0$ on the set $\arg \max_{\vy} F(\vx^k_0, \vy)$. From Lemma \ref{app-lem:agda-phi-prop}, we conclude that $\grad \Phi(\vx^k_0) = \grad_{\vx} F(\vx^k_0, \vy^{*}(\vx^k_0))$. Thus,   
\begin{align}
    \norm{\grad_{\vx} F(\vx^{k}_0, \vy^{k}_0) - \nabla \Phi(\vx^{k}_{0})}^2 &\leq \norm{\grad_{\vx} F(\vx^{k}_0, \vy^{k}_0) - \grad_{\vx} F(\vx^{k}_0, \vy^{*}(\vx^{k}_0))}^2 \nonumber \\
    &\leq l^2 \norm{\vy^k_0 - \vy^*(\vx^k_0)}^2
\label{app-eqn:agda-rr-pl-grad-update}
\end{align}
Furthermore, we note that Assumption \ref{as:2pl} implies that $-F(\vx^k_0, \vy)$ satisfies the P\L{} inequality with constant $\mu_2$. Thus, by Lemma \ref{as:pl-quad-growth}, we conclude that,
\begin{align*}
    b_k = \Phi(\vx^k_0) - F(\vx^k_0, \vy^k_0) \geq \frac{\mu_2}{2}\norm{\vy^k_0 - \vy^*(\vx^k_0)}^2.
\end{align*}
By substituting the above inequality in \eqref{app-eqn:agda-rr-pl-grad-update}, we conclude that,
\begin{align*}
    \norm{\grad_{\vx} F(\vx^{k}_0, \vy^{k}_0) - \nabla \Phi(\vx^{k}_{0})}^2 \leq \frac{2 l^2 b_k}{\mu_2}.
\end{align*}
We proceed by substituting the above bound in \eqref{app-eqn:agda-rr-v-quarter-update} to obtain the following upper bound on $\bE_k[V_{\lambda}(\vz^{k+1}_0)]$:
\begin{multline*}
   \bE_k[V_{\lambda}(\vz^{k+1}_0)] \leq \{ 1 - [\mu_1 n \alpha - (1 + \varepsilon)\frac{3 \mu_1 n \alpha}{50} - \lambda(1 - n \beta \mu_2)[(1 + \varepsilon)\frac{18 \mu_1 n\alpha}{5} - \mu_1 n \alpha]] \}a_k + C_3 \\
    + [1 - n\beta \mu_2 + \frac{2l^2n\alpha}{\mu_2\lambda} + (1+1/\varepsilon)\frac{3l^2n\alpha}{50\mu_2\lambda} + (1 - n\beta \mu_2)[\frac{2l^2n\alpha}{\mu_2} + (1 + 1/\varepsilon)\frac{18l^2 n\alpha}{5\mu_2}]] \lambda b_k.
\end{multline*}
Hence, we obtain the following update equation for $V_{\lambda}(\vz^{k+1}_0)$:
\begin{equation}
\label{app-eqn:agda-rr-v-mid-update}
    \bE_k[V_{\lambda}(\vz^{k+1}_0)] \leq C_1 a_k + C_2 \lambda b_k + C_3,
\end{equation}
where $C_1$, $C_2$ and $C_3$ are given by,
\begin{align*}
    C_1 &=  1 - [\mu_1 n \alpha - (1 + \varepsilon)\frac{3 \mu_1 n \alpha}{50} - \lambda(1 - n \beta \mu_2)[(1 + \varepsilon)\frac{18 \mu_1 n\alpha}{5} - \mu_1 n \alpha]] \\
    &= 1 - \mu_1 n \alpha [1 - \frac{3(1+\varepsilon)}{50} - \lambda(1 - n\beta \mu_2)[\frac{18(1 + \varepsilon)}{5} - 1]], \\
    C_2 &= 1 - n\beta \mu_2 + \frac{2l^2n\alpha}{\mu_2\lambda} + (1+1/\varepsilon)\frac{3l^2n\alpha}{50\mu_2\lambda} + (1 - n\beta \mu_2)[\frac{2l^2n\alpha}{\mu_2} + (1 + 1/\varepsilon)\frac{18l^2 n\alpha}{5\mu_2}] \\
    &= 1 - \frac{2l^2n\alpha}{\mu_2}[\frac{\beta \mu_2^2}{2 \alpha l^2} - \frac{1}{\lambda} - (1 + 1/\varepsilon)\frac{3}{100\lambda} - (1 - n\beta\mu_2)[1 + (1+1/\varepsilon)\frac{9}{5}]], \\
    C_3 &=  [\frac{1}{2} + \frac{13\lambda}{10}(1 - n \beta \mu_2)]n^2 \alpha^3 l^2 \sigma^2  + \frac{21\lambda n^2 \beta^3 l^2 \sigma^2}{10}.
\end{align*}
By setting $\lambda = \frac{1}{10}, \varepsilon = \frac{3}{7}$ and $\eta = \frac{73 l^2}{2 \mu_2^2}$ and using $l^2 \geq \mu_1 \mu_2$, we conclude:
\begin{align*}
    C_1 &\leq 1 - \mu_1 n \alpha / 2, \\
    C_2 &\leq 1- \frac{l^2 n \alpha}{2\mu_2} \leq 1 - \mu_1 n\alpha / 2, \\
    C_3 &\leq 0.21 n^2 \alpha^3 l^2\sigma^2 (3 + \eta^3)
\end{align*}
From the above bounds, it follows that:
\begin{align*}
    \bE_k[V_{\lambda}(\vz^{k+1}_0)] \leq (1 - \mu_1 n \alpha/2) V_{\lambda}(\vz^{k}_0) +  0.21 n^2 \alpha^3 l^2\sigma^2(3 + \eta^3).
\end{align*}
By taking expectations with respect to the epoch iterates $(\vx^j_0, \vy^j_0)_{j=1}^{k}$ and applying the law of iterated expectation, we obtain the following:
\begin{align*}
    \bE[V_{\lambda}(\vz^{k+1}_0)] \leq (1 - \mu_1 n \alpha/2) \bE[V_{\lambda}(\vz^{k}_0)] +  0.21 n^2 \alpha^3 l^2\sigma^2(3 + \eta^3).
\end{align*}
By unrolling the above inequality for $K$ steps, we obtain the following:
\begin{equation}
\label{app-eqn:agda-rr-v-unroll}
\bE[V_{\lambda}(\vz^{K+1}_0)] \leq (1 - \mu_1 n \alpha/2)^K V_{\lambda}(\vz_0) + 0.42 n \alpha^2 l^2 \sigma^2(3 + \eta^3)/\mu_1.
\end{equation}
By setting $\alpha = \min \{ 1/5\eta n l, 4 \log(V_{\lambda}(\vz_0)n^{1/2}K)/\mu_1 nK \}$ and using $\eta \leq 37 \kappa^2$, we observe that:
\begin{align*}
 n \alpha^2 l^2 \sigma^2(3 + \eta^3)/\mu_1 &\leq  \frac{16l^2 \sigma^2(3 + \eta^3)}{\mu_1^3} \frac{\log^2(V_{\lambda}(\vz_0)n^{1/2}K)}{nK^2} \\
 &\leq \frac{c \kappa^8 \sigma^2 \log^2(V_{\lambda}(\vz_0)n^{1/2}K)}{\mu_1 n K^2},
\end{align*}
where $c > 0$ is a constant that is independent of $\kappa, \mu$ and $\sigma^2$. By substituting the above into \eqref{app-eqn:agda-rr-v-unroll}, we get:
\begin{equation}
\label{app-eqn:agda-rr-v-unroll-mid}
    \bE[V_{\lambda}(\vz^{K+1}_0)] \leq (1 - \mu_1 n \alpha/2)^K V_{\lambda}(\vz_0) + \frac{c \kappa^8 \sigma^2 \log^2(V_{\lambda}(\vz_0)n^{1/2}K)}{\mu_1 n K^2}.
\end{equation}
To complete the proof, we consider the following cases:

\textbf{Case 1: $1/5\eta n l \leq 4 \log(V_{\lambda}(\vz_0)n^{1/2}K)/\mu_1 nK $}

In this case, $\alpha = 1/5 \eta n l$ and hence, we get:
\begin{align*}
    (1 - n\alpha \mu_1 /2)^K V_{\lambda}(\vz_0) \leq e^{-n\alpha \mu_1 K/2} V_{\lambda}(\vz_0) \leq e^{-\frac{K}{365\kappa^3}} V_{\lambda}(\vz_0).
\end{align*}
Substituting into \eqref{app-eqn:agda-rr-v-unroll-mid}, we get:
\begin{equation}
\label{app-eqn:agda-rr-case1-proof}
\bE[V_{\lambda}(\vz^{K+1}_0)] \leq e^{-\frac{K}{365\kappa^3}} V_{\lambda}(\vz_0) + \frac{c \kappa^8 \sigma^2 \log^2(V_{\lambda}(\vz_0)n^{1/2}K)}{\mu_1 n K^2}.
\end{equation}
\textbf{Case 2: $1/5\eta n l \leq 4\log(V_{\lambda}(\vz_0)n^{1/2}K)/(\mu_1 nK)$} 

In this case, $\alpha = 4 \log(V_{\lambda}(\vz_0)n^{1/2}K)/\mu_1 nK$ and we conclude that,
\begin{align*}
    (1 - n\alpha \mu_1 / 2)^KV_{\lambda}(\vz_0) &\leq e^{-n\alpha \mu_1 K/2}V_{\lambda}(\vz_0) \leq e^{-2\log(V_{\lambda}(\vz_0)^{1/2}n^{1/2}K)}V_{\lambda}(\vz_0) \leq \frac{1}{nK^2}.
\end{align*}
By substituting into \eqref{app-eqn:agda-rr-v-unroll-mid}, we get:
\begin{equation}
\label{app-eqn:agda-rr-case2-proof}
\bE[V_{\lambda}(\vz^{K+1}_0)] \leq \frac{\mu_1 + c \kappa^8 \sigma^2 \log^2(V_{\lambda}(\vz_0)n^{1/2}K)}{\mu_1 n K^2}.
\end{equation}
By taking the maximum of the right hand side of \eqref{app-eqn:agda-rr-case1-proof} and \eqref{app-eqn:agda-rr-case2-proof}, we obtain the following rate, which holds for any $K \geq 1$:
\begin{align*}
\bE[V_{\lambda}(\vz^{K+1}_0)] \leq  e^{-\frac{K}{365\kappa^3}} V_{\lambda}(\vz_0) + \frac{\mu_1 + c \kappa^8 \sigma^2 \log^2(V_{\lambda}(\vz_0)n^{1/2}K)}{\mu_1 n K^2}.
\end{align*}
Suppressing constant terms and logarithmic factors, we get,
\begin{align*}
    \bE[V_{\lambda}(\vz^{K+1}_0)] = \Tilde{O}(e^{-\frac{K}{365\kappa^3}} + \nicefrac{1}{nK^2}).
\end{align*}
\end{proof}
\paragraph{Bounded iterate assumption} We now demonstrate that our analysis of AGDA-RR easily adapts to the bounded iterate setting. To this end, we assume that all the iterates of AGDA-RR, i.e., the iterates $(\vx^k_i, \vy^k_0)$ and $(\vx^k_n, \vy^k_i)$ where $i \in [n]$ and $k \in [K]$ are contained within a compact set $\mathcal{W}$. Furthermore, we note that the $l$-smoothness of the $f_i$'s guarantees that the function $G_{\sigma}(\vx, \vy) = 1/n \sum_{i=1}^{n} \norm{\omega_i(\vx, \vy) - \nu(\vx, \vy)}^2$ is non-negative and continuous. Hence, $\sigma^2 = \max_{(\vx, \vy) \in \mathcal{W}} G_{\sigma}(\vx, \vy)$ is non-negative and finite. Since $(\vx^k_0, \vy^k_0) \in \mathcal{W}$ and $(\vx^{k+1}_0, \vy^k_0) \in \mathcal{W}$, we conclude that 
\begin{align*}
1/n \sum_{i=1}^{n} \norm{\omega_i(\vx^k_0, \vy^k_0) - \nu(\vx^k_0, \vy^k_0)}^2 \leq \sigma^2, \\
1/n \sum_{i=1}^{n} \norm{\omega_i(\vx^{k+1}_0, \vy^k_0) - \nu(\vx^{k+1}_0, \vy^k_0)}^2 \leq \sigma^2.
\end{align*}
Substituting the above inequalities into Step 2 of Theorem \ref{app-thm:agda-rr-convergence} shows that our convergence analysis of AGDA-RR straightforwardly adapts to the bounded iterate setting. The same argument also applies to our analysis of AGDA-AS.

\paragraph{Convergence in terms of $\textrm{dist}(\vz, \sZ^*)^2$} Let $\sZ^{*}$ denote the set of all saddle points of the objective $F$ and let $\textrm{dist}(\vz, \sZ^*)$ denote the distance of a point $\vz \in \bR^d$ from $\sZ^*$. As we demonstrate in Appendix \ref{app-sec:agda-conv-guarantee}, the Lyapunov function $V_{\lambda}$ used in our analysis satisfies the following:

 $$\textrm{dist}(\vz, \sZ^*)^2 \leq \max \left \{ \frac{2}{\mu_1}(\frac{l^2}{2\mu_2}+1) , \frac{4}{\lambda \mu_2} \right \} V_{\lambda}(\vz).$$

Using the above relation, the convergence guarantee of Theorem \ref{app-thm:agda-rr-convergence} can be translated to a guarantee of the form $\bE[\textrm{dist}(\vz^{K+1}_0, \sZ^*)^2] = \Tilde{O}(e^{-\nicefrac{K}{365\kappa^3}} + \nicefrac{1}{nK^2})$. Thus, the epoch iterates $\vz^{K+1}_0$ of AGDA-RR converge (in expectation) to some saddle point of $F$ at a rate of $\Tilde{O}(e^{-\nicefrac{K}{365\kappa^3}} + \nicefrac{1}{nK^2})$. A similar argument also applies to AGDA-AS.
\subsection{Analysis of AGDA-AS}
\label{app-sec:agda-as}
\begin{theorem}[Convergence of AGDA-AS]
\label{app-thm:agda-as-convergence}
Let Assumptions \ref{as:comp-smooth}, \ref{as:2pl}, and \ref{as:bgv} be satisfied and let $\eta = \nicefrac{73l^2}{2\mu^2_2}$. Then, for any $\alpha \leq \nicefrac{1}{5\eta n l}$,  $\beta = \eta \alpha$, and $K \geq 1$, the iterates of AGDA-AS satisfy the following:
$$\max_{\tau_1, \pi_1, \ldots, \tau_K, \pi_K \in \mathbb{S}_n}V_{\lambda}(\vz^{K+1}_0) \leq (1 - \mu_1 n \alpha/2)^K V_{\lambda}(\vz_0) + 2 n^2 \alpha^2 l^2 \sigma^2(2 + 3\eta^3)/\mu_1.$$
Setting $\alpha = \min \{ 1/5\eta n l, 4 \log(V_{\lambda}(\vz_0)K)/\mu_1 nK \}$ results in the following last iterate convergence guarantee which holds for any $K \geq 1$:
\begin{align*}
    \max_{\tau_1, \pi_1, \ldots, \tau_K, \pi_K \in \mathbb{S}_n} V_{\lambda}(\vz^{K+1}_0) &\leq e^{\nicefrac{-K}{365\kappa^3}}V_{\lambda}(\vz_0) + \frac{\mu_1 + \hat{c}\kappa^8\sigma^2 \log^2(V_{\lambda}(\vz_0)K)}{\mu_1 K^2} \\
    &= \Tilde{O}(e^{\nicefrac{-K}{365\kappa^3}} + \nicefrac{1}{K^2}),
\end{align*}
where $\kappa = \max \{ \nicefrac{l}{\mu_1}, \nicefrac{l}{\mu_2} \}$ and $\hat{c} > 0$ is a numerical constant that is independent of $l, \mu_1, \mu_2$ and $\sigma^2$.
\end{theorem}
\begin{proof}
\textbf{Step 1: Epoch level update rule}
We highlight that the derivation of the epoch level update rule in Step 1 of Theorem \ref{app-thm:agda-rr-convergence} applies to any permutations $\tau_k, \pi_k \in \mathbb{S}_n$. Thus, repeating the same arguments gives us the following epoch level update rule for AGDA-AS:
\begin{equation}
\label{app-eqn:agda-as-x-update}
    \vx^{k+1}_0 = \vx^k_0 - n\alpha \grad_{\vx} F(\vx^k_0, \vy^k_0) + \alpha^2 \vr_k,
\end{equation}
\begin{equation}
\label{app-eqn:agda-as-y-update}
\vy^{k+1}_0 = \vy^k_0 + n \beta \grad_{\vy} F(\vx^{k+1}_0, \vy^k_0) + \beta^2 \vs_k,
\end{equation}
where $\vr_k$ and $\vs_k$ are defined as follows:
\begin{equation}
\label{app-eqn:agda-as-rk}
    \vr_k = \sum_{i=1}^{n-1}\left[\revprod_{t=i+2}^{n} (\vI - \alpha \vH_{\tau_k(t)}) \right] \vH_{\tau_k(i+1)}\sum_{j=1}^{i} \grad_{\vx} f_{\tau_k(j)}(\vx^k_0, \vy^k_0),
\end{equation}
\begin{equation}
\label{app-eqn:agda-as-sk}
    \vs_k = \sum_{i=1}^{n-1}\left[\revprod_{t=i+2}^{n} (\vI + \beta \vJ_{\pi_k(t)}) \right] \vJ_{\pi_k(i+1)}\sum_{j=1}^{i} \grad_{\vy} f_{\pi_k(j)}(\vx^{k+1}_0, \vy^k_0).
\end{equation}
As before, $\vH_{\tau_k(i)}$ and $\vJ_{\pi_k(i)}$ are defined as,
\begin{align*}
    \vH_{\tau_k(i)} &= \int_{0}^{1} \grad^2_{\vx \vx} f_{\tau_k(i)}(\vx^k_0 + t(\vx^{k}_{i-1} - \vx^k_0), \vy^k_0) \textrm{d}t, \\
    \vJ_{\pi_k(i)} &= \int_{0}^{1} \grad^2_{\vy \vy} f_{\pi_k(i)}(\vx^{k+1}_0,\vy^k_0 + t(\vy^{k}_{i-1} - \vy^k_0)) \textrm{d}t.
\end{align*}
\textbf{Step 2: Uniform bound on the noise terms}
Once again, we repeat the same computations as in Step 2 of Theorem \ref{app-thm:agda-rr-convergence} to obtain:
\begin{equation*}
    \norm{\vr_k} \leq le^{1/5} \left[ n^2 \norm{\grad_{\vx}F(\vx^k_0, \vy^k_0)}/2 + \sum_{i=1}^{n-1} i |\grad_{\vx}F(\vx^k_0, \vy^k_0) -  1/i\sum_{j=1}^{i} \grad_{\vx} f_{\tau_k(j)}(\vx^k_0, \vy^k_0)|  \right].
\end{equation*}
Defining $d_{k,i} = |\grad_{\vx}F(\vx^k_0, \vy^k_0) -  1/i\sum_{j=1}^{i} \grad_{\vx} f_{\tau_k(j)}(\vx^k_0, \vy^k_0)|$ and using Young's inequality gives us the following:
\begin{equation}
\label{app-eqn:agda-as-rk-mid}
\norm{\vr_k} \leq le^{1/5} \left[ n^2 \norm{\grad_{\vx}F(\vx^k_0, \vy^k_0)}/2 + \sum_{i=1}^{n-1} i d_{k,i}  \right],
\end{equation}
\begin{equation}
\label{app-eqn:agda-as-rk-sq-mid}
\norm{\vr_k}^2 \leq 2l^2e^{2/5} \left[ n^4 \norm{\grad_{\vx}F(\vx^k_0, \vy^k_0)}^2/4 + (n-1)\sum_{i=1}^{n-1} i^2 d^2_{k,i}  \right].
\end{equation}
We proceed by using the bounded gradient variance assumption as follows:
\begin{align*}
    (n-1)\sum_{i=1}^{n-1} i^2 d^2_{k,i} &= (n-1)\sum_{i=1}^{n-1} |\sum_{j=1}^{i}\grad_{\vx}F(\vx^k_0, \vy^k_0) - \grad_{\vx} f_{\tau_k(j)}(\vx^k_0, \vy^k_0)|^2 \\
    &\leq (n-1) \sum_{i=1}^{n-1} i \sum_{j=1}^{i}|\grad_{\vx}F(\vx^k_0, \vy^k_0) - \grad_{\vx} f_{\tau_k(j)}(\vx^k_0, \vy^k_0)|^2 \\
    &\leq (n-1) \sum_{i=1}^{n-1} i \sum_{j=1}^{n}|\grad_{\vx}F(\vx^k_0, \vy^k_0) - \grad_{\vx} f_{\tau_k(j)}(\vx^k_0, \vy^k_0)|^2 \\
    &\leq (n-1) \sum_{i=1}^{n-1} i \sum_{j=1}^{n}|\grad_{\vx}F(\vx^k_0, \vy^k_0) - \grad_{\vx} f_{\tau_k(j)}(\vx^k_0, \vy^k_0)|^2 \\
    &\leq (n-1)(n \sigma^2)\sum_{i=1}^{n-1} i \leq n^4 \sigma^2 / 2.
\end{align*}
Substituting the above into \eqref{app-eqn:agda-as-rk-sq-mid} gives us the following:
\begin{equation}
\label{app-eqn:agda-as-rk-sq-bound}
\norm{\vr_k}^2 \leq \frac{3n^4 l^2}{4}\norm{\grad_{\vx} F(\vx^k_0, \vy^k_0)}^2 + \frac{3\sigma^2 l^2 n^4}{2}.
\end{equation}
Furthermore, by Young's inequality, we obtain,
\begin{align*}
    \sum_{i=1}^{n-1}i d_{k,i} \leq [(n-1) \sum_{i=1}^{n-1} i^2 d^2_{k,i}]^{1/2} \leq n^2 \sigma / \sqrt{2}.
\end{align*}
By substituting the above into \eqref{app-eqn:agda-as-rk-mid}, we obtain the following:
\begin{equation}
\label{app-eqn:agda-as-rk-bound}
    \norm{\vr_k} \leq \frac{2n^2l}{3}\norm{\grad_{\vx} F(\vx^k_0, \vy^k_0)} +  \sigma ln^{2}.
\end{equation}

Performing similar computations for $\vs_k$ gives us the following bounds:
\begin{equation}
\label{app-eqn:agda-as-sk-bound}
    \norm{\vs_k} \leq \frac{2n^2l}{3}\norm{\grad_{\vy} F(\vx^{k+1}_0, \vy^k_0)} +  \sigma ln^{2},
\end{equation}
\begin{equation}
\label{app-eqn:agda-as-sk-sq-bound}
\norm{\vs_k}^2 \leq \frac{3n^4 l^2}{4}\norm{\grad_{\vy} F(\vx^{k+1}_0, \vy^k_0)}^2 + \frac{3\sigma^2 l^2 n^4}{2}.
\end{equation}
We highlight that the bounds \eqref{app-eqn:agda-rr-rk-bound}, \eqref{app-eqn:agda-rr-rk-sq-bound} \eqref{app-eqn:agda-rr-sk-bound} and \eqref{app-eqn:agda-rr-sk-sq-bound} hold uniformly for any permutations $\tau_k, \pi_k \in \mathbb{S}_n$.

\textbf{Step 3: Obtaining a convergence guarantee}  As before, we define $a_k = \Phi(\vx^k_0) - \Phi^*$ and $b_k = \Phi(\vx^k_0) - F(\vx^k_0, \vy^k_0)$ which implies that $V_{\lambda}(\vz^k_0) = a_k + \lambda b_k$.

We follow the same procedure as Step 3 of Theorem \ref{app-thm:agda-rr-convergence} and substitute the uniform bounds \eqref{app-eqn:agda-rr-rk-bound}, \eqref{app-eqn:agda-rr-rk-sq-bound} \eqref{app-eqn:agda-rr-sk-bound} and \eqref{app-eqn:agda-rr-sk-sq-bound} wherever appropriate to obtain the following inequalities:
\begin{align}
\label{app-eqn:agda-as-phi-update}
\Phi(\vx^{k+1}_0) - \Phi(\vx^{k}_0) &\leq n \alpha \norm{\grad_\vx F(\vx^{k}_0, \vy^{k}_0) - \nabla \Phi(\vx^{k}_0)}^2 - \frac{n \alpha}{2} \norm{\nabla \Phi(\vx^{k}_0)}^2 \nonumber \\  &+ \frac{3 n^3 \alpha^3 l^2}{4} \norm{\grad_\vx F(\vx_{0}^{k}, \vy_{0}^{k})}^2
    + \frac{3n^3 \alpha^3 l^2 \sigma^2}{2},
\end{align}
\begin{equation}
\label{app-eqn:agda-as-fx-bound}
    F(\vx^{k}_0, \vy^{k}_0) - F(\vx^{k+1}_0, \vy^{k}_0) \leq (\frac{3 n \alpha}{2} + \frac{13}{10}l n^2 \alpha^2 + \frac{3 n^4 \alpha^4 l^3}{8})\norm{\grad_\vx F(\vx_{0}^{k}, \vy_{0}^{k})}^2 + \frac{9\alpha^3 n^3 l^2 \sigma^2}{10},
\end{equation}
\begin{equation}
\label{app-eqn:agda-as-fy-bound}
\Phi(\vx^{k+1}_0) -F(\vx^{k+1}_0, \vy^{k+1}_0) \leq (1 - n \beta \mu_2) [\Phi(\vx^{k+1}_0) - F(\vx^{k+1}_0, \vy^{k}_0)] + \frac{23n^2 \beta^3 l^2 \sigma^2}{10}.
\end{equation}
Once again, following the same steps as Step 3 of Theorem \ref{app-thm:agda-rr-convergence}, we conclude that,
\begin{align}
a_{k+1} &\leq a_k + n \alpha \norm{\grad_\vx F(\vx^{k}_0, \vy^{k}_0) - \nabla \Phi(\vx^{k}_0)}^2 - \frac{n \alpha}{2} \norm{\nabla \Phi(\vx^{k}_0)}^2  \nonumber\\
&+ \frac{3 n \alpha}{100} \norm{\grad_\vx F(\vx_{0}^{k}, \vy_{0}^{k})}^2 + \frac{3n^3 \alpha^3 l^2 \sigma^2}{2}.
\label{app-eqn:agda-as-a-update}
\end{align}
\begin{multline}
    b_{k+1} \leq (1 - n \beta \mu_2)b_k + (1 - n \beta \mu_2) n \alpha \norm{\grad_\vx F(\vx^{k}_0, \vy^{k}_0) - \grad \Phi(\vx^{k}_0)}^2 - (1 - n \beta \mu_2)\frac{n \alpha}{2} \norm{\grad \Phi(\vx^{k}_0)}^2 \\
    + (1 - n \beta \mu_2)\frac{9 n \alpha}{5}\norm{\grad_{\vx}F(\vx^{k}_0, \vy^{k}_0)}^2 + (1 - n \beta \mu_2)\frac{12n^3 \alpha^3 l^2 \sigma^2}{5}  + \frac{23n^3 \beta^3 l^2 \sigma^2}{10}.
\label{app-eqn:agda-as-b-update}
\end{multline}
By setting $\lambda = \frac{1}{10}$ and $\eta = \frac{73l^2}{2\mu_2^2}$ and following the same steps as in Step 3 of Theorem \ref{app-thm:agda-rr-convergence}, we obtain the following update equation for $V_{\lambda}(\vz^{k+1}_0)$:
\begin{align*}
    V_{\lambda}(\vz^{k+1}_0) \leq (1 - \mu_1 n\alpha/2)V_{\lambda}(\vz^{k+1}_0) +  n^3 \alpha^3 l^2 \sigma^2 (2 + 3\eta^3).
\end{align*}
By unrolling the above inequality for $K$ steps we obtain the following:
\begin{equation*}
V_{\lambda}(\vz^{K+1}_0) \leq (1 - \mu_1 n \alpha/2)^K V_{\lambda}(\vz_0) + 2 n^2 \alpha^2 l^2 \sigma^2(2 + 3\eta^3)/\mu_1.
\end{equation*}
Since the above holds for any sequence of permutations $\tau_1, \pi_1, \ldots \tau_K, \pi_K \in \mathbb{S}_n$, we conclude that:
\begin{equation}
\label{app-eqn:agda-as-v-unroll}
\max_{\tau_1, \pi_1, \ldots, \tau_K, \pi_K \in \mathbb{S}_n}V_{\lambda}(\vz^{K+1}_0) \leq (1 - \mu_1 n \alpha/2)^K V_{\lambda}(\vz_0) + 2 n^2 \alpha^2 l^2 \sigma^2(2 + 3\eta^3)/\mu_1.
\end{equation}
Setting $\alpha = \min \{ 1/5\eta n l, 4 \log(V_{\lambda}(\vz_0)K)/\mu_1 nK \}$ and using $\eta \leq 37 \kappa^2$, we observe that:
\begin{align*}
 n^2 \alpha^2 l^2 \sigma^2(2 + 3\eta^3)/\mu_1 &\leq  \frac{16l^2 \sigma^2(2 + 3\eta^3)}{\mu_1^3} \frac{\log^2(V_{\lambda}(\vz_0)n^{1/2}K)}{K^2} \\
 &\leq \frac{\hat{c}\kappa^8 \sigma^2 \log^2(V_{\lambda}(\vz_0)K)}{\mu_1 K^2},
\end{align*}
where $\hat{c} > 0$ is a numerical constant that is independent of $\kappa, \mu$ and $\sigma^2$. By substituting the above into \eqref{app-eqn:agda-as-v-unroll}, we get:
\begin{equation}
\label{app-eqn:agda-as-v-unroll-mid}
    \max_{\tau_1, \pi_1, \ldots, \tau_K, \pi_K \in \mathbb{S}_n} V_{\lambda}(\vz^{K+1}_0) \leq (1 - \mu_1 n \alpha/2)^K V_{\lambda}(\vz_0) + \frac{\hat{c}\kappa^8 \sigma^2 \log^2(V_{\lambda}(\vz_0)K)}{\mu_1 K^2}.
\end{equation}
\textbf{Case 1: $1/5\eta n l \leq 4 \log(V_{\lambda}(\vz_0)n^{1/2}K)/\mu_1 nK $}

In this case, $\alpha = 1/5 \eta n l$ and hence, we get:
\begin{align*}
    (1 - n\alpha \mu_1 /2)^K V_{\lambda}(\vz_0) \leq e^{-n\alpha \mu_1 K/2} V_{\lambda}(\vz_0) \leq e^{-\frac{K}{365\kappa^3}} V_{\lambda}(\vz_0).
\end{align*}
By substituting into \eqref{app-eqn:agda-as-v-unroll-mid}, we get:
\begin{equation}
\label{app-eqn:agda-as-case1-proof}
\max_{\tau_1, \pi_1, \ldots, \tau_K, \pi_K \in \mathbb{S}_n} V_{\lambda}(\vz^{K+1}_0) \leq e^{-\frac{K}{365\kappa^3}} V_{\lambda}(\vz_0) + \frac{\hat{c}\kappa^8 \sigma^2 \log^2(V_{\lambda}(\vz_0)K)}{\mu_1 K^2}.
\end{equation}
\textbf{Case 2: $1/5\eta n l \leq 4\log(V_{\lambda}(\vz_0)n^{1/2}K)/(\mu_1 nK)$} 

In this case, $\alpha = 4 \log(V_{\lambda}(\vz_0)n^{1/2}K)/\mu_1 nK$ and we conclude that,
\begin{align*}
    (1 - n\alpha \mu_1 / 2)^KV_{\lambda}(\vz_0) &\leq e^{-n\alpha \mu_1 K/2}V_{\lambda}(\vz_0) \leq e^{-2\log(V_{\lambda}(\vz_0)^{1/2}n^{1/2}K)}V_{\lambda}(\vz_0) \leq \frac{1}{nK^2}.
\end{align*}
By substituting into \eqref{app-eqn:agda-as-v-unroll-mid}, we get:
\begin{equation}
\label{app-eqn:agda-as-case2-proof}
\max_{\tau_1, \pi_1, \ldots, \tau_K, \pi_K \in \mathbb{S}_n} V_{\lambda}(\vz^{K+1}_0) \leq \frac{\mu_1 + \hat{c}\kappa^8 \sigma^2 \log^2(V_{\lambda}(\vz_0)K)}{\mu_1 K^2}.
\end{equation}
Taking the maximum of the right hand side of \eqref{app-eqn:agda-as-case1-proof} and \eqref{app-eqn:agda-as-case2-proof}, we obtain the following rate, which holds for any $K \geq 1$:
\begin{align*}
\max_{\tau_1, \pi_1, \ldots, \tau_K, \pi_K \in \mathbb{S}_n} V_{\lambda}(\vz^{K+1}_0) \leq e^{-\frac{K}{365\kappa^3}} V_{\lambda}(\vz_0) + \frac{\mu_1 + \hat{c}\kappa^8 \sigma^2 \log^2(V_{\lambda}(\vz_0)K)}{\mu_1 K^2}.
\end{align*}
Suppressing constant terms and logarithmic factors, we get,
\begin{align*}
    \bE[V_{\lambda}(\vz^{K+1}_0)] = \Tilde{O}(e^{-\frac{K}{365\kappa^3}} + \nicefrac{1}{K^2}).
\end{align*}
\end{proof}
\subsection{Convergence of AGDA-RR and AGDA-AS to a Saddle Point}
\label{app-sec:agda-conv-guarantee}
In this section, we demonstrate that our convergence guarantees for AGDA-RR and AGDA-AS can be expressed in terms of squared distance to the set of saddle points of $F$, denoted as $\sZ^*$. We do so by upper bounding $\textrm{dist}(\vz, \sZ^*)^2$ in terms of $V_{\lambda}(\vz)$. We begin by establishing some basic properties of $\sZ^*$ as follows.
\begin{lemma}[Set of Saddle Points of 2P\L{} functions]
\label{app-lem:2pl-saddle-closed}
Let $F : \bR^{d_\vx} \times \bR^{d_\vy} \rightarrow \bR$ be an $l$-smooth function satisfying Assumption \ref{as:2pl} (with constants $\mu_1$ and $\mu_2$) and let $\sZ^*$ denote the set of saddle points of $F$. Then $\sZ^*$ is a closed set and $\textrm{dist}(\vz, \sZ^*) = \min_{\vz^* \in \sZ^*} \norm{\vz - \vz^*}$ is a well defined continuous function of $\vz$. Furthermore, for any $\vx \in \bR^{d_\vx}$, the set $\sY^{*}(\vx) = \arg \max_{\hat{\vy} \in \bR^{d_\vy}} F(\vx, \hat{\vy})$ is a closed set.
\end{lemma}
\begin{proof}
By the $l$-smoothness of $F$, $\grad_{\vx} F(\vx, \vy)$ and $\grad_{\vy} F(\vx, \vy)$ are Lipschitz continuous functions, which implies that $\eta(\vz) = [\grad_{\vx} F(\vz), \grad_{\vy} F(\vz)]$ is a continuous map. Furthermore, by Lemma \ref{app-lem:2pl-eq-optimal}, the set of saddle points of $F$ are exactly the set of stationary points of $F$, which means,
\begin{align*}
    \sZ^* = \{ \vz \in \bR^d \ | \  \grad_{\vx} F(\vz) = 0, \grad_{\vy} F(\vz) = 0\} = \eta^{-1}(\{ 0 \}).
\end{align*}
Hence, $\sZ^*$ is a closed set as it is the inverse image of the closed set $\{ 0 \}$ under a continuous map. Consequently, $\textrm{dist}(\vz, \sZ^*) = \min_{\vz^* \in \sZ^*} \norm{\vz - \vz^*}$ is a well defined continuous function of $\vz$.

Consider any $\vx \in \bR^{d_{\vx}}$. Since $F$ is a 2P\L{} function, $-F(\vx, .)$ satisfies the (one-sided) P\L{} inequality. Thus, $\vy \in \sY^*(\vx)$ if and only if $\grad_{\vy} F(\vx, \vy) = 0$, i.e., $\sY^*(\vx) = \{ \vy \in \bR^{d_\vy} \ | \ \grad_{\vy} F(\vx, \vy) = 0 \}$. By the continuity of $\grad_{\vy} F(\vx, .)$, we conclude that $\sY^*(\vx)$ is a closed set.
\end{proof}
The following lemma, which is key to establishing the relation between $V_{\lambda}(\vz)$ and $\textrm{dist}(\vz, \sZ^*)$, is an adaptation of Lemma A.3 of \citet{NouhiedGDMax2019} for 2P\L{} functions. 
\begin{lemma}[Lipschitzness of $\sY^{*}(\vx)$]
\label{app-lem:2pl-argmax-lip}
Let $F : \bR^{d_\vx} \times \bR^{d_\vy} \rightarrow \bR$ be an $l$-smooth function satisfying Assumption \ref{as:2pl} (with constants $\mu_1$ and $\mu_2$). For any $\vx \in \bR^{d_{\vx}}$, let $\sY^{*}(\vx) = \arg \max_{\vy \in \bR^{d_\vy}} F(\vx, \vy)$. Then, for any $\vx_1, \vx_2 \in \bR^{d_{\vx}}$ and $\vy^{*}(\vx_1) \in \sY^{*}(\vx_1)$, there exists a $\vy^{*}(\vx_2) \in \sY^{*}(\vx_2)$ such that,
\begin{align*}
    \norm{\vy^{*}(\vx_1) - \vy^{*}(\vx_2)} \leq \frac{l}{2 \mu_2}\norm{\vx_1 - \vx_2}. 
\end{align*}
\end{lemma}
\begin{proof}
Consider any  $\vx_1, \vx_2 \in \bR^{d_{\vx}}$ and $\vy^{*}(\vx_1) \in \sY^{*}(\vx_1)$. Since $F$ is 2P\L{} with constants $\mu_1$ and $\mu_2$, $-F(\vx, .)$ satisfies the P\L{} inequality for any $\vx \in \bR^{d_{\vx}}$ with constant $\mu_2$. Furthermore, by Lemma \ref{app-lem:2pl-saddle-closed}, $\sY^*(\vx)$ is a closed set for any $\vx \in \bR^{d_{\vx}}$. Thus, by Lemma A.3 of \citet{NouhiedGDMax2019}, there exists a $\vy^{*}(\vx_2) \in \sY^{*}(\vx_2)$ satisfying $\norm{\vy^{*}(\vx_1) - \vy^{*}(\vx_2)} \leq \frac{l}{2 \mu_2}\norm{\vx_1 - \vx_2}$.
\end{proof}
Equipped with the above lemmas, we can finally upper bound $\textrm{dist}(\vz, \sZ^*)$ in terms of $V_{\lambda}(\vz)$.
\begin{lemma}
\label{app-lem:2pl-lyapunov-dist}
Let $F : \bR^{d_\vx} \times \bR^{d_\vy} \rightarrow \bR$ be an $l$-smooth function satisfying Assumption \ref{as:2pl} (with constants $\mu_1$ and $\mu_2$). Define $\Phi(\vx) = \max_{\vy \in \bR^{d_\vy}} F(\vx, \vy)$, $\Phi^* = \min_{\vx \in \bR^{d_{\vx}}} \Phi(\vx)$, $\sX^* = \arg \min_{\vx \in \bR^{d_{\vx}}} \Phi(\vx)$ and $\sY^{*}(\vx) = \arg \max_{\vy \in \bR^{d_\vy}} F(\vx, \vy)$. Moreover, for any $\lambda > 0$, define $V_{\lambda}(\vx, \vy) = [(\Phi(\vx) - \Phi^*) + \lambda(\Phi(\vx) - F(\vx, \vy)]$. Then the following holds for any $\vz \in \bR^d$:
\begin{align*}
    \textrm{dist}(\vz, \sZ^*)^2 \leq \max \left \{ \frac{2}{\mu_1}(\frac{l^2}{2\mu^2_2} + 1), \frac{4}{\lambda \mu_2} \right \} V_{\lambda}(\vz).
\end{align*}
\end{lemma}
\begin{proof}
Since the set of saddle points of $F$ is non-empty and $\Phi$ is a P\L{} function (with constant $\mu_1$) as per Lemma \ref{app-lem:agda-phi-prop}, we can use the same arguments as Lemma \ref{app-lem:2pl-saddle-closed} to show that $\sX^*$ is a closed set.

Consider any $\vz = (\vx, \vy)$. Let $\vx^{*}$ be the projection of $\vx$ on the closed set $\sX^{*}$ and let $\vy^*(\vx)$ be the projection of $\vy$ on the closed set (as per Lemma \ref{app-lem:2pl-saddle-closed}) $\sY^*(\vx)$. 

Applying Lemma \ref{as:pl-quad-growth} to $\Phi(\vx)$, we obtain,
\begin{equation}
\label{app-eqn:2pl-phi-qg}
\Phi(\vx) - \Phi^* \geq \frac{\mu_1}{2}\norm{\vx - \vx^*}^2.
\end{equation}
Similarly, applying Lemma \ref{as:pl-quad-growth} to $-F(\vx, \vy)$, we obtain,
\begin{equation}
\label{app-eqn:2pl-fy-qg}
\Phi(\vx) - F(\vx, \vy) \geq \frac{\mu_2}{2}\norm{\vy - \vy^*(\vx)}^2.
\end{equation}
By Lemma \ref{app-lem:2pl-argmax-lip}, there exists a $\vy^* \in \sY^{*}(\vx^*)$ satisfying,
\begin{equation}
    \norm{\vy^*(\vx) - \vy^{*}} \leq \frac{l}{2\mu_2}\norm{\vx - \vx^*}.
\end{equation}
Let $\vz^* = (\vx^*, \vy^*)$. We note that since $\vx^* \in \sX^*$ and $\vy^* \in \sY^*(\vx^*)$, $\vz^*$ is a global minimax point of $F$. Thus by Lemma \ref{app-lem:2pl-eq-optimal}, $\vz^*$ is a saddle point of $F$, i.e., $\vz^* \in \sZ^*$. We now proceed as follows:
\begin{align*}
    \norm{\vy - \vy^*} &\leq \norm{\vy - \vy^{*}(\vx)} + \norm{\vy^{*}(\vx) - \vy^*} \\
    &\leq \norm{\vy - \vy^{*}(\vx)} + \frac{l}{2\mu_2}\norm{\vx - \vx^*}.
\end{align*}
Applying Young's Inequality, we obtain,
\begin{equation}
    \norm{\vy - \vy^*}^2 \leq 2\norm{\vy - \vy^{*}(\vx)}^2 + \frac{l^2}{2\mu^2_2}\norm{\vx - \vx^*}^2,
\end{equation}
which implies that.
\begin{align*}
    \norm{\vz - \vz^*}^2 &= \norm{\vx - \vx^*}^2 + \norm{\vy - \vy^*}^2 \\
    &\leq 2\norm{\vy - \vy^{*}(\vx)}^2 + (\frac{l^2}{2\mu^2_2} + 1)\norm{\vx - \vx^*}^2.
\end{align*}
Substituting \eqref{app-eqn:2pl-phi-qg} and \eqref{app-eqn:2pl-fy-qg} in the above inequality, we obtain,
\begin{align*}
    \norm{\vz - \vz^*}^2 &\leq \frac{4}{\mu_2}[\Phi(\vx) - F(\vx, \vy)] + \frac{2}{\mu_1}(\frac{l^2}{2\mu^2_2} + 1)[\Phi(\vx) - \Phi^*] \\
    &\leq \max \left \{ \frac{2}{\mu_1}(\frac{l^2}{2\mu^2_2} + 1), \frac{4}{\lambda \mu_2} \right \} [(\Phi(\vx) - \Phi^*) + \lambda(\Phi(\vx) - F(\vx, \vy))] \\
    &\leq \max \left \{ \frac{2}{\mu_1}(\frac{l^2}{2\mu^2_2} + 1), \frac{4}{\lambda \mu_2} \right \} V_{\lambda}(\vz).
\end{align*}
Since $\vz^* \in \sZ^*$, we conclude that,
\begin{equation}
    \textrm{dist}(\vz, \sZ^*)^2 \leq \max \left \{ \frac{2}{\mu_1}(\frac{l^2}{2\mu^2_2} + 1), \frac{4}{\lambda \mu_2} \right \} V_{\lambda}(\vz).
\end{equation}
\end{proof}
Using Lemma \ref{app-lem:2pl-lyapunov-dist} and \ref{app-thm:agda-rr-convergence}, we can conclude that AGDA-RR converges in terms of $\textrm{dist}(\vz, \sZ^*)^2$ as $\bE[\textrm{dist}(\vz^{K+1}_0, \sZ^*)^2] = \Tilde{O}(e^{-\nicefrac{K}{365\kappa^3}} + \nicefrac{1}{nK^2})$. Since $\sZ^*$ is a closed set, this implies that the final epoch iterate $\vz^{K+1}_0$ of AGDA-RR converges (in expectation) to some saddle point of $F$, with the rate $\Tilde{O}(e^{-\nicefrac{K}{365\kappa^3}} + \nicefrac{1}{nK^2})$. Similarly, using Using Lemma \ref{app-lem:2pl-lyapunov-dist} and \ref{app-thm:agda-as-convergence}, we conclude that the final epoch iterate $\vz^{K+1}_0$ of AGDA-AS converges  to some saddle point of $F$ with the rate $\textrm{dist}(\vz^{K+1}_0, \sZ^*)^2 = \Tilde{O}(e^{-\nicefrac{K}{365\kappa^3}} + \nicefrac{1}{K^2})$.

\section{Literature Review}
While the convergence properties of SGD (and SGDA) with uniform sampling have been well studied for a long time, with matching upper and lower bounds of $\Theta(\nicefrac{1}{nK})$ for both strongly convex minimization and strongly convex-strongly concave minimax optimization \cite{RakhlinSridharan, FacchineiPang} ($n$ is the number of components in the finite-sum objective and $K$ is the number of epochs), advances in the theoretical understanding of without-replacement sampling for SGD are a fairly recent phenomenon. This is generally attributed to the unavailability of provably unbiased gradient estimates, leading to a significantly more complicated analysis that often requires sophisticated mathematical tools. 

Initial progress was made by the pioneering work of \citet{RechtReAMGM} who proposed the \emph{noncommutative AM-GM conjecture}, and proved that this conjecture implies the faster convergence of without-replacement SGD. Unfortunately, the conjecture was recently disproved by \citet{LaiLimAMGM}. Subsequent progress was made in \citet{GurbuzRR2019}, where the authors establish an asymptotic $O(\nicefrac{1}{K^2})$ convergence rate for GD with Random Reshuffling (or RR) on functions of the form $F(x) = 1/n \sum_{i=1}^{n} f_i(x)$ where $F$ is a strongly convex quadratic and the components $f_i$ satisfy a Hessian Lipschitz assumption. For the same function class, \citet{GurbuzIG2019} established asymptotic $O(\nicefrac{1}{K^2})$ rates for Gradient Descent with Shuffle Once (SO) and Incremental Gradient (IG). The analysis in both these works relied on the following intuition, also proposed earlier in \citet{NedicBertsekas}: \emph{since sampling without replacement ensures that each component function is used exactly once within an epoch, the overall progress made by RR (and IG) over a complete epoch closely tracks that of full batch GD}. To this end, the key technique in both these works was to interpret the time evolution of the epoch iterates $\vx^k_0$ as full batch GD with added noise, and to subsequently obtain a convergence rate by determining the (asymptotic) rate of decay for the noise using Chung's Lemma.  This approach still remains a highly useful strategy for understanding SGD without replacement, and has served as a key insight for several future works such as \citet{HaochenSraRR19, AhnRR2020}. In fact, formally extending this insight to minimax optimization (and more generally, to variational inequalities) in a manner that can simultaneously handle RR, SO and adversarial shuffling, is a cornerstone of our analysis.

The first nonasymptotic guarantees for RR, which also characterize the dependence of the convergence rate on $n$, were given by \citet{HaochenSraRR19}, who establish a convergence rate of $\Tilde{O}(\nicefrac{1}{n^2 K^2} + \nicefrac{1}{K^3})$ assuming $K \geq C \kappa \log(nK)$ for some $C > 0$, for the case when $F$ is a strongly convex quadratic and the components $f_i$ are convex quadratics. Under the assumption that $F$ satisfies the Polyak-\L{}ojasiewicz inequality and the components $f_i$ are smooth and Hessian Lipschitz, \citet{HaochenSraRR19} also establish an $\Tilde{O}(\nicefrac{1}{n^2 K^2} + \nicefrac{1}{K^3})$ rate assuming $K \geq C \kappa^2 \log(nK)$. These rates imply the provable superiority of RR when $K \geq \Omega(\sqrt{n})$. Subsequent advances were made by \citet{NagarajRR2020}, who used sophisticated Wasserstein coupling arguments to establish a rate of $O(\nicefrac{1}{nK^2})$ for RR which holds under the epoch requirement of $K \geq C \kappa^2 \log(nK)$, assuming $F$ is strongly convex and $f_i$ are smooth and Lipschitz. This rate demonstrates the provable improvement of RR over uniform sampling as soon as the epoch requirement is satisfied. 

Development of lower bounds for SGD without replacement began with \citet{SafranShamirRR2020} where the authors establish a lower bound of $\Omega(\nicefrac{1}{n^2K^2} + \nicefrac{1}{n K^3})$ for RR on strongly convex quadratic $F$ with $f_i$ being convex quadratics, assuming constant step-sizes. Under the same assumptions, \citet{SafranShamirRR2020} also establish an $\Omega(\nicefrac{1}{nK^2})$ lower bound for SO and an $\Omega(\nicefrac{1}{K^2})$ lower bound for IG. Further progress on lower bounds and their tightness was made in \citet{RajputRR2020} where the authors establish a $\Tilde{O}(\nicefrac{1}{n^2K^2} + \nicefrac{1}{n K^3})$ convergence rate for RR on strongly convex quadratic $F$ with convex quadratic $f_i$. \citet{RajputRR2020} also develop a lower bound of $\Omega(\nicefrac{1}{nK^2})$ for strongly convex $F$ with smooth $f_i$, thereby establishing the tightness of the upper bounds in \citet{NagarajRR2020} modulo logarithmic factors.

Recently, the theory of SGD without replacement has been significantly advanced by the parallel works of \citet{AhnRR2020} and \citet{MischenkoRR2020}. For the case of Polyak-\L{}ojasiewicz $F$ with smooth (but not necessarily convex) $f_i$, \citet{AhnRR2020} show that RR converges at an optimal rate (modulo logarithmic factors) of $\Tilde{O}(\nicefrac{1}{nK^2})$ when $K \geq C \kappa \log(n^{1/2}K)$. For the same problem class, \citet{MischenkoRR2020} obtain an anytime-valid near-optimal rate of $\Tilde{O}(e^{\nicefrac{-K}{4\kappa}} + \nicefrac{1}{nK^2})$ under a bounded gradient variance assumption. This rate matches the optimal rate modulo logarithmic factors when $K \geq C \kappa \log(n^{1/2}K)$. \citet{AhnRR2020} also analyze SO for this problem class and derive $\Tilde{O}(\nicefrac{1}{nK^2})$ rates assuming $K \geq C \kappa^2 \log(n^{1/2}K)$. The obtained rate matches the lower bound of \citet{SafranShamirRR2020} modulo logarithmic factors. Concurrently, \citet{MischenkoRR2020} present an analysis that covers both RR and SO for strongly convex $F$ with smooth convex $f_i$ and obtain a rate of $\Tilde{O}(e^{\nicefrac{-K}{4\kappa}} + \nicefrac{1}{nK^2})$. For IG, \citet{NgyuenRR2020} obtain a rate of $\Tilde{O}(\nicefrac{1}{K^2})$ assuming strongly convex $F$, smooth $f_i$ and an epoch requirement of $K \geq 12 \kappa^2 \log(K)$. Moving beyond the P\L{} class, the recent work of \citet{KLRR2021} analyze RR for the case when $f_i$ are smooth and $F$ satisfies a K\L{} inequality and, using variable step-sizes, establish an asymptotic convergence rate of $O(\nicefrac{1}{K^2})$ when the K\L{} exponent lies in $[0, \nicefrac{1}{2}]$.

Recent works have also analyzed RR and SO in the federated minimization setting. Notably, \citet{SuvritFED} developed an extension of the techniques used in \citet{AhnRR2020} to analyze RR variants of minibatch and local SGD, while \citet{MischenkoFED} extended the approach of \citet{MischenkoRR2020} to develop federated analogs of RR and SO. Among other developments, \citet{NgyuenMomentum} propose a variant of SGD with momentum that is amenable to theoretical analysis under without-replacement sampling. These developments are also complemented by the recent negative result of \citet{SafranShamir2021} which analyzes the condition number dependence of RR, and through an intricate lower bound analysis on convex quadratics, show that RR and SO do not significantly outperform uniform sampling unless $K$ is larger than the condition number.

Despite the recent wave of developments on SGD without replacement, theoretical analysis of sampling without replacement for minimax optimization has remained relatively unexplored. To the best of our knowledge, \citet{Eric_DRO} and \citet{Eric_DRO_OGDA} are the only prior works in this domain.  Both these works focus only on Random Reshuffling and are restricted to finite-sum minimax problems of the form $F(x, y) = 1/n \sum_{i=1}^{n} f_i(x, y)$ where each $f_i$ is smooth, Lipschitz, convex-concave and has bounded domain. Under this setting, \citet{Eric_DRO} proposes a stochastic proximal point method with RR that exhibits a rate of $\Tilde{O}(\nicefrac{1}{\sqrt{nK}})$ whereas \citet{Eric_DRO_OGDA} proposes a zeroth-order optimistic gradient method with RR that exhibits a rate of $\Tilde{O}(\nicefrac{nd^2}{K^{1/4}})$, where $d$ is the dimension of the domain. The analysis in both these works are based on an application of the Wasserstein coupling technique of \citet{NagarajRR2020} (for the convex case) to smooth, Lipschitz convex-concave minimax optimization. 
\section{Experimental Details}
\label{app-sec:experiments}
As described in Section \ref{sec:exps} of the main text, we perform our benchmarks on finite-sum strongly convex-strongly concave quadratic minimax games. The objective $F$ and the components $f_i$ are given by:
\begin{align*}
F(\vx, \vy) &= 1/n \sum_{i=1}^{n} f_i(\vx, \vy) =  \frac{1}{2} \vx^T \vA \vx + \vx^T \vB \vy - \frac{1}{2} \vy^T \vC \vy, \\
f_i(\vx, \vy) &= \frac{1}{2} \vx^T \vA_i \vx + \vx^T \vB_i \vy - \frac{1}{2} \vy^T \vC_i \vy - \vu_i^T \vx - \vv_i^T \vy,
\end{align*}
where $\vA$ and $\vC$ are positive definite and $\vu_i$ and $\vv_i$ are chosen such that $\sum_{i=1}^{n} \vu_i = \sum_{i=1}^{n} \vv_i = 0$. In all our experiments, we set $n=100$ and $\textrm{dim } \vx = \textrm{dim } \vy = 25$. We randomly generate the components $f_i$ such that 20 randomly selected components are strongly convex-strongly concave. To this end, the components $\vA_i$ are selected as follows. We begin by sampling a random orthogonal matrix $\vO_A$, and then generate two $d$-dimensional vectors $\vm_A$ and $\delta_A$, such that each component of $\vm_A$ is uniformly sampled from $[\mu_A, L_A]$ whereas that of $\delta_A$ is uniformly sampled from $[\mu_\delta, L_\delta]$. Subsequently, for each $i \in [n]$, we set $\vA_i = \vO_A \Lambda_i \vO_A^T$ where $\Lambda_i = \textrm{diag}(-\delta_A)$ or 20 randomly selected indexes and $\Lambda_i = \textrm{diag}( 5\vm_A/4 + \delta_A/4)$ for the other 80 components. A similar procedure is followed for generating the matrices $\vB$ and $\vC$. For the vectors $\vu_i$, we generate a vector $\delta_u$ with components that are uniformly sampled from $[\mu_\delta, L_\delta]$. Then, for 20 randomly sampled indexes, we set $\vu_i = -\delta_u$ and for the rest, we set $\vu_i = \delta_u / 4$. The vectors $\vv_i$ are sampled in a similar fashion. We run each algorithm for $K = 100$ epochs using constant step-sizes of the form $\alpha = \nicefrac{\gamma}{n}$ where $\gamma$ is chosen independently for each algorithm via grid search. In all our experiments, we set $\mu_A = \mu_C = 0.5$, $\mu_B = 5.0$ and $\mu_{\delta} = 50.0$. Furthermore, we set $L_A = L_C = 2 \mu_A$, $L_B = 2 \mu_B$ and $L_\delta = 2 \mu_\delta$. We perform our experiments on a Jupyter notebook with a Python 3.7 kernel, executed on a 2.8 GHz Intel Core i7 processor with 8 GB of memory. The total running time is 40 minutes.

\end{document}